\documentclass[final,12pt]{elsarticle}
\usepackage{float}
\usepackage{mathtools}
\usepackage{amsmath}
\usepackage{xcolor}
\definecolor{darkgreen}{rgb}{0.0, 0.2, 0.13}
\newcommand{\Lloc}[1]{\mathbf{L^{#1}_{loc}}}
\newcommand{\norma}[1]{{\left\|#1\right\|}}

\usepackage{cases}
\usepackage{amsfonts}
\usepackage{amssymb}
\usepackage{amsthm}
\usepackage{graphicx}
\usepackage{mathrsfs}
\usepackage{xcolor}
\usepackage{lineno}
\usepackage{caption}
\PassOptionsToPackage{unicode}{hyperref}
\PassOptionsToPackage{naturalnames}{hyperref}
\usepackage{enumerate}
\usepackage[shortlabels]{enumitem}
\usepackage[ddmmyyyy]{datetime}
\usepackage[margin=2cm]{geometry}
\makeatletter
\g@addto@macro\normalsize{%
  \setlength\abovedisplayskip{4pt}
  \setlength\belowdisplayskip{4pt}
  \setlength\abovedisplayshortskip{4pt}
  \setlength\belowdisplayshortskip{4pt}
}
\numberwithin{equation}{section}
\everymath{\displaystyle}
\usepackage[capitalize,nameinlink]{cleveref}
\crefname{section}{Section}{Sections}
\crefname{subsection}{Subsection}{Subsections}
\crefname{condition}{Condition}{Conditions}
\crefname{hypothesis}{Hypothesis}{Conditions}
\crefname{assumption}{Assumption}{Assumptions}
\crefname{lemma}{Lemma}{Lemmas}
\crefname{definition}{Definition}{Definitions}

\crefformat{equation}{\textup{#2(#1)#3}}
\crefrangeformat{equation}{\textup{#3(#1)#4--#5(#2)#6}}
\crefmultiformat{equation}{\textup{#2(#1)#3}}{ and \textup{#2(#1)#3}}
{, \textup{#2(#1)#3}}{, and \textup{#2(#1)#3}}
\crefrangemultiformat{equation}{\textup{#3(#1)#4--#5(#2)#6}}%
{ and \textup{#3(#1)#4--#5(#2)#6}}{, \textup{#3(#1)#4--#5(#2)#6}}%
{, and \textup{#3(#1)#4--#5(#2)#6}}

\Crefformat{equation}{#2Equation~\textup{(#1)}#3}
\Crefrangeformat{equation}{Equations~\textup{#3(#1)#4--#5(#2)#6}}
\Crefmultiformat{equation}{Equations~\textup{#2(#1)#3}}{ and \textup{#2(#1)#3}}
{, \textup{#2(#1)#3}}{, and \textup{#2(#1)#3}}
\Crefrangemultiformat{equation}{Equations~\textup{#3(#1)#4--#5(#2)#6}}%
{ and \textup{#3(#1)#4--#5(#2)#6}}{, \textup{#3(#1)#4--#5(#2)#6}}%
{, and \textup{#3(#1)#4--#5(#2)#6}}

\crefdefaultlabelformat{#2\textup{#1}#3}
\numberwithin{equation}{section}

\newtheorem{theorem} {Theorem}[section]

\newtheorem{lemma}{Lemma}[section]

\newtheorem{counter example}{Counter Example}[section]

\def\CC{{\rm \kern.24em \vrule width.02em height1.4ex depth-.05ex \kern-.26emC}}

\def\TagOnRight

\def\AA{{it I} \hskip-3pt{\tt A}}

\def\QQ{\rlap {\raise 0.4ex \hbox{$\scriptscriptstyle |$}} {\hskip -0.1em Q}}

\makeatletter
\newcommand{\vo}{\vec{o}\@ifnextchar{^}{\,}{}}
\makeatother

\def\YYint#1#2#3{{\setbox0=\hbox{$#1{#2#3}{\iint}$}
    \vcenter{\hbox{$#2#3$}}\kern-.50\wd0}}

\def\XXint#1#2#3{{\setbox0=\hbox{$#1{#2#3}{\int}$}
    \vcenter{\hbox{$#2#3$}}\kern-.50\wd0}}

\makeatletter
\def\namedlabel#1#2{\begingroup
   \def\@currentlabel{#2}%
   \label{#1}\endgroup
}
\makeatother
\makeatletter
\newcommand{\rmh}[1]{\mathpalette{\raisem@th{#1}}}
\newcommand{\raisem@th}[3]{\hspace*{-1pt}\raisebox{#1}{$#2#3$}}
\makeatother

\newcommand{\descref}[2]{\hyperref[#1]{\textnormal{\textcolor{black}{(}\textcolor{blue}{\bf #2}\textcolor{black}{)}}}}

\newcommand{\dref}[2]{\hyperref[#1]{\textcolor{black}{(}\textcolor{blue}{\bf #2}\textcolor{black}{)}}}
\newcommand{\be} {\begin{eqnarray}}
\newcommand{\ee} {\end{eqnarray}}
\newcommand{\Bea} {\begin{eqnarray*}}
\newcommand{\Eea} {\end{eqnarray*}}

\newcommand{\R}{\mathbb{R}}

\renewcommand{\L}[1]{\mathbf{L^#1}}

\newcounter{whitney}
\refstepcounter{whitney}

\newcounter{ineqcounter}
\refstepcounter{ineqcounter}
\makeatletter
\def\ps@pprintTitle{%
\let\@oddhead\@empty
\let\@evenhead\@empty
\def\@oddfoot{}%
\let\@evenfoot\@oddfoot}
\makeatother
\usepackage{pgf,tikz}
\usetikzlibrary{arrows}
\usetikzlibrary{decorations.pathreplacing}
\begin{document}

\title{Positivity--preserving numerical scheme for hyperbolic systems with $\delta\,-$ shock solutions and its convergence analysis}

\author[myaddress1]{Aekta Aggarwal}
\ead{aektaaggarwal@iimidr.ac.in}

\address[myaddress1]{Indian Institute of Management, Prabandh Shikhar, Rau--Pithampur Road, Indore, Madhya Pradesh 453556.}

\author[myaddress2]{Ganesh Vaidyan\corref{cor1}}
\cortext[cor1]{Corresponding author}
\ead{ganesh@tifrbng.res.in}
\author[myaddress2]{G.~D.~Veerappa Gowda}
\ead{gowda@tifrbng.res.in}
\address[myaddress2]
{Centre for Applicable Mathematics, Tata Institute of Fundamental Research, Post Bag No 6503, Sharadanagar, Bangalore - 560065, India.}

\begin{abstract}
In this article convergent numerical schemes are proposed for approximating the solutions, possibly measure--valued with concentration (delta shocks), for a class of non--strictly hyperbolic systems. These systems are known to model physical phenomena such as the collision of clouds and dynamics of sticky particles, for example. The scheme is constructed by extending the theory of discontinuous flux for scalar conservation laws, to capture measure--valued solutions with concentration. The numerical approximations are analytically shown to be entropy stable in the framework of \cite{bouchut1994zero}, satisfy the physical properties of the state variables, and converge to the weak solution. The construction allows natural extensions of the scheme to its higher--order and multi--dimensional versions. The scheme is also extended for some more classes of fluxes, which admit delta shocks and are also known to model physical phenomena. Various physical systems are simulated both in one dimension and multi-dimensions to display the performance of the numerical scheme and comparisons are made with the test problems available in the literature.
\end{abstract}
\maketitle
\section{Introduction}\label{intro}
This paper studies the following $2 \times 2$ hyperbolic system
\begin{eqnarray}
\label{E1}\rho_t+F(\rho,w)_x&=&0,\\
\label{E2}w_t+G(\rho,w)_x&=&S(\rho,w),
\end{eqnarray}
where $F,G$ and $S$ are sufficiently smooth real-valued functions. The numerical approximation of such systems has been of interest, in the past decades, and the proposed numerical schemes have been primarily based on the eigenstructure of the system. The article aims to develop relatively simpler {convergent} numerical schemes for the system using the techniques for scalar conservation laws. If we assume that $w(x,t)$ is known for all $(x,t) \in \R \times \R^+$,  the equation \eqref{E1} can be viewed as a scalar conservation law in $\rho$, whose flux function $F(\rho,w(x,t))$ may be discontinuous in the space variable $x$, and vice versa for the second equation \eqref{E2}, but unfortunately, this
technique may not work in general. However, in this article, we
propose numerical methods based on this type of technique for certain class of $F$ and $G$, and also prove their convergence under some additional assumptions on the fluxes. 
From now on, we restrict ourselves to the following:
\[F(\rho, w)=w, G(\rho,w)=wg\left(\frac{w}{\rho}\right)+P\left(\rho,\frac{w}{\rho}\right).\]
In particular, the system \eqref{E1}-\eqref{E2} can be rewritten as:
\begin{eqnarray}\label{hy}
\rho_t + (\rho g(u))_x&=&0,\\
\label{hy1}
 w_t + (w g(u)+P(\rho,\:u))_x&=&S(\rho,w),
\end{eqnarray}
where $w=\rho u.$
This class finds numerous physical applications that depend on the nature of the functions $g,S$ and $P,$ and may admit non-classical shocks, which are called as $\delta\,-$ shocks.
 The $\delta\,-$ shock wave is a generalization of a classical shock wave and is a kind of discontinuity, on which the state variables of the system \eqref{hy}-\eqref{hy1} develop an extreme concentration in the form of a weighted Dirac delta function with the discontinuity as its support. Physically, the delta shock wave represents the process where the mass is concentrated and maybe interpreted as the galaxies in the universe. This generalization was introduced in the Ph.D. thesis of \cite{korchinski1978solution}, 
 post which, it has been explored extensively in the literature, see for example, \cite{danilov2005delta,panov2006delta, shelkovich2006riemann,tan1994delta} and references therein.
 
 These kind of systems have not only been of  mathematical interest due to the admission of $\delta\,-$ shock solution, but they are also known to model physically important phenomena. Some of the important ones are generalized pressureless gas dynamics \textbf{(GPGD)} system, where $g$ is a non-decreasing function, $S=0=P$, and isentropic Euler equations for modified Chaplygin gas dynamics  \textbf{(CGD)} system, where $g(u)=u,\:S=s\rho^{-\alpha}$ and $\:P=0$. \textbf{(GPGD)} system has been studied in \cite{bouchut1994zero,huang, mitrovic2007delta}. With $g(u)=u,$ the system is called as pressureless gas dynamics \textbf{(PGD)} system and can be used to describe the process of the motion of free particles sticking under collision. When \textbf{(PGD)} system is augmented with a Coulomb friction source term $S=\beta\rho$, see \cite{shen2015riemann}, it can be used to model the sticky particle dynamics with interaction. We will call it as \textbf{(PGDS)} in this paper.
\textbf{(CGD)} system was studied in \cite{chaplygin1944gas, wang2013riemann} and was shown to work as a suitable mathematical approximation to calculate the lift on a wing of an airplane in aerodynamics. The system also finds presence in cosmology, and is also used as a possible model for dark energy. 

There have been various numerical studies in the past for 
\textbf{(PGD)} system. To name a few, schemes were proposed in \cite{bouchut1994zero, bouchut2003numerical, leveque2004dynamics} which used the analytical expression for the shock location
derived in \cite{sheng1999riemann} to determine the numerical flux in the case, $u_l>0>u_r$. Generalised eigenvectors obtained from Jordan canonical form were used to construct 
flux difference splitting based numerical schemes in \cite{garg2020class}. First and second-order relaxation and kinetic schemes were proposed in \cite{berthon2006relaxation,bouchut2003numerical}, where the authors established that the solutions preserve both physical properties of the system and discrete entropy inequality proposed in \cite{bouchut1994zero}. Discontinuous Galerkin based higher order schemes were proposed in \cite{yang2013discontinuous} and the authors showed that the solutions preserve physical properties.
 Additionally, semi--discrete central--upwind scheme of \cite{bryson2005semi} and non-oscillatory central difference scheme of 
 \cite{nessyahu1990non} were used for approximating \textbf{(CGD)} in \cite{wang2013riemann} respectively.

This paper aims to construct efficient numerical schemes to approximate these systems and capture $\delta\,-$ shock solutions by suitably 
treating the system through, two interdependent scalar conservation laws with spatially dependent discontinuous flux.  Godunov type schemes will be constructed by solving appropriate local Riemann problems for each of them at each numerical interface, and {the set of two schemes thus obtained, will be taken as a scheme for the system}. The scheme will be called as \textbf{(DDF)} scheme, "Decoupled Discontinuous Flux Scheme" in the paper.
In \textbf{(GPGD)}, it can be noted that for a given $u(x,\:t)$, the first equation
\begin{equation}\label{1}
 \rho_t +(\rho g(u(x,\:t)))_x=0
\end{equation}
is a linear conservation law in $\rho$ with a possibly discontinuous variable coefficient $g(u(x,\:t))$. Similarly, given a $\rho(x,\:t)>0,\:$ the second equation
\begin{equation}\label{3}
 w_t +\left(w g\left(\frac{w}{\rho(x,\:t)}\right)\right)_x=0
\end{equation}
can be treated as a non-linear scalar conservation laws with discontinuous flux, if we assume the following \begin{eqnarray} \label{g1} g(0)=0,\:w \mapsto w g\displaystyle\left(\frac{w}{\rho(x,\:t)}\right) \text{is a function with only one local minimum} .\end{eqnarray}The above condition implies that $w \mapsto w g\displaystyle\left(\frac{w}{\rho(x,\:t)}\right) $ has a minimum at $w=0,$ independent of $\rho(\cdot, \cdot).
$ 
The two equations \eqref{1} and \eqref{3} behave differently as \eqref{3} is a non-linear conservation law with discontinuous convex flux, which admits bounded solutions, while \eqref{1} is a conservation law with a linear advective flux with sign changing coefficient, which may admit measure valued solutions. Both the conservation laws admit infinitely many solutions. It is thus necessary to choose appropriate individual entropy setups so that they are compatible to converge to the right expected physical solution of \eqref{hy}-\eqref{hy1}. We will choose the entropy setup of \cite{adimurthi2000conservation} for \eqref{3}, while, for \eqref{1}, the current setups of discontinuous flux, will be appropriately modified to capture measure valued solutions.
It is important to note that, the scheme proposed in \cite{agg1} together with the proposed scheme for \eqref{3}, does not converge to the right rarefaction solution of the system \eqref{hy}-\eqref{hy1}.
Hence, the construction of the scheme in the article for \eqref{1}, plays a crucial rule in the success of the numerical scheme.

 In this article, the schemes are constructed for \eqref{hy}-\eqref{hy1} in absence of pressure and source and then are adapted to \textbf{(PGDS)} which has non--zero $S$. The scheme is also extended to capture the $\delta\,-$
shock solutions of strictly hyperbolic systems like \textbf{(CGD)}
which has non--zero $P$. 
Since the scheme is based on theory of scalar conservation laws, it is easily extended to higher order, using appropriate limiters, and to multi-dimensions using the dimension splitting techniques. The scheme is tested with various initial data, modeling physical applications in both, one and multi dimensions. For \textbf{(GPGD)} system, with some additional assumptions on $g$, and in particular, for \textbf{(PGD)} system, the numerical solutions are shown to preserve the physical properties of the system. The numerical scheme is shown to converge to the weak solution satisfying the entropy inequality, introduced in \cite{bouchut1994zero}. Though the proofs are based on the idea of writing the scheme in the incremental form, the estimates on the incremental coefficients are non trivial and more involved unlike in the case of decoupled systems which are detailed in the last part of this article. 

The paper has been organized as follows: In \S\ref{LI}, we revisit the Riemann problems
for the scalar conservation law with one spatial discontinuity and derive the flux approximations at the interface of discontinuity. In \S\ref{sys}, we propose the \textbf{(DDF)} scheme to approximate \textbf{(GPGD)}. The scheme is also extended to its higher order version.  In \S\ref{st}, in addition to the convergence of the scheme, the numerical solutions are shown to satisfy the entropy inequality and  preserve the physical properties of the system, under certain assumptions on the fluxes. In \S\ref{NS}, the efficiency of the scheme, along with its extensions, is displayed, by comparing their performance with the existing literature. The schemes are also tested for hyperbolic systems that admit $\delta\,-$ shocks in presence of pressure and source terms such as \textbf{(PGDS)}
and \textbf{(CGD)}, and are also extended to multi dimensions, using dimensional splitting.
\section{Preliminaries}\label{LI}
This section aims to discuss the notion of solution for the transport equation with spatial discontinuity given by:
\begin{equation}\label{scalar}
\begin{array}{lll}
\rho_t + F(x,\:\rho)_x &=& 0  \,\,\quad \quad\quad(x,\:t)\in \R\times\R^+, \\
\quad\quad\quad\rho(x,\:0) &=&\rho_0(x)\quad\quad  x\in \R, 
\end{array} 
\end{equation}
where $F(x,\:\rho)=H(x)b\rho + (1-H(x))a\rho.$
For Riemann data, 
all cases except for $a\ge0,\:b\le0,\:$ can be handled by \cite{adimurthi2005optimal,adimurthi2000conservation}. The overcompressive pair $a\ge0,\:b\le0,\:$ was recently studied in \cite{agg1}, where measure valued solutions were proposed, while bounded solutions were proposed in \cite{mishra2005}.
In the overcompressive case, 
characteristics overlap each other at the interface $x=0$ and cases may arise when there may not exist a weak solution satisfying: 
 \begin{eqnarray*}\label{weak1}\begin{aligned}
\displaystyle \int_0^\infty \displaystyle \int_{-\infty}^{\infty} \left(\rho(x,\:t) \phi_t(x,\:t)+F (x,\:\rho(x,\:t)) \phi_x(x,\:t)\right) dxdt=0,\:\forall \phi \in C_c^{\infty}(\R\times\R^+)\end{aligned}.\end{eqnarray*} 
To this end, we concentrate on the case, $a\ge0,\:b\le0$ and 
for each $\epsilon>0,\:$, we consider a non-linear approximation of \eqref{scalar}:
\begin{equation}\label{appro}
 \rho_t+(F_{\epsilon}(x,\:\rho))_x=0,\:F_{\epsilon}(x,\:\rho)=H(x)f_{\epsilon}(\rho) + (1-H(x))g_{\epsilon}(\rho)
\end{equation}
 with initial data as Riemann data $(\rho_l,\:\rho_r)$ and with
\begin{eqnarray}
\label{g}
g_{\epsilon}(\rho) =(a\rho-a\epsilon \rho^2)\chi_{\raisebox{-0.5ex}{$\scriptstyle\{\rho\ge0\}$}}+(a\rho+a\epsilon \rho^2)\chi_{\raisebox{-0.5ex}{$\scriptstyle\{\rho<0\}$}}, \\
 \label{f}f_{\epsilon}(\rho)=(b\rho-b\epsilon \rho^2)\chi_{\raisebox{-0.5ex}{$\scriptstyle\{\rho\ge0\}$}}+(b\rho+b\epsilon \rho^2)\chi_{\raisebox{-0.5ex}{$\scriptstyle\{\rho<0\}$}}.
 \end{eqnarray}
 It can be noted that the above nonlinearification is not the same as the one proposed in \cite{agg1} and depends on $a$ and $b$ linearly. This choice is crucial as pointed out in \S\ref{intro} and will also be seen in \S\ref{sys}.

\begin{figure}[H]
 \centering
 \includegraphics[scale=.1]{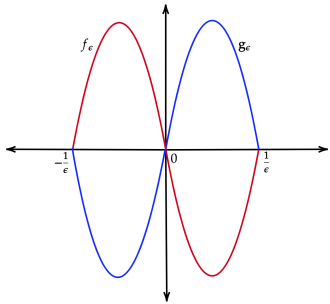}
 \caption{Flux Structure}
 \label{Figallf}
\end{figure}
The flux structure for the equation \eqref{appro} is displayed in \textbf{Figure~\ref{Figallf}}.
 Various cases arise and can be handled by theory of discontinuous flux,\:
 detailed in \cite{adimurthi2005optimal, adimurthi2000conservation, mishra2existence,mishra2005}, which are briefly exhibited below for completeness:
\begin{enumerate}[(a)]
 \item \underline{$\rho_l,\:\rho_r\le0 :$}
 Note that $\displaystyle g_\epsilon, \displaystyle f_\epsilon$ are convex and concave respectively. The solution $\rho^\epsilon(x,\:t)$ is given in \textbf{Figure~\ref{Figcase1}}, where \[\left(\rho_l-\frac{-1}{\epsilon}\right)s_g=g_{\epsilon}(\rho_l),\:\left(\rho_r-\frac{-1}{\epsilon}\right)s_f=f_{\epsilon}(\rho_r).\]
 \begin{figure}[H]
 \centering
 \includegraphics[scale=.2]{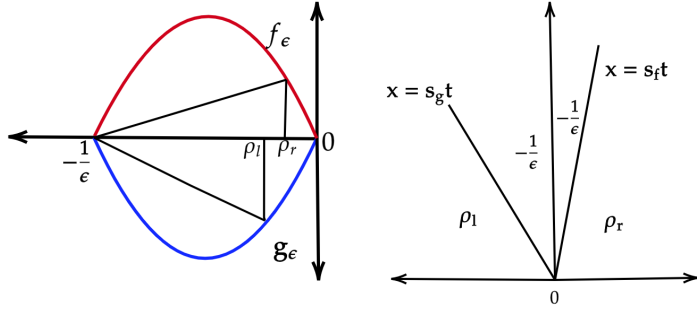} 
 \caption{$\rho_l,\:\rho_r\le0$}
 \label{Figcase1}
\end{figure}
\item \underline{$\rho_l,\:\rho_r\ge0 :$}
Note that $\displaystyle g_\epsilon,\:\displaystyle f_\epsilon$ are concave and convex respectively.\begin{figure}[H]
 \centering
 \includegraphics[scale=.2]{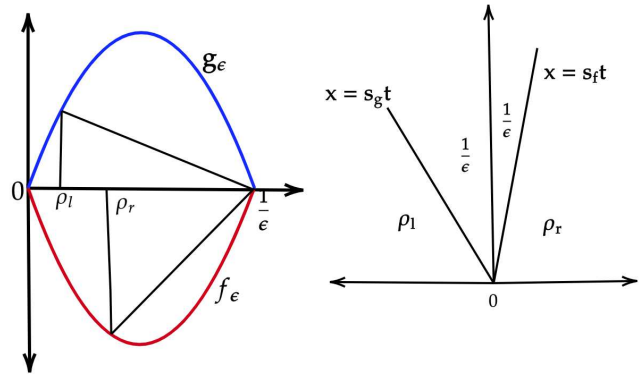}
 \caption{$\rho_l,\:\rho_r\ge0$}
 \label{Figcase2}
 \end{figure}
The solution $\rho^\epsilon(x,\:t)$ is given in \textbf{Figure~\ref{Figcase2}}, where \[\left(\rho_l-\displaystyle \frac{1}{\epsilon}\right)s_g=g_{\epsilon}(\rho_l),\:\left(\rho_r-\displaystyle \frac{1}{\epsilon}\right)s_f=f_{\epsilon}(\rho_r).\]
 \item \underline{$\rho_l<0,\:\rho_r>0,\:$}:\\
 Note that $\displaystyle g_\epsilon,\:\displaystyle f_\epsilon$ are convex and convex respectively.
 \begin{figure}[H]
 \centering
 \includegraphics[scale=.2]{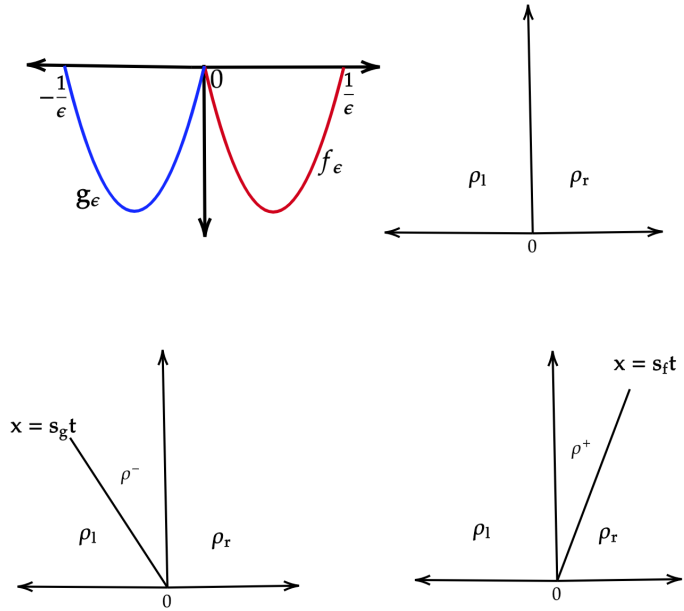}
 \caption{$\rho_l<0,\:\rho_r>0,\:$}
 \label{Figcase3}
 \end{figure}
The flux structure and the solution $\rho^\epsilon(x,\:t)$ is given in \textbf{Figure~\ref{Figcase3}}.
The second, third and fourth figures in \textbf{Figure~\ref{Figcase3}} represent the cases $\displaystyle g_\epsilon(\rho_l)= \displaystyle f_\epsilon(\rho_r)$, $\displaystyle g_\epsilon(\rho_l)< \displaystyle f_\epsilon(\rho_r)$ and $\displaystyle g_\epsilon(\rho_l)> \displaystyle f_\epsilon(\rho_r)$ respectively, where, 
 \begin{equation*}
 g_{\epsilon}(\rho^-)=\displaystyle f_\epsilon(\rho_r),\:
(\rho_l-\rho^-)s_g=g_{\epsilon}(\rho_l)-g_{\epsilon}(\rho^-)
\end{equation*} and \begin{equation*}
 g_{\epsilon}(\rho_l)=\displaystyle f_\epsilon(\rho^+),\:
(\rho_r-\rho^+)s_f=f_{\epsilon}(\rho_r)-f_{\epsilon}(\rho^+).
\end{equation*}
\item \underline{$\rho_l\ge0,\:\rho_r\le0$}:
 \begin{figure}[H]
 \centering
 \includegraphics[scale=.2]{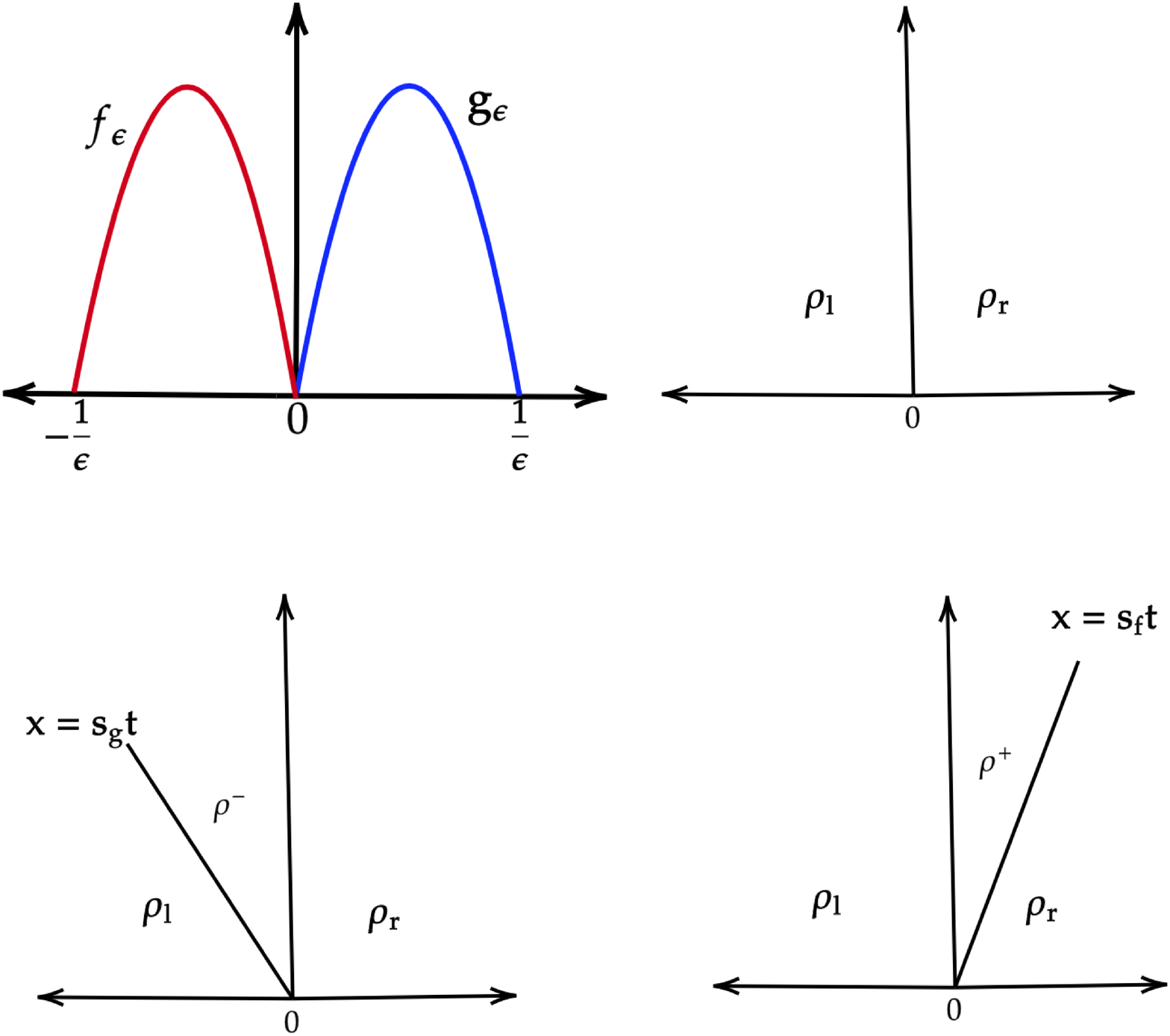}
 \caption{$\rho_l\ge0,\:\rho_r\le0$}
 \label{Figcase4}
 \end{figure}
 Note that $\displaystyle g_\epsilon,\:\displaystyle f_\epsilon$ are concave and concave respectively. The
 flux structure and the solution $\rho^\epsilon(x,\:t)$ is given in \textbf{Figure~\ref{Figcase4}}. The second third and fourth figures in \textbf{Figure~\ref{Figcase4}} represent the cases $\displaystyle g_\epsilon(\rho_l)= \displaystyle f_\epsilon(\rho_r)$, $\displaystyle g_\epsilon(\rho_l)> \displaystyle f_\epsilon(\rho_r)$ and $\displaystyle g_\epsilon(\rho_l)< \displaystyle f_\epsilon(\rho_r)$ respectively, where, 
 \begin{equation*}g_{\epsilon}(\rho^-)=\displaystyle f_\epsilon(\rho_r),\:
(\rho_l-\rho^-)s_g=g_{\epsilon}(\rho_l)-g_{\epsilon}(\rho^-),\end{equation*}
and
\begin{equation*}g_{\epsilon}(\rho_l)=\displaystyle f_\epsilon(\rho^+),\:
(\rho_r-\rho^+)s_f=f_{\epsilon}(\rho_r)-f_{\epsilon}(\rho^+).\end{equation*}
\end{enumerate}
It can be deduced that the point wise limit
\begin{equation}\label{AG}
\lim_{\epsilon\rightarrow 0} \rho^\epsilon(x,\:t)=\overline{\rho}(x,\:t):=\rho_l\chi_{\raisebox{-0.5ex}{$\scriptstyle\{x<0\}$}}+\rho_r\chi_{\raisebox{-0.5ex}{$\scriptstyle\{x>0,\:\}$}}\end{equation}
is, in fact, the solution proposed in \cite{adimurthi2005optimal}.
It is worthwhile noting
that the point wise limit in \eqref{AG} does not respect the conservation of mass in the interval $[\alpha,\:\beta],\:\alpha<0<\beta$ as\[0=\frac{d}{dt}\displaystyle \int_\alpha^\beta \rho(x,\:t)dx \ne -b \rho_r +a \rho_l. \]
Instead, the weak convergence of $\{\rho_{\epsilon}\}_{\{\epsilon>0\:\}} \in \Lloc1({\R})$ in the space of signed Radon measures 
gives \begin{equation}\rho(x,\:t):=\overline \rho(x,\:t)+t(a\rho_l-b\rho_r)\delta_{\{x=0\}}, 
\end{equation}
which takes care of the missing mass by concentrating it at the point $x=0,\:$ through the term $t(a\rho_l-b\rho_r)\delta_{\{x=0\}}$. This motivates to look for solutions of the type 
\begin{equation}\label{Definition} \overline \rho(x,\:t)+ w_\delta(t)\delta_{0}\end{equation} which solves the problem \eqref{scalar} in the following sense: $\forall \phi \in C_c^{\infty}(\R \times (0, \infty)), $
\begin{equation}\begin{aligned}\label{weak}
\displaystyle \int_0^\infty\displaystyle \int_{-\infty}^{\infty} \left(\overline \rho(x,\:t) \phi_t(x,\:t)+F (x,\:\overline \rho(x,\:t)) \phi_x(x,\:t)\right) dxdt\\
+\displaystyle \int_0^\infty w_\delta(t) \phi_t(0, t)ds=0.
\end{aligned}\end{equation}
Since the solution of the Riemann problem \eqref{scalar} is measure valued at the interface $x=0,\:$ the Godunov flux at $x=0$ for the numerical scheme for \eqref{scalar} cannot be evaluated in the usual way.

 Other cases except $a\ge0,\:b\le0$, can also be obtained with $g_{\epsilon}(\rho) =a\rho-a\epsilon \rho^2,\:f_{\epsilon}(\rho)=b\rho-b\epsilon \rho^2.$
Also, 
 $\lim_{\epsilon\rightarrow0}g_{\epsilon}(\rho)=a\rho,\:\lim_{\epsilon\rightarrow0}f_{\epsilon}(\rho)=b\rho,\:
 \rho(x,\:t)=\lim_{\epsilon\rightarrow0}\rho_{\epsilon}(x,\:t).$
 The solution for the Riemann Problem \eqref{appro} is known and hence the flux at the interface $x=0$ for \eqref{appro} is given by \[F_{\epsilon,\:0}(a,\:b,\:\rho_l,\:\rho_r):=g_{\epsilon}(\rho_{\epsilon}^-)=f(\rho_{\epsilon}^+), \]where $\rho_\epsilon^-=\lim_{x\rightarrow0^{-}}\rho_\epsilon(x,\:t),\:\rho_\epsilon^+=
 \lim_{x\rightarrow0^{+}}\rho_\epsilon(x,\:t)$.
It can be derived using the results in \cite{adimurthi2005optimal, adimurthi2004godunov} that 
$F_{\epsilon,\:0}(a,\:b,\:\rho_l,\:\rho_r)=$
 \begin{equation}
{ \left\{\begin{array}{ccl}
\min\left(\displaystyle g_\epsilon\left(\min(\rho_l,\:\displaystyle \frac{1}{\epsilon})\right),\:\displaystyle f_\epsilon\left(\max(\rho_r,\:\displaystyle \frac{1}{\epsilon})\right)\right) & \, \mbox{if}\, &a\geq0,\:b>0,\:\\[1mm]
\max\left(\displaystyle g_\epsilon\left(\max(\rho_l,\:\displaystyle \frac{1}{\epsilon})\right),\:\displaystyle f_\epsilon\left(\min(\rho_r,\:\displaystyle \frac{1}{\epsilon})\right)\right) & \, \mbox{if}\, &a<0,\:b\leq0,\:\\[1mm]
{\displaystyle \left\{\begin{array}{ccl}
\displaystyle \max\Big(\displaystyle g_\epsilon(\rho_l), \displaystyle f_\epsilon(\rho_r)\Big) & \, \, \mbox{if}\, \, &\rho_l<0,\:\rho_r> 0 , \\[.1mm]
\displaystyle \min\Big(\displaystyle g_\epsilon(\rho_l), \displaystyle f_\epsilon(\rho_r)\Big) & \, \, \mbox{if}\, \, & \rho_l>0,\:\rho_r< 0 , \\[.1mm]
\end{array}\right.} & \, \, \mbox{if}\, \, &a\geq0,\:b\leq0,\:\\[.1mm]
0 & \, \, \, \, & \mbox{otherwise}.
\end{array}\right.}.\label{overapp}
\end{equation}
 Owing to the behavior of the solutions and the fluxes of \eqref{scalar} and \eqref{appro} as $\epsilon\rightarrow0$, we define the flux at the interface $x=0$ for \eqref{scalar} as
 \[F_0(a, b, \rho_l,\:\rho_r):=\lim_{\epsilon\rightarrow0}g_{\epsilon}(\rho^-)=\lim_{\epsilon\rightarrow0}f_{\epsilon}(\rho^+)=\lim_{\epsilon\rightarrow0}F_{\epsilon, 0}(a, b, \rho_l,\:\rho_r), \]
 which implies that $F_0(a, b, \rho_l,\:\rho_r)=$
 \begin{equation}\label{over}
{ \left\{\begin{array}{ccl}
a\rho_l & \, \mbox{if}\, &a\geq0,\:b>0,\:\\[.1mm]
b\rho_r & \, \mbox{if}\, &a<0,\:b\leq0,\:\\[.1mm]
{\displaystyle \left\{\begin{array}{ccl}
\displaystyle \max\Big(a\rho_l,\:b\rho_r\Big) & \, \, \mbox{if}\, \, &\rho_l<0,\:\rho_r> 0 , \\[.1mm]
\displaystyle \min\Big(a\rho_l,\:b\rho_r\Big) & \, \, \mbox{if}\, \, & \rho_l>0,\:\rho_r< 0 , \\[.1mm]
\end{array}\right.} & \, \, \mbox{if}\, \, &a\geq0,\:b\leq0,\:\\[.1mm]
0 & \, \, \, \, & \mbox{otherwise}.
\end{array}\right.}
\end{equation}
Knowing the flux at the interface $x=0,\:$ a finite volume scheme can now be proposed with general initial data $\rho_0(x).$ For $h>0,\:$ let the space grid points as $\displaystyle x_{i+\frac{1}{2}}=ih,\:i\in\mathbb{Z}$ such that $x_\frac{1}{2}=0.$
For  $\delta\,- t>0,\:$, 
define the time discretization points $t_n=n\Delta t $ for non-negative integer $n$, and $\lambda=\Delta t/h.$
Define
 $\rho_{i}^{n}=\displaystyle\frac{1}{h}\int_{C_i}\rho(x,\:t^n)dx,\:$
as the approximation for $\rho(x,\:t)$ in the cell $C_{i}=[x_{i-\frac{1}{2}} ,\:x_{i+\frac{1}{2}})$ at time $t_n$. Then, the finite volume scheme is given by
\begin{eqnarray}\label{linear}
\rho_{i}^{n+1}=\rho_{i}^{n}-\lambda\big(\hat{F}(a_{i},\:a_{i+1},\:\rho_{i}^{n},\:\rho_{i+1}^{n})-\hat{F}(a_{i-1}, a_{i},\:\rho_{i-1}^{n},\:\rho_{i}^{n})\Big),\:
\end{eqnarray}
where, 
\begin{equation*}a_i=
{ \left\{\begin{array}{ccl}
a & \, \mbox{if}\, &i\le0,\:\\[2mm]
b & \, \mbox{if}\, & i>0,\:
\end{array}\right.}, \, \, \, 
\hat{F}(a_{i},\:a_{i+1},\:\rho_{i}^{n},\:\rho_{i+1}^{n})=
F_0(a_{i},\:a_{i+1},\:\rho_{i}^{n},\:\rho_{i+1}^{n}), 
\end{equation*}
where $\hat{F}(a_{i},\:a_{i+1},\:\rho_{i}^{n},\:\rho_{i+1}^{n})$ is the numerical flux associated with the flux
$F(x,\:\rho)$ at the interface $x_{i+\frac{1}{2}}$ at the time $t^n$, with $F_0$ given by \eqref{over}. 
Since $F(x,\:\rho)=
a\rho$ if $x<0$ and $b\rho$ for $x>0,\:$ 
with $a\ge0,\:b\le0,\:$ the flux at any point away from the point $x_\frac{1}{2}=0$ is the usual upwind flux for the linear transport equation. At the interface $x_\frac{1}{2}=0,\:$ though the solution is measure-valued and a Godunov flux cannot be calculated in the usual way, we have, however, the flux,\:owing to the non--linearification in the previous section and is given by $F_0(a,\:b,\:\rho_i^n ,\:\rho_{i+1}^n), $ whose expression is given by \eqref{over}.
 

\section{Numerical Scheme}\label{sys}
We start by proposing a numerical scheme for \eqref{hy} with $S=P=0$. Instead of creating a Riemann solver based on the eigenstructure, each equation of the system will be treated separately, assuming that the flux of the equation is a function of the remaining state variable at the previous time step. Consider the system \textbf{(GPGD)}
\begin{equation}\label{pressurelessU}
\begin{array}{ccl}
U_t+ F(U)_x=0,\:U_0(x)=U(x,\:0), 
\end{array}
\end{equation}
 with \begin{equation}\label{FP}
 U=\begin{pmatrix}
 \rho\\
 w
 \end{pmatrix},\:F(U)=\begin{pmatrix}
 F^\rho\\
 F^w 
 \end{pmatrix},\:F^\rho(\rho,\:w)=\rho g\Big(\displaystyle\frac{w}{\rho}\Big),\:F^w(\rho,\:w)= w g\Big(\displaystyle\frac{w}{\rho}\Big),\:\rho>0.\end{equation} 
 This system has double eigenvalue $g(u)$ and it has been shown in \cite{mitrovic2007delta} that it admits $\delta\,-$ shocks. Therefore, as in the case of strictly hyperbolic systems, traditional approximate/exact Riemann solvers cannot be used here.
We now propose the \textbf{(DDF)} scheme.

 \subsection{\textbf{(DDF)} Scheme}
 With the notations same as before, consider
Let
{
\begin{eqnarray*}&U_{i}^{n}=\displaystyle\frac{1}{h}\int_{C_i}U(x,\:t^n)dx,\:\quad u_{i}^{n}=\displaystyle\frac{w_i^n}{\rho_i^n}
 \end{eqnarray*}
as the approximation for $U$ and $u$ in the cell $C_{i}=[x_{i-\frac{1}{2}} ,\:x_{i+\frac{1}{2}})$ at time $t_n$. Let $(U_i^n ,\:U_{i+1}^n):=(\rho_i^n,\:\rho_{i+1}^n,\:w_i^n,\:w_{i+1}^n)$. Then, the finite volume scheme for the system \eqref{pressurelessU} is given by:
\begin{equation}\label{Numer}
\begin{array}{ccl}
U_{i}^{n+1}&=&U_{i}^{n}-\lambda\big(\hat{F}(U_{i}^{n},\:U_{i+1}^{n})-\hat{F}(U_{i-1}^{n},\:U_{i}^{n})\big)
\end{array}, 
\end{equation}
where
\begin{equation}\label{Numer1}
 \hat{F}(U_{i}^{n},\:U_{i+1}^{n})=\begin{pmatrix} \hat{F}^\rho(U_{i}^{n},\:U_{i+1}^{n})\\[1mm]\hat{F}^{w}(U_{i}^{n},\:U_{i+1}^{n})\end{pmatrix}
\end{equation}
where $\hat{F}^{\rho}(U_{i}^{n},\:U_{i+1}^{n}),\:\hat{F}^{w}(U_{i}^{n},\:U_{i+1}^{n})$ and $\hat{F}(U_{i}^{n},\:U_{i+1}^{n})$ are the numerical fluxes associated with $F^\rho,\:F^w $ and $F$ (defined by \eqref{FP}), at 
$ x_{i+\frac{1}{2}}$ at time $t^n$\@. We will have the following fluxes:
\begin{equation}\label{FrhoN}
 \hat{F}^{\rho}(U_{i}^{n},\:U_{i+1}^{n})=F_{0}\Big(g(u_i^n),\:g(u_{i+1}^n),\:\rho_i^n,\:\rho_{i+1}^n\Big),\:
\end{equation}}
 where $F_{0}$ is given by \eqref{over} and \begin{equation}\label{FwN}
 \hat{F}^{w}(U_i^{n},\:U_{i+1}^{n})=\max\Big(F^{w}\left(\rho_{i}^{n},\:\max(w_i^n,\:0)\right),\:F^{w}\left(\rho_{i+1}^{n},\:\min(w_{i+1}^n,\:0)\right)\Big).
\end{equation}
It is to be noted that $ \hat{F}^{\rho}(U_{i}^{n},\:U_{i+1}^{n})$ is computationally less expensive than the one proposed in \cite{agg1}.
The flux $\hat{F}^\rho$ is obtained by solving the local Riemann problems at the interface $x_{i+\frac{1}{2}}$ as described below:\\
On each $C_i\times (t^n,\:t^{n+1})$, we look at the conservation law, 
\begin{equation*}
\rho_{t}+(F^\rho(u_i^n,\:\rho))_{x}=0
,\end{equation*}
with $F^\rho(u_i^n,\:\rho)=g(u_{i}^{n})\rho$ and the initial condition $\rho(x,\:t^n)=\rho_i^{n}$ for $x\in (x_{i-\frac{1}{2}} ,\:x_{i+\frac{1}{2}})$.
\begin{figure}[H]
 \centering
 \includegraphics[width=\textwidth,keepaspectratio]{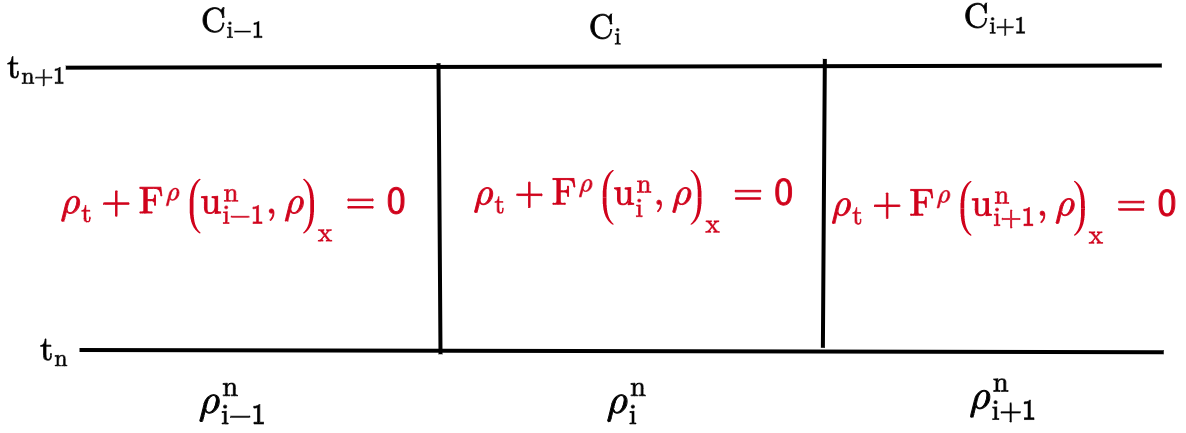}
 \caption{Local Riemann Problem Structure for \eqref{rei1s}}
 \label{FigL0}
\end{figure}
Hence, the problem reduces to the corresponding local Riemann problem, as in \cite{adimurthi2016godunov, aggarwal2016godunov}, 
 \begin{equation}\label{rei1s}
 \rho_t+l^{\rho}(x,\:\rho)_x=0\, \, \, \, \, \, \, \, \mbox{in} \, \, \, (x_{i-\frac{1}{2}} ,\:x_{i+\frac{1}{2}})\times(t^n,\:t^{n+1})
,\end{equation}
 where
\[l(x,\:\rho)=\displaystyle \left\{\begin{array}{ccl}
F^\rho(u_i^n,\:\rho) & \, \, \mbox{if}\, \, & x<x_{i+\frac{1}{2}},\:\\
F^\rho(u_{i+1}^n,\:\rho) & \, \, \mbox{if}\, \, & x>x_{i+\frac{1}{2}}
,\end{array}\right.\]
with the initial data
\[\rho(x,\:t^n)=\displaystyle \left\{\begin{array}{ccl}
\rho_i^{n} & \, \, \mbox{if}\, \, & x<x_{i+\frac{1}{2}},\:\\
\rho_{i+1}^{n}& \, \, \mbox{if}\, \, & x>x_{i+\frac{1}{2}}
.\end{array}\right.\]
Each local Riemann Problem \eqref{rei1s} at the interface $x_{i+\frac{1}{2}}$ is of the form \eqref{scalar} and the flux at each interface 
is given by $F_0(g(u_i^n),\:g(u_{i+1}^n),\:\rho_i^n,\:\rho_{i+1}^n).$
 For the second equation \begin{equation}\label{second}
 w_t+\Big(wg\Big(\frac{w}{\rho(x,\:t)}\Big)\Big)_x=0,\end{equation}let us assume that $\rho(x,\:t)$ is a known function at time $t^n$ which is allowed to be discontinuous at the space discretization points\@. Therefore on each $C_i\times (t^n,\:t^{n+1})$, we look at the conservation law, 
\begin{equation*}
{w _{t}+\Big(wg\Big(\frac{w}{\rho_i^n}\Big)\Big)_x}=0
,\end{equation*}
 with the initial condition $w(x,\:t^n)=w_i^{n}$ for $x\in (x_{i-\frac{1}{2}} ,\:x_{i+\frac{1}{2}})$. {Hence, the problem reduces to the corresponding local Riemann problem 
 \begin{equation}\label{rei4s}
 w_t+l^{w}(x,\:w)_x=0\, \, \, \, \, \, \, \, \mbox{in} \, \, \, (x_{i-\frac{1}{2}} ,\:x_{i+\frac{1}{2}})\times(t^n,\:t^{n+1})
,\end{equation}
 where
\[l^{w}(x,\:w)=\displaystyle \left\{\begin{array}{ccl}
\displaystyle wg\Big(\frac{w}{\rho_i^n}\Big) & \, \, \mbox{if}\, \, & x<x_{i+\frac{1}{2}},\:\\[2mm]
\displaystyle wg\Big(\frac{w}{\rho_{i+1}^n}\Big) & \, \, \mbox{if}\, \, & x>x_{i+\frac{1}{2}}
,\end{array}\right.\]
with the initial data
\[w(x,\:t^n)=\displaystyle \left\{\begin{array}{ccl}
w_i^{n} & \, \, \mbox{if}\, \, & x<x_{i+\frac{1}{2}},\:\\
w_{i+1}^{n}& \, \, \mbox{if}\, \, & x>x_{i+\frac{1}{2}}
.\end{array}\right.\]
Using theory of discontinuous flux for convex conservation laws of \cite{adimurthi2016godunov, aggarwal2016godunov}, we get the following required flux at each interface $x_{i+\frac{1}{2}},\:$
\begin{equation*}
 \hat{F}^w(U_{i}^{n},\:U_{i+1}^{n})=\max\Big(F^{w}\left(\rho_{i}^{n},\:\max(w_i^n,\:0)\right),\:F^{w}\left(\rho_{i+1}^{n},\:\min(w_{i+1}^n,\:0)\right)\Big).
\end{equation*}}
We now provide a higher order extension of  the scheme.
\subsection{Higher Order Extension of \textbf{(DDF)} scheme}\label{hi}
We first consider the forward Euler time discretization. Let $\overline{U}_i^n$ be the cell average of $U(x,\:t)$ in the cell $C_i$ at time $t^n$. Let $U_{i+\frac{1}{2}}^{L,\:R}$ be the second-order approximations of $U(x,\:t^n)$ at the cell interface $x_{i+\frac{1}{2}}$ within the cell $C_i$ and $C_{i+1}$, which are reconstructed from the cell average $\overline{U}_i^n$. Let $\left(U_{i+\frac{1}{2}}^{L},\:U_{i+\frac{1}{2}}^{R}\right):=\left(\rho_{i+\frac{1}{2}}^{L},\:\rho_{i+\frac{1}{2}}^{R},\:w_{i+\frac{1}{2}}^{L},\:w_{i+\frac{1}{2}}^{R}\right)$ and $u_{i+\frac{1}{2}}^{ L,\:R}=\displaystyle\frac{w_{i+\frac{1}{2}}^{ L,\:R}}{\rho_{i+\frac{1}{2}}^{ L,\:R}}.$ The finite volume scheme approximating \eqref{pressurelessU} is given by:
\begin{equation}\label{seco}
\overline{U}_i^{n+1}=\overline{U}_i^n-
\lambda\left[\hat{F}\left(U_{i+\frac{1}{2}}^{L},\:U_{i+\frac{1}{2}}^{R}\right)-\hat{F}\left(U_{i-\frac{1}{2}}^{L},\:U_{i-\frac{1}{2}}^{R}\right) \right], 
\end{equation}
where $\hat{F}$ is given by \eqref{Numer1} and is the first-order numerical flux associated with the flux $F$ at the interface
$ x_{i+\frac{1}{2}}$ at time $t^n$\@, which preserves the positivity of $\rho$ and bounds of the velocity $u$. 
For $z=\rho,\:u, $ let
\[z_{i+\frac{1}{2}}^{R,\:L}:=p^z_{i}(x_{i\pm\frac{1}{2}}), \overline{z}_i^n=\frac{1}{2}\left(z_{i-\frac{1}{2}}^{R}+ z_{i+\frac{1}{2}}^{L}\right),\:\]
 where $p^{z}_i(x)$ is the linear approximation in the cell $C_i$ for the piecewise constant solution $z_i^n$ such that \[\frac{1}{h}\int_{C_i}p^{z}_i(x)dx=\overline{z}_i^n,\:p^{z}_i(x)=\overline{z}_i^n+ \sigma^z_i(x-x_i).\]
The slope $\sigma^{z}_i$ is controlled by the choice of a suitable limiter ensuring that the physical properties of the system are preserved. Define $p^{w}_i(x):=p^{\rho}_i(x)p^{u}_i(x).$
For higher order time discretization, we use 2nd order Strong Stability Preserving Runge Kutta Method of \cite{shu1989efficient}, which we summarize below:
\begin{eqnarray}\label{RK}
 U_i^{*}&=&U_i^n-\lambda\left[\hat{F}\left(U_{i+\frac{1}{2}}^{L},\:U_{i+\frac{1}{2}}^{R}\right)-\hat{F}\left(U_{i-\frac{1}{2}}^{L},\:U_{i-\frac{1}{2}}^{R}\right) \right], \\
 \label{RK1}
U_i^{**}&=&U_i^{*}-\lambda\left[\hat{F}\left(U_{i+\frac{1}{2}}^{* ,\:L}, U_{i+\frac{1}{2}}^{* ,\:R}\right)-\hat{F}\left(U_{i-\frac{1}{2}}^{* ,\:L}, U_{i-\frac{1}{2}}^{* ,\:R}\right) \right], \\
\label{RK2}
U_i^{n+1}&=&\frac{1}{2}(U_i^n+U_i^{**}), 
 \end{eqnarray}
 $U_{i+\frac{1}{2}}^{*,  \:L,\:R}$ is reconstructed from $U_i^{*}, $ in the same way as 
 $U_{i+\frac{1}{2}}^{ L,\:R}$ is reconstructed from $U_i^n.$

\section{Stability and Convergence Analysis}\label{st}
We now prove the numerical solutions given by the first order \textbf{(DDF)} scheme are entropy stable in the framework of \cite{bouchut1994zero}, satisfy the physical properties of the state variables, and converge to the weak solution of \textbf{(GPGD)} in one space dimension. 

Let us define the piece-wise constant approximate solution to \textbf{(GPGD)}, 
$U_{h} (x,\:t)=\begin{pmatrix}\rho_{h}\\w_{h}
 \end{pmatrix}$ such that
 $U_{h}(x,\:t) = U^{n}_{i}=\begin{pmatrix}\rho_i^n\\w_i^n \end{pmatrix},\:
 t\in[t^n,\:t^{n+1}),\:x\in C_i,\:n \in\mathbb{N},\:
 i\in\mathbb{Z},$
where $U_i^n$ is the numerical solution obtained by the $3$-points algorithm \eqref{Numer}. 
Physically, the density $\rho\ge0$ and the velocity $u$ satisfies the maximum principle.
 Let \begin{eqnarray*}S=\left\{(\rho,\:w):\rho\ge0,\:
 m\rho\le w\le M\rho\right\}\end{eqnarray*}
 where $m=\min_{x\in\R}g(u_0(x)),\:M=\max_{x\in\R}g(u_0(x)).$
{We denote $F^{z,\:n}_{i+\frac{1}{2}}:=\hat{F}^z(U_{i}^{n},\:U_{i+1}^{n}), z=\rho,\:w.$}
 We then have the following theorem:
 \begin{theorem}\label{theorem1}
 Under the CFL--like condition \eqref{cfl}, 
 $U^{n+1}\in S$ if $U^{n}\in S$.
 \end{theorem}
 The above theorem is a consequence of the following two lemmas:
 \begin{lemma}[Positivity of $\rho^n$]\label{lemma1}
For each $n\in \mathbf{N},\:$
 \[\rho^n>0,\:\implies\rho^{n+1}>0,\:\] under the condition \begin{equation}\label{cfl}
 \lambda\max_{i,\:n}|g(u_{i}^n)|\leq 1.
 \end{equation}
 \end{lemma}
 \begin{proof}
 {It can be easily shown that $F^{\rho,\:n}_{i+\frac{1}{2}}$ given by \eqref{FrhoN}} is increasing in the variable $\rho_i^n$ and decreasing in the variable $\rho_{i+1}^n$
%
and hence,
\begin{equation*}
 1-\lambda\left(\frac{\partial F^{\rho,\:n}_{i+\frac{1}{2}}}{\partial\rho_i^n}-\frac{\partial F^{\rho,\:n}_{i-\frac{1}{2}}}{\partial\rho_i^n}\right)= { \left\{\begin{array}{ccl}
1-\lambda g(u_{i}^n)& \, \mbox{if}\, &g(u_{i}^n)>0,\:g(u_{i+1}^n)>0,\:\\[.1mm] 
 1+\lambda g(u_{i}^n)& \, \mbox{if}\, &g(u_{i-1}^n)<0,\:g(u_{i}^n)<0,\:\\[.1mm] 
 1 & \, \mbox{else}, 
 \end{array}\right.}
\end{equation*}
which shows that $\displaystyle\frac{\partial \rho_i^{n+1}}{\partial\rho_j^n}\geq 0\, \, \, \forall i,\:j=\{i\pm1,\:i\}$ under the condition \eqref{cfl} which gives the desired result.
 \end{proof} 
Using \eqref{Numer}, we have the following:
\[
 \rho_i^{n+1}(u_i^{n+1}-u_i^n)=
-\lambda(F^{w,\:n}_{i+\frac{1}{2}}-F^{\rho,\:n}_{i+\frac{1}{2}} u_i^n)+\lambda(F^
{w,\:n}_{i-\frac{1}{2}}-F^{\rho,\:n}_{i-\frac{1}{2}} u_i^n), \]
 and hence, we have
\begin{equation}\label{incre}
 u_i^{n+1}=u_i^n(1-\displaystyle C^{n}_{i-\frac{1}{2}}-\displaystyle D^{n}_{i+\frac{1}{2}})+u_{i-1}^n\displaystyle C^{n}_{i-\frac{1}{2}}+u_{i+1}^{n}\displaystyle D^{n}_{i+\frac{1}{2}},\:
\end{equation}
where for each $i,\:n,$
\begin{equation}\label{C}
 \displaystyle C^{n}_{i-\frac{1}{2}}=-\displaystyle\lambda\frac{F^{w,\:n}_{i-\frac{1}{2}}-F^{\rho,\:n}_{i-\frac{1}{2}}u_i^{n}}{\rho_i^{n+1}(u_{i}^{n}-u_{i-1}^n)}, \quad \displaystyle D^{n}_{i+\frac{1}{2}}=-\displaystyle\lambda\frac{F^{w,\:n}_{i+\frac{1}{2}}-F^{\rho,\:n}_{i+\frac{1}{2}}u_i^{n}}{\rho_i^{n+1}(u_{i+1}^{n}-u_{i}^n)}.
\end{equation}
 Using the above incremental form, we have the following result which will be proved in \textbf{Lemma~\ref{cd}}.
 \begin{lemma}[Bounds on $g(u^n)$]\label{lemma2}
 Under the CFL-like condition \eqref{cfl},
 $C^{n}_{i-\frac{1}{2}} ,\:D^{n}_{i+\frac{1}{2}}\ge0$ and $C^{n}_{i-\frac{1}{2}}+D^{n}_{i+\frac{1}{2}}\le1$. Also, if 
 \[g^{-1}(m)\rho_i^n\le w_i^n\le g^{-1}(M)\rho_i^n,\:m=\min_i g(u_{i}^0),\:M=\max_i g(u_{i}^0), \]
 then \[g^{-1}(m)\rho_i^{n+1}\le w_i^{n+1}\le g^{-1}(M)\rho_i^{n+1}.\]\end{lemma}
The following lemma is an easy consequence of the above lemmas and establish conservative property of $U$.
 \begin{lemma}($\L1$ Stability of $\rho,\:w,\:wu, wg(u)$):\label{theorem2}
 For for any time $t^n,\:n\ge0$, the following holds true:
\[\norma{z^n}_{\L1 (\R\times \R^+)}\leq K_z\norma{z^0}_{\L1 (\R\times \R^+)}, \] where 
$z$ can be $\rho,\:w,\:wu,\:w{g(u)}$ and the constant $K_z$ depends on $\norma{u_0}_{\L\infty(\R\times\R^+)}$.
 \end{lemma} 
 \begin{theorem}[Existence of Weak Solution]\label{thm1}
 For every $\phi\in C^{\infty}_{c}(\R \times \R^+), $ and for $z=\rho,\:w$, we have \begin{equation*}\begin{aligned} \lim\limits_{h \rightarrow 0}\displaystyle\int\limits_{\R \times\R^+ }\displaystyle z_h \phi_t + z_h g(u_h)\phi_x =0,
\end{aligned}\end{equation*}
 \end{theorem}
 \begin{proof}
 Let $i\in\mathbb{Z},\:n\in\mathbb{N},\:\phi_i^n=\phi(x_i,\:t_n)$ and $$\phi_{x,\:i}^n=\displaystyle\frac{\phi(x_i,\:t^n)-\phi(x_{i-1},\:t^n)}{h},\:\forall\,\, i,\:n.$$To prove theorem, it is enough to show that for $z=\rho,\:w, $
\begin{equation*}
\lim\limits_{h \rightarrow 0} A(h)=0 \, \, \, \mbox{with}\, \, \, A(h)= -h\sum\limits_{i\in\mathbb{Z},\:n\in\mathbb{N}} \left[z_i^{n+1} - z_i^{n} + \lambda \Big(z_i^ng(u_i^n)-z_{i-1}^ng(u_{i-1}^n)\Big)\right]\phi_i^n,
\end{equation*}
By the definition of the scheme, 
\begin{eqnarray*}
 A(h)= -h \sum\limits_{i,\:n} \left[ -\lambda(F_{i+\frac{1}{2}}^{z,\:n} - F_{i-\frac{1}{2}}^{z,\:n}) + \lambda\Big (z_i^ng(u_i^n)-z_{i-1}^ng(u_{i-1}^n)\Big)\right] \phi_i^n,\end{eqnarray*}
which, on rearranging the terms, gives, 
\begin{eqnarray*} A(h)=\lambda h \sum\limits_{i,\:n} \left[ \Big(F_{i+\frac{1}{2}}^{z,\:n} - z_i^ng(u_i^n)\Big)-\Big( F_{i-\frac{1}{2}}^{z,\:n}-z_{i-1}^ng(u_{i-1}^n)\Big)\right] \phi_i^n.
\end{eqnarray*} Now, applying summation by parts, we get
\begin{equation}\label{df1}
 A(h)= -\lambda h^2 \sum\limits_{i,\:n} \left[ \displaystyle F_{i-\frac{1}{2}}^{z,\:n} - z_{i-1}^ng(u_{i-1}^n)\right]\phi_{x,\:i}^n,\:
\end{equation}
where $F_{i-\frac{1}{2}}^{\rho,\:n}$ and $F_{i-\frac{1}{2}}^{w,\:n}$ are given by \eqref{FrhoN} and \eqref{FwN} respectively. Note that the possible expressions for $F_{i-\frac{1}{2}}^{z,\:n}$ are $z_{i-1}^{n}g(u_{i-1}^n),\:
z_{i}^{n} g(u_{i}^n)$ and $0$. Let $j\in\mathbb{Z}$. For a fixed $n$,the terms containing $g(u_j^n)$ in the expression \eqref{df1}  are
\begin{eqnarray*}
-\lambda h^2z_j^ng(u_j^n)\left[ \displaystyle \phi_{x,\:{j}}^n-\phi_{x,\:{j}+1}^n \right] 
\le L_\phi\lambda h^2 \norma{z_0 g(u_{0})}_{\L1}, L_\phi=2\norma{\phi_{x}}_{\L\infty(\R\times\R^+)}
\end{eqnarray*}
using the fact that the scheme is a $3-$ point scheme and $zg(u)$ is  $L^1$ stable. Summing over $n\in \mathbb{N},$ we get
\begin{equation*}
 -\lambda h^2 \sum\limits_{i,\:n} \left[ F_{i-\frac{1}{2}}^{z,\:n} - z_{i-1}^ng(u_{i-1}^n) \right]\phi_{x,\:i}^n\le \sum\limits_{n}L_\phi\lambda h^2 \norma{z_0g(u_0)}_{\L1}=L_\phi T h \norma{z_0 g(u_0)}_{\L1}, 
\end{equation*}
where $T$ is the final time. This shows that $A(h)=\mathcal{O}(h)$, which proves the claim.
 \end{proof}
 The above theorem shows that the distribution limit $U$ of the approximate solution $U_h$ is a solution of \textbf{(GPGD)} in the sense of distributions. It has been pointed in the seminal paper of \cite{bouchut1994zero} that the solutions of the system \textbf{(PGD)} satisfy the following inequality for any convex real--valued function $S$
 \begin{equation*}
 \Big(\rho S(u)\Big)_t + \Big(\rho u S(u)\Big)_x\le 0, 
 \end{equation*}
which can be extended to \textbf{(GPGD)} and reads as: \begin{eqnarray*}
\left( \rho S(u) \right)_t+ \left(\rho g(u) S(u)\right)_x \leq 0.
 \end{eqnarray*}
%
In the following theorem, we show that the numerical solution $U_i^n$ satisfies the discrete form of the above inequality.
\begin{theorem}[Discrete Entropy Inequality]\label{main}
 For every $\phi \in C^{\infty}_{c}(\R \times \R^+), $ and for $z=\rho S(u)$, we have \begin{equation*}\begin{aligned} \lim\limits_{h \rightarrow 0}\displaystyle\int\limits_{\R \times\R^+ }\displaystyle z_h \phi_t + z_h g(u_h) \phi_x \ge 0.
\end{aligned}\end{equation*}
 \end{theorem}
 For proving theorem, we first prove the following lemma.

\begin{lemma}
 For all convex functions $S$, 
 \begin{equation}\label{en}
 \rho_i^{n+1}S(u_i^{n+1})-\rho_i^n S(u_i^{n})+\lambda\Big(G(\rho^n_i,\:\rho^n_{i+1},\:u_i^{n},\:u_{i+1}^{n})-G(\rho^n_{i-1}, \rho^n_{i},\:u_{i-1}^{n},\:u_{i}^{n})\Big)\le0,
 \end{equation}
 where $G(\cdot, \cdot, \cdot)$ is the numerical entropy flux, consistent in the following sense $$G(\rho,\:\rho,\:u,\:u)= S(u)\rho g(u).$$
 \end{lemma}
 \begin{proof}
 Step 1:~\underline{We derive an equivalent formulation of \eqref{en}.} \\
Using \textbf{Lemma~\ref{lemma2}}, 
the convexity of $S$ and Jensen's inequality on the incremental form \eqref{incre}, one gets, 
\[S(u_i^{n+
1}) \le S(u_i^n)(1-\displaystyle C^{n}_{i-\frac{1}{2}}-\displaystyle D^{n}_{i+\frac{1}{2}})+S(u_{i-1}^n) \displaystyle C^{n}_{i-\frac{1}{2}}+S(u_{i+1}^{n})\displaystyle D^{n}_{i+\frac{1}{2}},\:\]
which implies
\begin{eqnarray*}S(u_i^{n+
1})-S(u_i^n)\le -\Big(S(u_{i}^n)-S(u_{i-1}^n)\Big)\displaystyle C^{n}_{i-\frac{1}{2}}+\Big(S(u_{i+1}^n)-S(u_i^n)\Big)\displaystyle D^{n}_{i+\frac{1}{2}}.
\end{eqnarray*}
Substituting the values of $\displaystyle C^{n}_{i-\frac{1}{2}}$ and $\displaystyle D^{n}_{i+\frac{1}{2}},\:$ we have, 
\begin{eqnarray*}
\rho_i^{n+1}S(u_i^{n+
1})-\rho_i^{n+1}S(u_i^n)&\le -\lambda\Big(S(u_{i-1}^n)-S(u_{i}^n)\Big)\displaystyle\frac{F^{w,\:n}_{i-\frac{1}{2}}-F^{\rho,\:n}_{i-\frac{1}{2}}u_i^{n}}{(u_{i}^{n}-u_{i-1}^n)}\\ &\quad-\lambda\Big(S(u_{i+1}^n)-S(u_i^n)\Big)\displaystyle\frac{F^{w,\:n}_{i+\frac{1}{2}}-F^{\rho,\:n}_{i+\frac{1}{2}}u_i^{n}}{(u_{i+1}^{n}-u_{i}^n)}.
\end{eqnarray*}
Step 2:~\underline{The inequality \eqref{en} and the last inequality are equivalent.}\\
Putting the value of $\rho_i^{n+1}, $ in the last inequality, we have, \[-\lambda\Big(S(u_{i-1}^n)-S(u_i^n)\Big)\frac{F^{w,\:n}_{i-\frac{1}{2}}-F^{\rho,\:n}_{i-\frac{1}{2}}u_i^{n}}{(u_{i}^{n}-u_{i-1}^n)}-\lambda\Big(S(u_{i+1}^n)-S(u_i^n)\Big)\frac{F^{w,\:n}_{i+\frac{1}{2}}-F^{\rho,\:n}_{i+\frac{1}{2}}u_i^{n}}{(u_{i+1}^{n}-u_{i}^n)}\]
\begin{eqnarray*}\ge\rho_i^{n+1}S(u_i^{n+
1})-\Big(\rho_i^{n}-\lambda(F^{\rho,\:n}_{i+\frac{1}{2}}-F^{\rho,\:n}_{i-\frac{1}{2}})\Big)S(u_i^n), 
\end{eqnarray*}
which on further rearrangement, 
\begin{eqnarray*}
\rho_i^{n+1}S(u_i^{n+
1})-\rho_i^{n}S(u_i^n) &\le -\lambda\Big(F^{\rho,\:n}_{i+\frac{1}{2}}-F^{\rho,\:n}_{i-\frac{1}{2}}\Big)S(u_i^n)+\lambda\Big(S(u_{i}^n)-S(u_{i-1}^n)\Big)\displaystyle\frac{F^{w,\:n}_{i-\frac{1}{2}}-F^{\rho,\:n}_{i-\frac{1}{2}}u_i^{n}}{(u_{i}^{n}-u_{i-1}^n)}\\
&-\lambda\Big(S(u_{i+1}^n)-S(u_i^n)\Big)\displaystyle\frac{F^{w,\:n}_{i+\frac{1}{2}}-F^{\rho,\:n}_{i+\frac{1}{2}}u_i^{n}}{(u_{i+1}^{n}-u_{i}^n)}, \\
&= \lambda\Bigg(F^{\rho,\:n}_{i-\frac{1}{2}}S(u_i^n)+\Big(S(u_{i}^n)-S(u_{i-1}^n)\Big)\displaystyle\frac{F^{w,\:n}_{i-\frac{1}{2}}-F^{\rho,\:n}_{i-\frac{1}{2}}u_i^{n}}{(u_{i}^{n}-u_{i-1}^n)}\Bigg)\\&\quad\quad - \lambda\Bigg(F^{\rho,\:n}_{i+\frac{1}{2}}S(u_i^n)+\Big(S(u_{i+1}^n)-S(u_i^n)\Big)\displaystyle\frac{F^{w,\:n}_{i+\frac{1}{2}}-F^{\rho,\:n}_{i+\frac{1}{2}}u_i^{n}}{(u_{i+1}^{n}-u_{i}^n)}\Bigg).
\end{eqnarray*}
Now, the term \[F^{\rho,\:n}_{i+\frac{1}{2}}S(u_i^n)+\left(\Big(S(u_{i+1}^n)-S(u_i^n)\Big)\frac{F^{w,\:n}_{i+\frac{1}{2}}-F^{\rho,\:n}_{i+\frac{1}{2}}u_i^{n}}{(u_{i+1}^{n}-u_{i}^n)}\right)\]
\begin{eqnarray*}
&=&\frac{(u_{i+1}^{n}-u_{i}^n)F^{\rho,\:n}_{i+\frac{1}{2}}S(u_i^n)+(S(u_{i+1}^n)-S(u_i^n))(F^{w,\:n}_{i+\frac{1}{2}}-F^{\rho,\:n}_{i+\frac{1}{2}}u_i^{n})}{(u_{i+1}^{n}-u_{i}^n)}\\
&=&\frac{u_{i+1}^{n}F^{\rho,\:n}_{i+\frac{1}{2}}S(u_i^n)+S(u_{i+1}^n)(F^{w,\:n}_{i+\frac{1}{2}}-F^{\rho,\:n}_{i+\frac{1}{2}}u_i^{n})-S(u_{i}^n)F^{w,\:n}_{i+\frac{1}{2}}}{(u_{i+1}^{n}-u_{i}^n)}\\
&=&\frac{(
u_{i+1}^{n}F^{\rho,\:n}_{i+\frac{1}{2}}-F^{w,\:n}_{i+\frac{1}{2}})S(u_i^n)+S(u_{i+1}^n)(F^{w,\:n}_{i+\frac{1}{2}}-F^{\rho,\:n}_{i+\frac{1}{2}}u_i^{n})}{(u_{i+1}^{n}-u_{i}^n)}.
\end{eqnarray*}
Similarly, consider the term 
\[F^{\rho,\:n}_{i-\frac{1}{2}}S(u_i^n)+\Big(S(u_{i}^n)-S(u_{i-1}^n)\Big)\frac{F^{w,\:n}_{i-\frac{1}{2}}-F^{\rho,\:n}_{i-\frac{1}{2}}u_i^{n}}{(u_{i}^{n}-u_{i-1}^n)}\]
\begin{eqnarray*}
&=& \frac{(u_{i}^{n}-u_{i-1}^n)F^{\rho,\:n}_{i-\frac{1}{2}}S(u_i^n)+\Big(S(u_{i}^n)-S(u_{i-1}^n)\Big)(F^{w,\:n}_{i-\frac{1}{2}}-F^{\rho,\:n}_{i-\frac{1}{2}}u_i^{n})}{(u_{i}^{n}-u_{i-1}^n)}\\
&=& \frac{(u_{i}^{n}-u_{i-1}^n)F^{\rho,\:n}_{i-\frac{1}{2}}S(u_i^n)+S(u_{i}^n)(F^{w,\:n}_{i-\frac{1}{2}}-F^{\rho,\:n}_{i-\frac{1}{2}}u_i^{n})-S(u_{i-1}^n) (F^{w,\:n}_{i-\frac{1}{2}}-F^{\rho,\:n}_{i-\frac{1}{2}}u_i^{n})}{(u_{i}^{n}-u_{i-1}^n)}\\
&=& \frac{(F^{w,\:n}_{i-\frac{1}{2}}-u_{i-1}^n F^{\rho,\:n}_{i-\frac{1}{2}})S(u_i^n)-S(u_{i-1}^n) (F^{w,\:n}_{i-\frac{1}{2}}-F^{\rho,\:n}_{i-\frac{1}{2}}u_i^{n})}{(u_{i}^{n}-u_{i-1}^n)}, 
\end{eqnarray*}
which implies that
\begin{equation}\label{dS}
 \rho_i^{n+1} S(u_i^{n+1})-\rho_i^n S(u_i^{n})+\lambda\Big(F^{\rho S(u),\:n}_{i+\frac{1}{2}}-F^{\rho S(u),\:n}_{i-\frac{1}{2}}\Big)\le0,\:
\end{equation}
where 
\begin{equation}\label{dSG}
 F^{\rho S(u),\:n}_{i+\frac{1}{2}}=\frac{(
u_{i+1}^{n}F^{\rho,\:n}_{i+\frac{1}{2}}-F^{w,\:n}_{i+\frac{1}{2}})S(u_i^n)+S(u_{i+1}^n)(F^{w,\:n}_{i+\frac{1}{2}}-F^{\rho,\:n}_{i+\frac{1}{2}}u_i^{n})}{(u_{i+1}^{n}-u_{i}^n)}.
\end{equation}
To prove the consistency of the flux $F^{\rho S(u),\:n}_{i+\frac{1}{2}}$, we assume that both $u_i^n$ and $u_{i+1}^n$ have the same sign, say, $u_i^n,\:u_{i+1}^n>0$ and are not equal, then, we have, 
\begin{eqnarray*}
 F^{\rho S(u),\:n}_{i+\frac{1}{2}}&=&\frac{(
u_{i+1}^n\rho_{i}^{n}g(u_i^n)-w_i^ng(u_i^n)u_i^n)S(u_i^n)+S(u_{i+1}^n)(g(u_i^n)u_i^n-g(u_i^n)u_i^{n})}{(u_{i+1}^{n}-u_{i}^n)}\\
&=&\frac{(
u_{i+1}^{n}-u_i^n)\rho_i^ng(u_i^n)S(u_i^n)}{(u_{i+1}^{n}-u_{i}^n)}, 
\end{eqnarray*}
which gives the consistency of $F^{\rho S(u)}.$ 
This indicates to choose $G= F^{\rho S(u)}$ in the equation \eqref{en}.\end{proof}
 Finally, we give the proof of the main theorem, \textbf{Theorem~\ref{main}} which says that the numerical solution $U_i^n$ satisfies the discrete form of entropy inequality.
 \begin{proof}
 Let $i\in\mathbb{Z},\:n\in\mathbb{N},\:\phi_i^n=\phi(x_i,\:t_n)$ and $\phi_{x,\:i}^n=\displaystyle\frac{\phi(x_i,\:t^n)-\phi(x_{i-1}, t^n)}{h},\:\forall\,\,\, i,\:n.$ To prove theorem, it is enough to show that, for $z=\rho S(u), $
\begin{equation*}
\lim \limits_{h\rightarrow0} A(h)\ge 0, \, \, \, \mbox{where}, \, \, A(h)= -h\sum\limits_{i,n} \left[ z_i^{n+1} - z_i^{n} + \lambda \Big(z_i^n g(u_i^n)-z_{i-1}^n g(u_{i-1}^n)\Big)\right] \phi_{i}^n.\end{equation*}
By the equations \eqref{dS} and \eqref{dSG}, we have
\begin{equation*}
 \rho_i^{n+1} S(u_i^{n+1})-\rho_i^n S(u_i^{n})+\lambda\Big(F^{\rho S(u),\:n}_{i+\frac{1}{2}}-F^{\rho S(u),\:n}_{i-\frac{1}{2}}\Big)\le0,\:
\end{equation*}
 where $F^{\rho S(u),\:n}_{i+\frac{1}{2}}$ is given by \eqref{dSG} and can be rewritten as 
\begin{equation*}\displaystyle F^{\rho S(u),\:n}_{i+\frac{1}{2}}=
{ \left\{\begin{array}{ccl}
 S(u^n_i)\rho_i^n g(u_{i}^n) & \, \mbox{if}\, & g(u_{i}^n) \geq 0,\:g(u_{i+1}^n)>0,\:\\[.2mm]
S(u^n_{i+1})\rho_{i+1}^n g(u_{i+1}^n) & \, \mbox{if}\, & g(u_{i}^n)<0,\:g(u_{i+1}^n) \leq0,\:\\[.2mm]
0 & \, \mbox{if}\, & g(u_{i}^n)<0,\:g(u_{i+1}^n)>0,\:\\[.2mm]
 \displaystyle F^{w,\:n}_{i+\frac{1}{2}}\frac{S(u^n_{i+1})-S(u^n_{i})}{ {{u}}_{i+1}^n - {{u}}_i^n } & \, \mbox{if}\, & g(u_{i}^n) \geq 0,\:g(u_{i+1}^n) \leq 0.\\[.2mm]
\end{array}\right.}.
\end{equation*}
This implies that with $z=\rho S(u), $
\begin{equation*}
 -(z_i^{n+1}-z_i^n)\ge \lambda\Big(F^{z,\:n}_{i+\frac{1}{2}}-F^{z,\:n}_{i-\frac{1}{2}}\Big).
\end{equation*}
This further implies 
\begin{eqnarray*}
 A(h)\ge h \sum\limits_{i,\:n} \left[ \lambda\Big(F^{z,\:n}_{i+\frac{1}{2}}-F^{z,\:n}_{i-\frac{1}{2}}\Big) - \lambda \Big(z_i^ng(u_i^n)-z_{i-1}^ng(u_{i-1}^n)\Big)\right] \phi_i^n,,\end{eqnarray*}
which, on rearranging the terms, gives, 
\begin{eqnarray*} A(h)\ge\lambda h \sum\limits_{i,\:n} \left[ \Big(F^{z,\:n}_{i+\frac{1}{2}} - z_i^ng(u_i^n)\Big)-\Big( F^{z,\:n}_{i-\frac{1}{2}}-z_{i-1}^ng(u_{i-1}^n)\Big)\right] \phi_i^n.
\end{eqnarray*} Now, applying summation by parts, we get
\begin{equation*}
 A(h)\ge -\lambda h^2 \sum\limits_{i,\:n} \left[ \displaystyle F^{z,\:n}_{i-\frac{1}{2}} - z_{i-1}^ng(u_{i-1}^n)\right]\phi_{x,\:i}^n.
\end{equation*}
Now, the proof follows by similar argument as in \textbf{Theorem~\ref{thm1}} and by using, in addition, the Lipschitz continuity of the function $S$.
\end{proof}

We now prove that the solutions of the higher order scheme also preserve the physical properties of the system. Denote \[\hat{F}^z\left(U_{i\pm\frac{1}{2}}^{L},\:U_{i\pm\frac{1}{2}}^{R}\right)=\hat{F}^z_{i\pm\frac{1}{2}},\:z=\rho,\:w.\]
\begin{lemma}[Positivity] Under the CFL-like condition, 
\begin{equation}\label{cfl2}
 \lambda\max_{i,\:n}|g(\overline{u}_{i}^n)|\leq \mbox{cfl}, 
 \end{equation}
 if $\rho_{i-\frac{1}{2}}^{R},\:\rho_{i+\frac{1}{2}}^{L}>0$ and $\mbox{cfl}=\frac{1}{2}$, then $\overline{\rho}_i^{n+1}>0$.
 \end{lemma}
\begin{proof}
\begin{eqnarray*}
 \overline{\rho}_i^{n+1}&= &H(\rho^L_{i-\frac{1}{2}},\:\rho^R_{i-\frac{1}{2}},\:\rho^L_{i+\frac{1}{2}},\:\rho^R_{i+\frac{1}{2}},\:u^L_{i-\frac{1}{2}},\:u^R_{i-\frac{1}{2}},\:u^L_{i+\frac{1}{2}},\:u^R_{i+\frac{1}{2}})\\&=&\overline{\rho}_i^n-\lambda\left[\hat{F}^\rho_{i+\frac{1}{2}}-\hat{F}^\rho_{i-\frac{1}{2}}\right]=\frac{1}{2}\left(\rho_{i-\frac{1}{2}}^{R}+ \rho_{i+\frac{1}{2}}^{L}\right)-\lambda\left[\hat{F}^\rho_{i+\frac{1}{2}}-\hat{F}^\rho_{i-\frac{1}{2}} \right]
 \end{eqnarray*}
Repeating the arguments in \textbf{Lemma~\ref{lemma1}}, $H$ is non decreasing in its first four arguments under the CFL-like condition \eqref{cfl2}. Since $H(0, 0, 0, 0, \cdot, \cdot, \cdot, \cdot)=0,\:$ the result follows.
\end{proof}
The same result is true for the first-order scheme with $\mbox{cfl}=1, $ see \textbf{Lemma.~\ref{lemma1}}. $\rho_{i\pm\frac{1}{2}}^{L,\:R}$ can now be made positive by appropriate choice of $\sigma_i^\rho$.
\begin{lemma}[Bounds on Momentum]\label{cd} Under the CFL-like condition \eqref{cfl2} with $\mbox{cfl}=\frac{1}{3}$, 
 if 
 \[g^{-1}(m)\overline{\rho}_i^n\le \overline{w}_i^n\le g^{-1}(M)\overline{\rho}_i^n,\:\, m=\min_i g(\overline{u}_{i}^0), \, M=\max_i g(\overline{u}_{i}^0), \]
 then \[g^{-1}(m)\overline{\rho}_i^{n+1}\le \overline{w}_i^{n+1}\le g^{-1}(M)\overline{\rho}_i^{n+1}. \]\end{lemma}
 \begin{proof}
The strategy to prove the lemma  is to write the finite volume scheme \eqref{seco} in an incremental form for $\overline{u}_i^{n+1}$. For simplicity, we show the results for Minmod limiter.
\begin{enumerate}[label=Case \textbf{\arabic*}.]
 \item We start with the case when neither $u^{L}_{i+\frac{1}{2}}\ge0$ and $u^{R}_{i+\frac{1}{2}}\le0$ nor $u^{L}_{i-\frac{1}{2}}\ge0$ and $u^{R}_{i-\frac{1}{2}}\le0$. Then, we have
 \[\hat{F}^w_{i+\frac{1}{2}}= u^{L}_{i+\frac{1}{2}}\max(\hat{F}^\rho_{i+\frac{1}{2}},\:0)+u^{R}_{i+\frac{1}{2}}\min(\hat{F}^\rho_{i+\frac{1}{2}},\:0).\] Using the finite volume scheme \eqref{seco}, we get
\begin{eqnarray*}
 \overline{\rho}_i^{n+1}&=&\overline{\rho}_i^n-\lambda\left[\hat{F}^\rho_{i+\frac{1}{2}}-\hat{F}^\rho_{i-\frac{1}{2}}\right],\:\\
 \overline{w}_i^{n+1}=\overline{\rho}_i^{n+1}\overline{u}_i^{n+1}&=&\overline{u}_i^n\overline{\rho}_i^{n}-\lambda\left[\hat{F}^w_{i+\frac{1}{2}}-\hat{F}^w_{i-\frac{1}{2}}\right].
\end{eqnarray*}
Multiplying first equation by $\overline{u}_i^n$, we get, 
\begin{eqnarray}\label{incre1}
 \overline{\rho}_i^{n+1}(\overline{u}_i^{n+1}- \overline{u}_i^{n})&=&\lambda\left[\hat{F}^w_{i-\frac{1}{2}}-\overline{u}_i^{n}\hat{F}^\rho_{i-\frac{1}{2}}\right]-\lambda\left[\hat{F}^w_{i+\frac{1}{2}}-\overline{u}_i^{n}\hat{F}^\rho_{i+\frac{1}{2}}\right].
\end{eqnarray}
Putting the values of $u^{L}_{i\pm\frac{1}{2}}$ and $u^{R}_{i\pm\frac{1}{2}}$ and $\hat{F}^w_{i\pm\frac{1}{2}}$, $ \overline{\rho}_i^{n+1}(\overline{u}_i^{n+1}- \overline{u}_i^{n})$ is equal to
 \begin{eqnarray*}
 &&\lambda\left[u^{L}_{i-\frac{1}{2}}\max(\hat{F}^\rho_{i-\frac{1}{2}},\:0)+u^{R}_{i-\frac{1}{2}}\min(\hat{F}^\rho_{i-\frac{1}{2}},\:0)-\overline{u}_i^{n}\Big(\max(\hat{F}^\rho_{i-\frac{1}{2}},\:0)+
\min(\hat{F}^\rho_{i-\frac{1}{2}},\:0\Big)\right]\\
&&-\lambda\left[u^{L}_{i+\frac{1}{2}}\max(\hat{F}^\rho_{i+\frac{1}{2}},\:0)+u^{R}_{i+\frac{1}{2}}\min(\hat{F}^\rho_{i+\frac{1}{2}},\:0)-\overline{u}_i^{n}\Big(\max(\hat{F}^\rho_{i+\frac{1}{2}},\:0)+
\min(\hat{F}^\rho_{i+\frac{1}{2}},\:0\Big)\right]\\
&=&\lambda\left[(u^{L}_{i-\frac{1}{2}}-\overline{u}_i^{n})\max(\hat{F}^\rho_{i-\frac{1}{2}},\:0)+(u^{R}_{i-\frac{1}{2}}-\overline{u}_i^{n})\min(\hat{F}^\rho_{i-\frac{1}{2}},\:0)\right]\\
&&-\lambda\left[(u^{L}_{i+\frac{1}{2}}-\overline{u}_i^{n})\max(\hat{F}^\rho_{i+\frac{1}{2}},\:0)+(u^{R}_{i+\frac{1}{2}}-\overline{u}_i^{n})\min(\hat{F}^\rho_{i+\frac{1}{2}},\:0)\right].
\end{eqnarray*}
 Note that 
\[u^{L}_{i+\frac{1}{2}}=p_{i}^u(x_{i+\frac{1}{2}})=\overline{u}_i^{n}+{\overline{\sigma}}_i^u,\:u^{R}_{i-\frac{1}{2}}=p_{i}^u(x_{i-\frac{1}{2}})=\overline{u}_{i}^{n}-{\overline{\sigma}}_i^u, \]
where 
${\overline{\sigma}}_i^u:=\displaystyle\frac{h}{2}\sigma_i^u=\frac{1}{2}\text{minmod}(\overline{u}_i^{n}-\overline{u}_{i-1}^{n},\:\overline{u}_{i+1}^{n}-\overline{u}_{i}^{n}).$
Hence, we have 
\begin{eqnarray*}
 \overline{\rho}_i^{n+1}(\overline{u}_i^{n+1}- \overline{u}_i^{n})&=&-\lambda\left[(-\overline{u}_{i-1}^{n}+\overline{u}_i^{n}-{\overline{\sigma}}_{i-1}^u)\max(\hat{F}^\rho_{i-\frac{1}{2}},\:0)+{\overline{\sigma}}_i^u\min(\hat{F}^\rho_{i-\frac{1}{2}},\:0)\right]\\&&-\lambda\left[{\overline{\sigma}}_i^u\max(\hat{F}^\rho_{i+\frac{1}{2}},\:0)+(\overline{u}_{i+1}^{n}-\overline{u}_i^{n}-{\overline{\sigma}}_{i+1}^u)\min(\hat{F}^\rho_{i+\frac{1}{2}},\:0)\right]\end{eqnarray*}
 which implies that
 \begin{eqnarray*}
 \overline{u}_i^{n+1}&=&
\overline{u}_i^{n}-\lambda\frac{(\overline{u}_{i}^{n}-\overline{u}_{i-1}^{n})}{\overline{\rho}_i^{n+1}}\left[\Big(1-\frac{{\overline{\sigma}}_{i-1}^u}{\overline{u}_{i}^{n}-\overline{u}_{i-1}^{n}}\Big)\max(\hat{F}^\rho_{i-\frac{1}{2}},\:0)\right]-\lambda\frac{{\overline{\sigma}}_i^u\max(\hat{F}^\rho_{i+\frac{1}{2}},\:0)}{\overline{\rho}_i^{n+1}}\\
 &&-\lambda\frac{(\overline{u}_{i+1}^{n}-\overline{u}_{i}^{n})}{\overline{\rho}_i^{n+1}}\left[\Big(1-\frac{{\overline{\sigma}}_{i+1}^u}{\overline{u}_{i+1}^{n}-\overline{u}_{i}^{n}}\Big)\min(\hat{F}^\rho_{i+\frac{1}{2}},\:0)\right]-\lambda\frac{{\overline{\sigma}}_i^u\min(\hat{F}^\rho_{i-\frac{1}{2}},\:0)}{\overline{\rho}_i^{n+1}}\\
&=&\overline{u}_i^{n}-(\overline{u}_{i}^{n}-\overline{u}_{i-1}^{n})\frac{\lambda}{\overline{\rho}_i^{n+1}}\left[\Big(1-\frac{{\overline{\sigma}}_{i-1}^u}{\overline{u}_{i}^{n}-\overline{u}_{i-1}^{n}}\Big)\max(\hat{F}^\rho_{i-\frac{1}{2}},\:0)+\frac{{\overline{\sigma}}_i^u\max(\hat{F}^\rho_{i+\frac{1}{2}},\:0)}{(\overline{u}_{i}^{n}-\overline{u}_{i-1}^{n})}\right]\\
 &&-(\overline{u}_{i+1}^{n}-\overline{u}_i^{n})\frac{\lambda}{\overline{\rho}_i^{n+1}}\left[\Big(1-\frac{{\overline{\sigma}}_{i+1}^u}{\overline{u}_{i+1}^{n}-\overline{u}_i^{n}}\Big)\min(\hat{F}^\rho_{i+\frac{1}{2}},\:0)+
 \frac{{\overline{\sigma}}_i^u\min(\hat{F}^\rho_{i-\frac{1}{2}},\:0)}{(\overline{u}_{i+1}^{n}-\overline{u}_{i}^{n})}\right]\\
 &=&\overline{u}_i^{n}-(\overline{u}_{i}^{n}-\overline{u}_{i-1}^{n})\tilde{C}_{i-\frac{1}{2}}+(\overline{u}_{i+1}^{n}-\overline{u}_i^{n})\tilde{D}_{i+\frac{1}{2}},\:
\end{eqnarray*}
and hence
\begin{eqnarray}\label{incre3}
\overline{u}_i^{n+1} =\overline{u}_i^{n}(1-\tilde{C}_{i-\frac{1}{2}}-\tilde{D}_{i+\frac{1}{2}})+\overline{u}_{i-1}^{n})\tilde{C}_{i-\frac{1}{2}}+\overline{u}_{i+1}^{n}\tilde{D}_{i+\frac{1}{2}}
\end{eqnarray}
where \[\tilde{D}_{i+\frac{1}{2}}=-\frac{\lambda}{\overline{\rho}_i^{n+1}}\left[\Big(1-\frac{{\overline{\sigma}}_{i+1}^u}{\overline{u}_{i+1}^{n}-\overline{u}_i^{n}}\Big)\min(\hat{F}^\rho_{i+\frac{1}{2}},\:0)+\frac{{\overline{\sigma}}_i^u\min(\hat{F}^\rho_{i-\frac{1}{2}},\:0)}{(\overline{u}_{i+1}^{n}-\overline{u}_{i}^{n})}\right], \]
\[\tilde{C}_{i-\frac{1}{2}}=\frac{\lambda}{\overline{\rho}_i^{n+1}}\left[\Big(1-\frac{{\overline{\sigma}}_{i-1}^u}{\overline{u}_{i}^{n}-\overline{u}_{i-1}^{n}}\Big)\max(\hat{F}^\rho_{i-\frac{1}{2}},\:0)+\frac{{\overline{\sigma}}_i^u\max(\hat{F}^\rho_{i+\frac{1}{2}},\:0)}{(\overline{u}_{i}^{n}-\overline{u}_{i-1}^{n})}\right]\]
Since for any $k$, $0\le\displaystyle\frac{{\overline{\sigma}}_k^u}{(\overline{u}_{k}^{n}-\overline{u}_{k-1}^{n})},\:\ \displaystyle\frac{{\overline{\sigma}}_k^u}{(\overline{u}_{k+1}^{n}-\overline{u}_{k}^{n})}\le \frac{1}{2}$, hence, $\tilde{C}_{i-\frac{1}{2}},\:\tilde{D}_{i+\frac{1}{2}}\ge0$. We now prove that $\overline{\rho}_i^{n+1}(1-\tilde{C}_{i-\frac{1}{2}}-\tilde{D}_{i+\frac{1}{2}})\ge0$. 
We introduce some notations:
\[u^{L ,\:+}_{i\pm\frac{1}{2}}:=\max(u^{L}_{i\pm\frac{1}{2}} ,\:0),\:u^{R,\:-}_{i\pm\frac{1}{2}}:=\min(u^{R}_{i\pm\frac{1}{2}} ,\:0).\]
\begin{enumerate}[\textbf{(\roman*)}]
    \item When $\hat{F}^\rho_{i-\frac{1}{2}},\hat{F}^\rho_{i+\frac{1}{2}}\ge0 :$\\
    Then, $\hat{F}^\rho_{i\pm\frac{1}{2}}=u^{L ,\:+}_{i\pm\frac{1}{2}}\rho^L_{i\pm\frac{1}{2}}.$
    We have $\tilde{D}_{i+\frac{1}{2}}=0$ and
    \[\tilde{C}_{i-\frac{1}{2}}=\frac{\lambda}{\overline{\rho}_i^{n+1}}\left[\Big(1-\frac{{\overline{\sigma}}_{i-1}^u}{\overline{u}_{i}^{n}-\overline{u}_{i-1}^{n}}\Big)\max(\hat{F}^\rho_{i-\frac{1}{2}},\:0)+\frac{{\overline{\sigma}}_i^u\max(\hat{F}^\rho_{i+\frac{1}{2}},\:0)}{(\overline{u}_{i}^{n}-\overline{u}_{i-1}^{n})}\right]\le \frac{\lambda}{\overline{\rho}_i^{n+1}}\left[\hat{F}^\rho_{i-\frac{1}{2}}+\frac{1}{2}\hat{F}^\rho_{i+\frac{1}{2}}\right],\]
    which implies that
   \begin{eqnarray*}
   \overline{\rho}_i^{n+1}(1- \tilde{C}_{i-\frac{1}{2}})&\ge& \overline{\rho}_i^{n+1}-\lambda \hat{F}^\rho_{i-\frac{1}{2}}-\frac{1}{2}\lambda\hat{F}^\rho_{i+\frac{1}{2}}\\
   &=&\overline{\rho}_i^n-\frac{3}{2}\lambda\hat{F}^\rho_{i+\frac{1}{2}}\\
   &=&\frac{1}{2}\rho^L_{i+\frac{1}{2}}\Big(1-3\lambda g(u^L_{i+\frac{1}{2}})\Big)+\frac{1}{2}\rho^R_{i-\frac{1}{2}}\ge0,
   \end{eqnarray*}
   under the CFL-like condition \eqref{cfl2} with cfl=$\frac{1}{3}$.
\item  When $\hat{F}^\rho_{i-\frac{1}{2}},\hat{F}^\rho_{i+\frac{1}{2}}\le0 :$\\
Then, $\hat{F}^\rho_{i\pm\frac{1}{2}}=u^{R ,\:-}_{i\pm\frac{1}{2}}\rho^R_{i\pm\frac{1}{2}}.$
    We have $\tilde{C}_{i-\frac{1}{2}}=0$ and
    \[\tilde{D}_{i+\frac{1}{2}}=-\frac{\lambda}{\overline{\rho}_i^{n+1}}\left[\Big(1-\frac{{\overline{\sigma}}_{i+1}^u}{\overline{u}_{i+1}^{n}-\overline{u}_i^{n}}\Big)\min(\hat{F}^\rho_{i+\frac{1}{2}},\:0)+\frac{{\overline{\sigma}}_i^u\min(\hat{F}^\rho_{i-\frac{1}{2}},\:0)}{(\overline{u}_{i+1}^{n}-\overline{u}_{i}^{n})}\right]\]
    The proof is similar to previous case.
    \item When $\hat{F}^\rho_{i-\frac{1}{2}}\ge0,\hat{F}^\rho_{i+\frac{1}{2}}\le0 :$\\
Then, $\hat{F}^\rho_{i-\frac{1}{2}}=u^{L ,\:+}_{i-\frac{1}{2}}\rho^L_{i-\frac{1}{2}},\hat{F}^\rho_{i+\frac{1}{2}}=u^{R ,\:-}_{i+\frac{1}{2}}\rho^R_{i+\frac{1}{2}}.$
    \[\tilde{D}_{i+\frac{1}{2}}=-\frac{\lambda}{\overline{\rho}_i^{n+1}}\Big(1-\frac{{\overline{\sigma}}_{i+1}^u}{\overline{u}_{i+1}^{n}-\overline{u}_i^{n}}\Big)\hat{F}^\rho_{i+\frac{1}{2}}\le -\frac{\lambda \hat{F}^\rho_{i+\frac{1}{2}}}{\overline{\rho}_i^{n+1}},\]
\[\tilde{C}_{i-\frac{1}{2}}=\frac{\lambda}{\overline{\rho}_i^{n+1}}\Big(1-\frac{{\overline{\sigma}}_{i-1}^u}{\overline{u}_{i}^{n}-\overline{u}_{i-1}^{n}}\Big)\hat{F}^\rho_{i-\frac{1}{2}}\le \frac{\lambda \hat{F}^\rho_{i-\frac{1}{2}}}{\overline{\rho}_i^{n+1}},\] which implies that
   \[\overline{\rho}_i^{n+1}(1- \tilde{C}_{i-\frac{1}{2}}-\tilde{D}_{i+\frac{1}{2}})\ge\overline{\rho}_i^{n+1}-\lambda \hat{F}^\rho_{i-\frac{1}{2}}+\lambda\hat{F}^\rho_{i+\frac{1}{2}}=\overline{\rho}_i^{n}\ge0\]
    \item When $\hat{F}^\rho_{i-\frac{1}{2}}\le0,\hat{F}^\rho_{i+\frac{1}{2}}\ge0 :$\\
Then, $\hat{F}^\rho_{i-\frac{1}{2}}=u^{R ,\:-}_{i-\frac{1}{2}}\rho^R_{i-\frac{1}{2}},\hat{F}^\rho_{i+\frac{1}{2}}=u^{L ,\:+}_{i+\frac{1}{2}}\rho^L_{i+\frac{1}{2}}.$
    We have 
    \[\tilde{D}_{i+\frac{1}{2}}=-\frac{\lambda}{\overline{\rho}_i^{n+1}}\frac{{\overline{\sigma}}_i^u\min(\hat{F}^\rho_{i-\frac{1}{2}},\:0)}{(\overline{u}_{i+1}^{n}-\overline{u}_{i}^{n})}\le -\frac{\lambda\hat{F}^\rho_{i-\frac{1}{2}}}{2\overline{\rho}_i^{n+1}},\]
    and \[ \tilde{C}_{i-\frac{1}{2}}=\frac{\lambda}{\overline{\rho}_i^{n+1}}\frac{{\overline{\sigma}}_i^u\max(\hat{F}^\rho_{i+\frac{1}{2}},\:0)}{(\overline{u}_{i}^{n}-\overline{u}_{i-1}^{n})}\le \frac{\lambda \hat{F}^\rho_{i+\frac{1}{2}}}{2\overline{\rho}_i^{n+1}}. \]
  Now, \begin{eqnarray*}
  \overline{\rho}_i^{n+1}(1- \tilde{C}_{i-\frac{1}{2}}-\tilde{D}_{i+\frac{1}{2}})&\ge&\overline{\rho}_i^{n+1}-\lambda \hat{F}^\rho_{i+\frac{1}{2}}+\lambda\hat{F}^\rho_{i-\frac{1}{2}}\\
  &=&\overline{\rho}_i^{n}-\frac{3}{2}\lambda \hat{F}^\rho_{i+\frac{1}{2}}+\frac{3}{2}\lambda\hat{F}^\rho_{i-\frac{1}{2}}\\
  &=&\frac{1}{2}\rho^L_{i+\frac{1}{2}}\Big(1-3\lambda u^L_{i+\frac{1}{2}}\Big)+\frac{1}{2}\rho^R_{i-\frac{1}{2}}\Big(1+3\lambda u^R_{i-\frac{1}{2}}\Big)\ge0
  \end{eqnarray*}
   under the CFL-like condition \eqref{cfl2} with cfl=$\frac{1}{3}$.
\end{enumerate}
\item For the remaining cases on $u^{L,\:R}_{i\pm\frac{1}{2}}, $ using \eqref{incre1}, we have
\begin{eqnarray}\label{incre4}
\overline{u}_i^{n+1} =\overline{u}_i^{n}(1-\tilde{C}_{i-\frac{1}{2}}-\tilde{D}_{i+\frac{1}{2}})+\overline{u}_{i-1}^{n})\tilde{C}_{i-\frac{1}{2}}+\overline{u}_{i+1}^{n}\tilde{D}_{i+\frac{1}{2}}
\end{eqnarray}
where 
\begin{eqnarray*}
 \displaystyle \tilde{C}_{i-\frac{1}{2}}=-\displaystyle\lambda\frac{\hat{F}^w_{i-\frac{1}{2}}-\hat{F}^\rho_{i-\frac{1}{2}}\overline{u}_i^{n}}{\overline{\rho}_i^{n+1}(\overline{u}_{i}^{n}-\overline{u}_{i-1}^n)},\:\quad \tilde{D}_{i+\frac{1}{2}}=-\displaystyle\lambda\frac{\hat{F}^w_{i+\frac{1}{2}}-\hat{F}^\rho_{i+\frac{1}{2}}\overline{u}_i^{n}}{\overline{\rho}_i^{n+1}(\overline{u}_{i+1}^{n}-\overline{u}_{i}^n)}.
\end{eqnarray*}
We prove for the case when $\overline{u}_i^n\ge\max(\overline{u}_{i-1}^n,\:
\overline{u}_{i+1}^n)$. Since $\overline{u}_i^n\ge\max(\overline{u}_{i-1}^n,\:\overline{u}_{i+1}^n)$, 
 we have ${\overline{\sigma}}_i^u=0$ and hence $\overline{u}_i^n=u^L_{i+\frac{1}{2}}=u^R_{i-\frac{1}{2}}>0.$
 Let us further assume that $u^{L}_{i-\frac{1}{2}}\ge0,\:u^{L}_{i+\frac{1}{2}}=u^{R}_{i-\frac{1}{2}}=\overline{u}_{i}^n>0$ and $u^{R}_{i+\frac{1}{2}}\le0$. Then, using the assumption \eqref{g1}, we have $g(u^{L}_{i-\frac{1}{2}})\ge0,\:g(u^{L}_{i+\frac{1}{2}})=g(u^{R}_{i-\frac{1}{2}})=g(\overline{u}_{i}^n)>0$ and $g(u^{R}_{i+\frac{1}{2}})\le0$.
 
Also, 
$$u^R_{i+\frac{1}{2}}\le0\implies \overline{u}_{i+1}^n\le0,\:$$ which we prove for completeness. Let us assume to the contrary that $\overline{u}_{i+1}^n>0.$ 
 Then, since $u^R_{i+\frac{1}{2}}=\overline{u}_{i+1}^n-{\overline{\sigma}}_{i+1}^u\le 0,\:$ ${\overline{\sigma}}_{i+1}^u$ should necessarily be positive. Now, if $\overline{u}_{i+2}^n\ge\overline{u}_{i+1}^n,\:$ then ${\overline{\sigma}}_{i+1}^u=0$ and if $\overline{u}_{i+2}^n<\overline{u}_{i+1}^n,\:$ then ${\overline{\sigma}}_{i+1}^u<0$. Both are contradictions to the fact that ${\overline{\sigma}}_{i+1}^u>0$. Hence, $\overline{u}_{i+1}^n\le0.$
 Hence
\[\hat{F}^\rho_{i+\frac{1}{2}}=0,\:\hat{F}^w_{i+\frac{1}{2}}=\max\Big(u^{L,\:+}_{i+\frac{1}{2}}g\big(u^{L,\:+}_{i+\frac{1}{2}}\big)\rho^L_{i+\frac{1}{2}},\:u^{R,\:-}_{i+\frac{1}{2}}g\big(u^{R, -}_{i+\frac{1}{2}}\big)\rho^R_{i+\frac{1}{2}}\Big)\]
and
\[\hat{F}^\rho_{i-\frac{1}{2}}= g\big(u^L_{i-\frac{1}{2}}\big)u^L_{i-\frac{1}{2}}\rho^L_{i-\frac{1}{2}},\:\hat{F}^w_{i-\frac{1}{2}}=u^L_{i-\frac{1}{2}}\hat{F}^\rho_{i-\frac{1}{2}}.\]
Now, keeping these inferences on $\overline{u}_{i}^n,\:\overline{u}_{i+1}^n$ in mind, we prove that $\tilde{C}_{i-\frac{1}{2}},\:\tilde{D}_{i+\frac{1}{2}}$ are non--negative and $\overline{\rho}_i^{n+1}(1-\tilde{C}_{i-\frac{1}{2}}-\tilde{D}_{i+\frac{1}{2}})\ge0$. 
Let us first consider $\tilde{C}_{i-\frac{1}{2}}$.
\begin{eqnarray*}\tilde{C}_{i-\frac{1}{2}}&=&-\lambda\hat{F}^\rho_{i-\frac{1}{2}}
 \frac{u_{i-\frac{1}{2}}^{L}-\overline{u}_i^{n}}
 {\overline{\rho}_i^{n+1}(\overline{u}_{i}^{n}-\overline{u}_{i-1}^n)}=-\lambda\hat{F}^\rho_{i-\frac{1}{2}}
 \frac{\left(\overline{u}_{i-1}^n+{\overline{\sigma}}_{i-1}^u-\overline{u}_i^{n}\right)}
 {\overline{\rho}_i^{n+1}(\overline{u}_{i}^{n}-\overline{u}_{i-1}^n)}=\lambda \frac{\hat{F}^\rho_{i-\frac{1}{2}}}
 {\overline{\rho}_i^{n+1}}\left(
 1-
 \frac{{\overline{\sigma}}_{i-1}^u}
 {\overline{u}_{i}^{n}-\overline{u}_{i-1}^n}\right),
 \end{eqnarray*}
 which is positive under the condition $K=1-
 \displaystyle\frac{{\overline{\sigma}}_{i-1}^u}
 {\overline{u}_{i}^{n}-\overline{u}_{i-1}^n}\ge0,\:$ which is always true. Also, 
 \[\tilde{D}_{i+\frac{1}{2}}=-\displaystyle\lambda\frac{\hat{F}^w_{i+\frac{1}{2}}-\hat{F}^\rho_{i+\frac{1}{2}}\overline{u}_i^{n}}{(\overline{u}_{i+1}^{n}-\overline{u}_{i}^n)}=-\displaystyle\lambda\frac{\hat{F}^w_{i+\frac{1}{2}}\overline{u}_i^{n}}{(\overline{u}_{i+1}^{n}-\overline{u}_{i}^n)}\ge0\] since $\overline{u}_{i+1}^{n}\le0,\:\overline{u}_{i}^n>0$ and $\hat{F}^w_{i+\frac{1}{2}}\ge0.$
Finally, we prove that $\tilde{C}_{i-\frac{1}{2}}+\tilde{D}_{i+\frac{1}{2}}\le 1.$
 \begin{eqnarray*}
\overline{\rho}_{i}^{n+1}&=& \overline{\rho}_{i}^{n}-\lambda(F^\rho_{i+\frac{1}{2}}-F^\rho_{i-\frac{1}{2}})= \overline{\rho}_{i}^{n}-\lambda(-F^\rho_{i-\frac{1}{2}})= \overline{\rho}_{i}^{n}+\displaystyle \frac{\tilde{C}_{i-\frac{1}{2}}{\overline{\rho}}_i^{n+1}}{K}
\end{eqnarray*}
Adding $-\tilde{D}_{i+\frac{1}{2}}\overline{\rho}_{i}^{n+1}$ on both sides, we have
\begin{eqnarray*}
\overline{\rho}_{i}^{n+1}\left(1-\frac{\tilde{C}_{i-\frac{1}{2}}}{K}-\tilde{D}_{i+\frac{1}{2}}\right)
&=&\overline{\rho}_{i}^{n}+\displaystyle\lambda\frac{\hat{F}^w_{i+\frac{1}{2}}-\hat{F}^\rho_{i+\frac{1}{2}}\overline{u}_i^{n}}{(\overline{u}_{i+1}^{n}-\overline{u}_{i}^n)}=\overline{\rho}_{i}^{n}+\displaystyle\lambda\frac{\hat{F}^w_{i+\frac{1}{2}}\overline{u}_i^{n}}{(\overline{u}_{i+1}^{n}-\overline{u}_{i}^n)}
\end{eqnarray*}
Now, let us assume that \[\hat{F}^w_{i+\frac{1}{2}}=u^{L, +}_{i+\frac{1}{2}}g\big(u^{L, +}_{i+\frac{1}{2}}\big)\rho^L_{i+\frac{1}{2}}=\overline{u}_{i}^{n}g\big(\overline{u}_{i}^{n}\big)\rho^L_{i+\frac{1}{2}}\]
Then, 
\begin{eqnarray*}
\overline{\rho}_{i}^{n+1}\left(1-\frac{\tilde{C}_{i-\frac{1}{2}}}{K}-\tilde{D}_{i+\frac{1}{2}}\right)&=& \overline{\rho}_{i}^{n}+\displaystyle\lambda\frac{\hat{F}^w_{i+\frac{1}{2}}}{(\overline{u}_{i+1}^{n}-\overline{u}_{i}^n)}=
 \overline{\rho}_{i}^{n}+\displaystyle\lambda\frac{\overline{u}_{i}^{n}g\big(\overline{u}_{i}^{n}\big)\rho^L_{i+\frac{1}{2}}}{(\overline{u}_{i+1}^{n}-\overline{u}_{i}^n)}\\
 &=&\frac{1}{2}(\rho^L_{i+\frac{1}{2}}+\rho^R_{i+\frac{1}{2}})+\displaystyle2\lambda\frac{\overline{u}_{i}^{n}g\big(\overline{u}_{i}^{n}\big)\rho^L_{i+\frac{1}{2}}}{2(\overline{u}_{i+1}^{n}-\overline{u}_{i}^n)}\\
 &=&\frac{1}{2}\rho^L_{i+\frac{1}{2}}\Big(1-2\displaystyle\lambda g(\overline{u}_i^n)\frac{\overline{u}_i^n}{\overline{u}_{i}^{n}-\overline{u}_{i+1}^n}\Big)+\frac{1}{2}\rho^R_{i+\frac{1}{2}}\ge0
\end{eqnarray*}
under the CFL-like condition \eqref{cfl2} and using the fact that
$
\displaystyle\frac{\overline{u}_i^n}{-\overline{u}_{i+1}^{n}+\overline{u}_{i}^n}\le1
\iff\overline{u}_{i+1}^{n}\le0
$.
 Hence, we get
 \[
\overline{\rho}_{i}^{n+1}\left(1-\frac{\tilde{C}_{i-\frac{1}{2}}}{K}-\tilde{D}_{i+\frac{1}{2}}\right)>0,\:\]
and since $0<K<1, $
we have \[
\overline{\rho}_{i}^{n+1}\left(1-\tilde{C}_{i-\frac{1}{2}}-\tilde{D}_{i+\frac{1}{2}}\right)> \overline{\rho}_{i}^{n+1}\left(1-\frac{\tilde{C}_{i-\frac{1}{2}}}{K}-\tilde{D}_{i+\frac{1}{2}}\right)>0.\]
Other cases follow similarly. 
\end{enumerate} 
We have now obtained that $\overline{\rho}_{i}^{n+1}\left(1-\tilde{C}_{i-\frac{1}{2}}-\tilde{D}_{i+\frac{1}{2}}\right)\ge0$, $\tilde{C}_{i-\frac{1}{2}}\ge0$ and $\tilde{D}_{i+\frac{1}{2}}\ge0$. Since $\overline{\rho}_i^{n+1}>0$ by previous lemma under the CFL-condition \eqref{cfl2},  we have $0\le\tilde{C}_{i-\frac{1}{2}}+\tilde{D}_{i+\frac{1}{2}}\le1.$ Then, using the incremental form \eqref{incre3} and \eqref{incre4}, we get that $\overline{u}_i^{n+1}$ is bounded by bounds of $\overline{u}_i^{0}.$ This completes the proof.
\end{proof}
When ${\overline{\sigma}}_{k}^u=0$ for all $k$, the scheme \eqref{seco} reduces to the first-order scheme \eqref{Numer} and hence the proof of \textbf{Lemma~\ref{cd}} gives the proof of \textbf{Lemma~\ref{lemma2}} under the CFL-like condition \eqref{cfl2} with cfl $=1$.
Since \eqref{RK}-\eqref{RK2} is a convex combination of Euler forward, hence it also preserves the properties of the system.
We now present some numerical experiments to show the performance of the scheme.
\section{Numerical Experiments}\label{NS}
The following experiments display the performance of first and higher order \textbf{(DDF)} scheme for capturing the solutions of Euler system with varying pressure and sources, both in one and two dimensions. The scheme will be tested on the non-trivial test cases from the existing literature and it will be seen that it can capture shocks, both classical and  $\delta\,-$ shocks as well as rarefactions efficiently.
\subsection{One Dimensional Experiments}
\begin{enumerate}[label=Experiment \arabic*:]
\item\underline{Moving shocks of \textbf{(PGD)}}:
The first experiment is to compare the performance of \textbf{(DDF)} scheme with the schemes proposed in \cite{berthon2006relaxation, leveque2004dynamics, yang2013discontinuous} for the data producing moving $\delta\,-$ shock. We take the initial data as Riemann Data with $(\rho_l,\:\rho_r)=(1,\:0.25), (u_l,\:u_r)=(1, 0)$ with the domain $[-0.5,\:0.5],\:T=0.4998,\:h=0.025,\:\mbox{cfl}=0.5$.
Limiters such as Minmod and Superbee have been used for reconstruction of $\rho,\:u$ for constructing the higher order version.
It can be observed from \textbf{Figure~\ref{cfl5s}} that both the limiters capture the right location of  $\delta\,-$ shock, though Superbee limiter performs slightly better than the others. As expected, the first order scheme produces  $\delta\,-$ shock with less height and more smearing around the location of  $\delta\,-$ shock as compared to the higher order. 
\begin{figure}[H]
 \centering
 \includegraphics[width=\textwidth,keepaspectratio ]{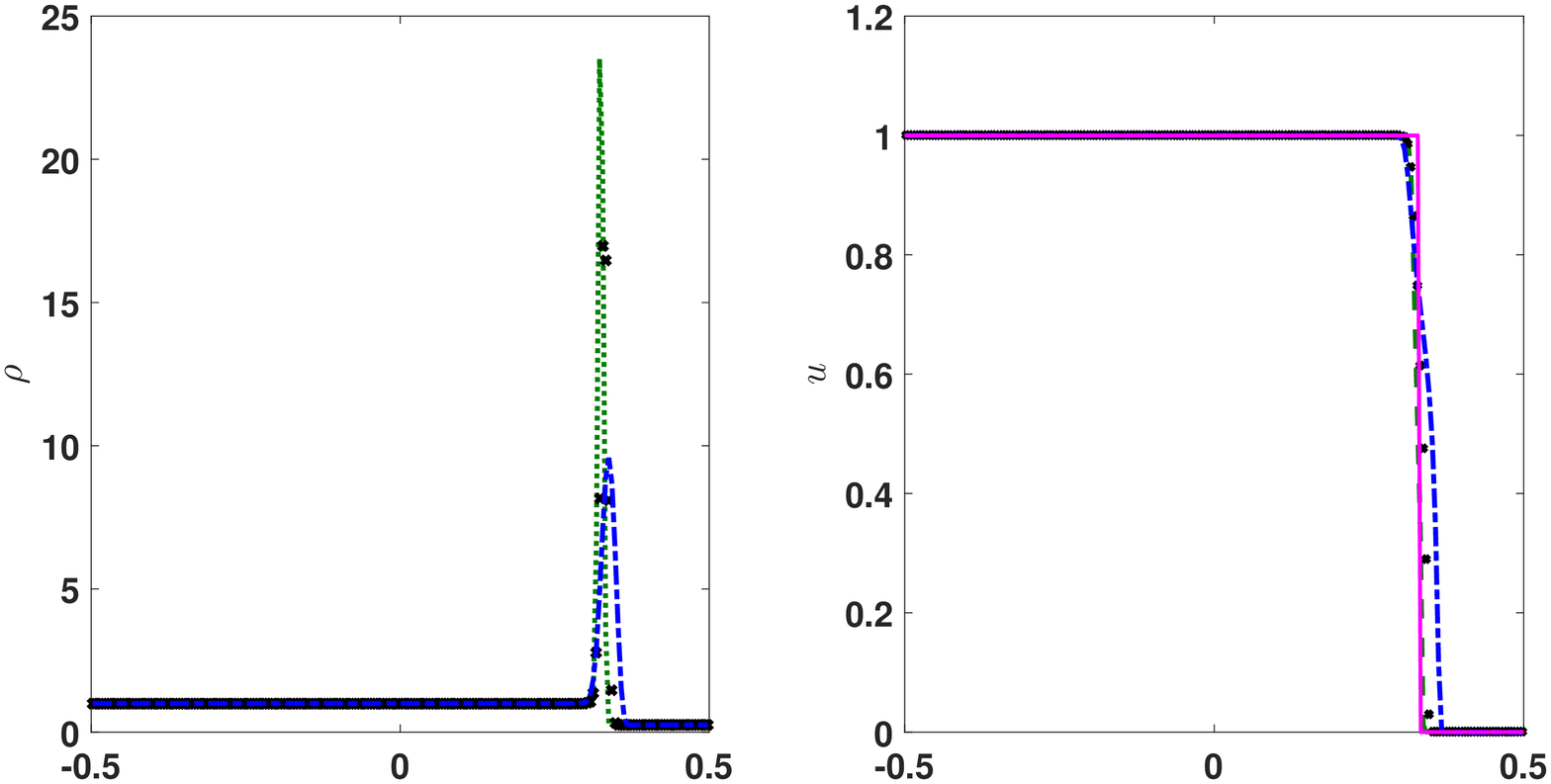}
 \caption{\textbf{(PGD)};Comparison of First and second-order \textbf{(DDF)} scheme: first-order Numerical Solution $U$({\color{blue}\protect\tikz[baseline]{\protect\draw[line width=0.5mm, dash dot] (0, .8ex)--++(1, 0) ;}}), Higher Order Numerical Solution $U$ with Minmod Limiter({\color{black}******}); Higher Order Numerical Solution $U$ with Superbee Limiter({\color{green}\protect\tikz[baseline]{\protect\draw[line width=0.5mm, loosely dashed] (0, .8ex)--++(1, 0)}}); Exact Solution $u$({\color{magenta}\protect\tikz[baseline]{\protect\draw[line width=.5mm] (0, .8ex)--++(1, 0) ;}})}
 \label{cfl5s}
\end{figure} 
\textbf{Figure~\ref{cfl6}} shows the graph of the primitive of the approximate solution $\rho_h(T)$ and indicates that the scheme captures the expected weight of the  $\delta\,-$ shock, i.e., $u_\delta T=0.3325.$
\begin{figure}[H]
 \centering
 \includegraphics[width=\textwidth,keepaspectratio]{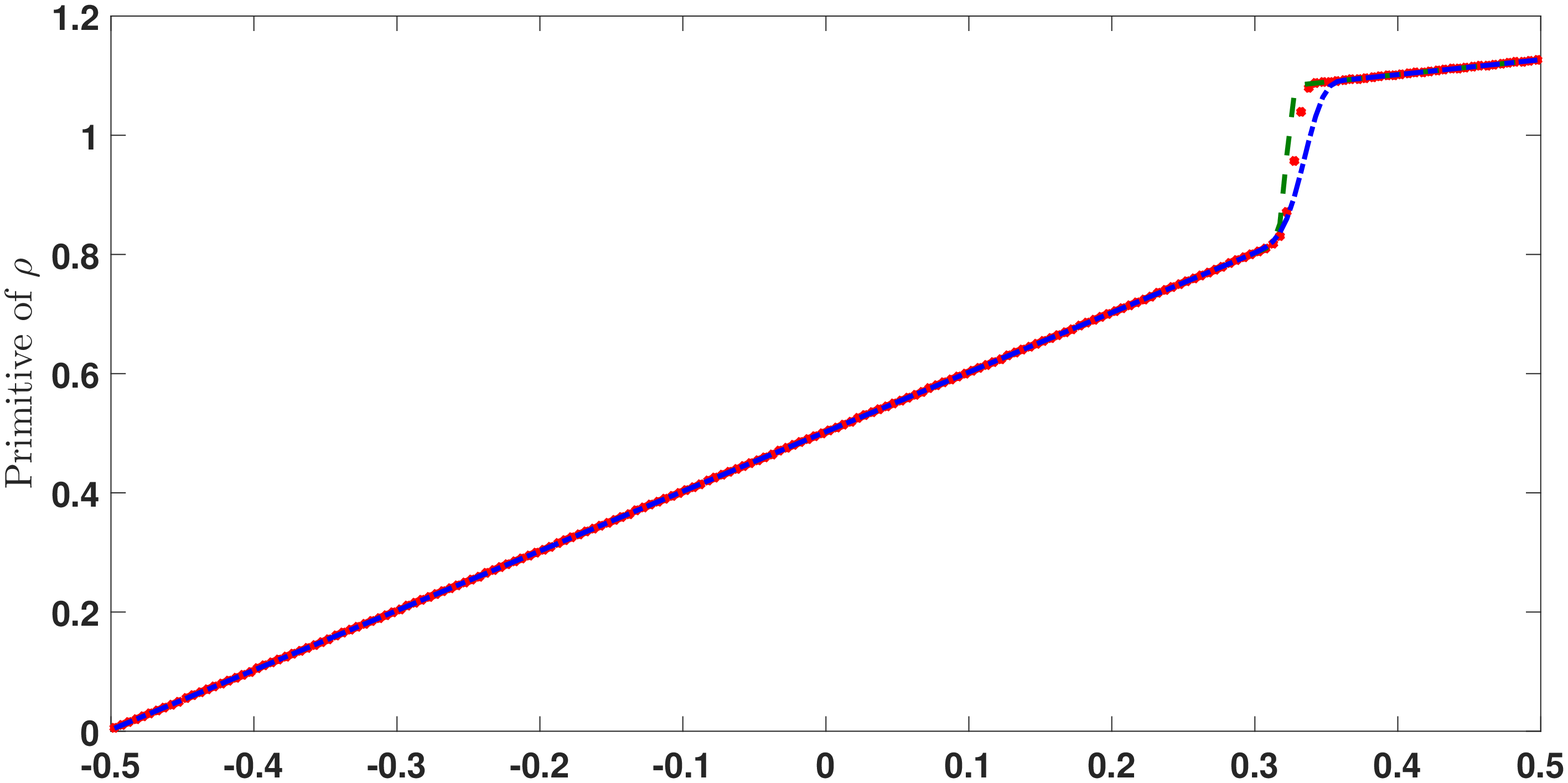}
 \caption{\textbf{(PGD)};Comparison of Primitives of $\rho_h$ from First and second-order \textbf{(DDF)} scheme: first-order Scheme({\color{blue}\protect\tikz[baseline]{\protect\draw[line width=0.5mm, dash dot] (0, .8ex)--++(1, 0) ;}}), Higher Order Scheme with Minmod Limiter({\color{red}******}); Higher Order Scheme with Superbee Limiter({\color{green}\protect\tikz[baseline]{\protect\draw[line width=0.5mm, loosely dashed] (0, .8ex)--++(1, 0)}})}\label{cfl6}
\end{figure} 
\item {\underline{Mixed Type Solutions of \textbf(PGD) system:}}
This experiment will show the performance of the first order and higher order \textbf{(DDF)} schemes to capture the solutions of \textbf{(PGD)} system, which can have both shocks and vaccum. The first initial data under consideration is given by:
\[\rho_0(x)=.5, 
, u_0(x)=
{ \left\{\begin{array}{ccl}
-0.5 & \, \mbox{if}\, &x<-0.5, \\
0.4 & \, \mbox{if}\, &-0.5<x<0,\\
0.4-x & \, \mbox{if}\, &0<x<0.5,\\
-0.4 & \, \mbox{if}\, &x>0.5.
\end{array}\right.}.
\] Let the domain be $[-1,\: 1],\: T=0.4998,\:h=0.025.$ It can be observed that the first order \textbf{(DDF)} scheme produces a hump in the solution 
\begin{figure}[H]
 \centering
 \includegraphics[width=\textwidth,keepaspectratio ]{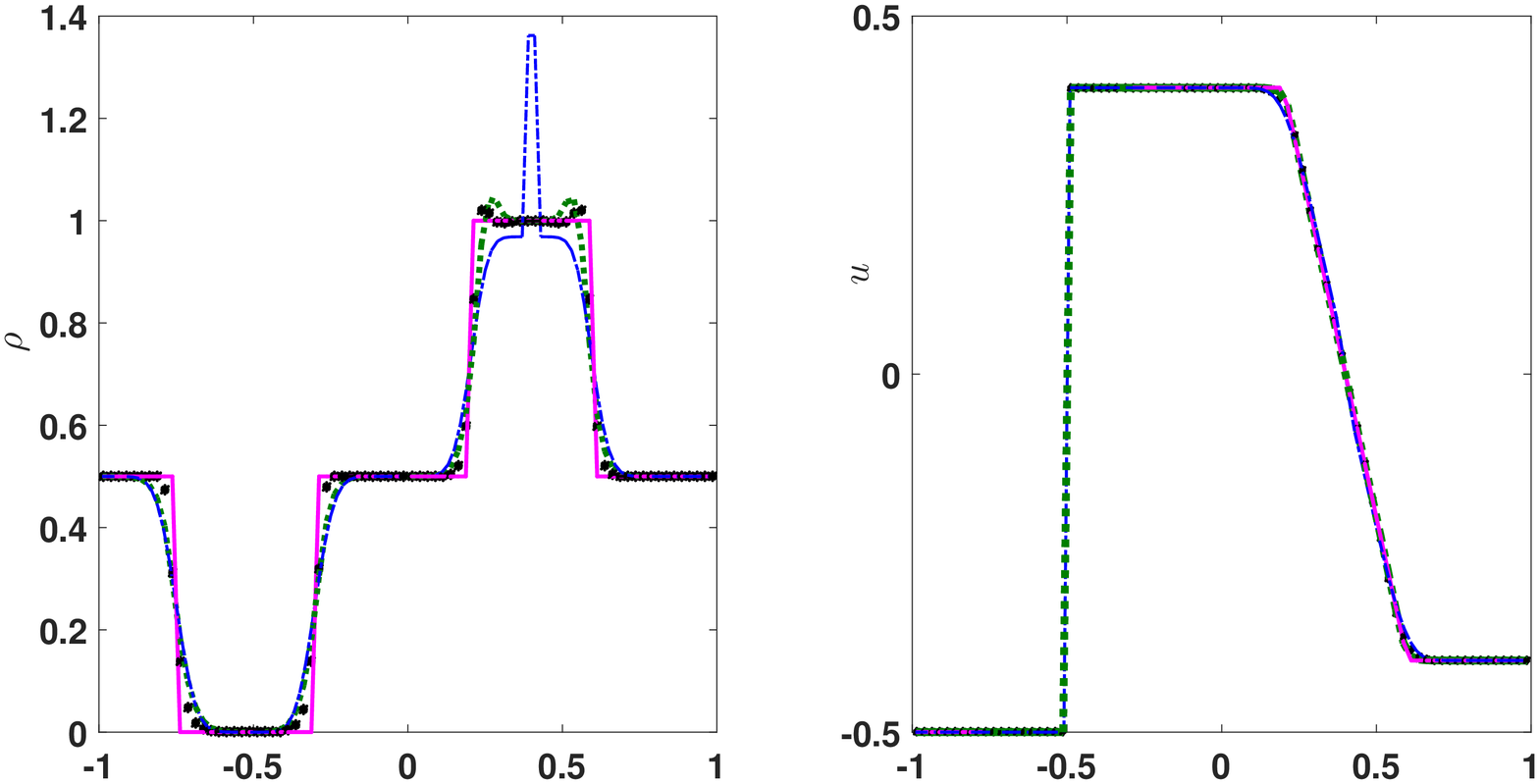}
 \caption{\textbf{(PGD)};Comparison of First and second-order \textbf{(DDF)} scheme: first-order Numerical Solution $U$({\color{blue}\protect\tikz[baseline]{\protect\draw[line width=0.5mm, dash dot] (0, .8ex)--++(1, 0) ;}}), Higher Order Numerical Solution $U$ with Minmod Limiter({\color{black}******}); Higher Order Numerical Solution $U$ with Superbee Limiter({\color{green}\protect\tikz[baseline]{\protect\draw[line width=0.5mm, loosely dashed] (0, .8ex)--++(1, 0)}}); Exact Solution $u$({\color{magenta}\protect\tikz[baseline]{\protect\draw[line width=.5mm] (0, .8ex)--++(1, 0) ;}})} \label{cfl8}
\end{figure}
in $\rho,\:$ which is resolved by the higher order version by all the limiters, though Superbee limiter, as in the case of $\delta\,-$ shock data, gives the best performance of all. The other schemes in literature, for example, the ones in \cite{berthon2006relaxation, bouchut2003numerical} also exhibit same behavior.

Next, we consider an example with the vaccum solutions. Consider \textbf{(PGD)} on the domain $[-0.5,\:0.5], T=0.5, M=200$ with 
\[\rho_0(x)=0.5, 
u_0(x)=
{ \left\{\begin{array}{ccl}
-0.5& \, \mbox{if}\, &x<0,\:\\
0.4 & \, \mbox{if}\, &x>0.
\end{array}\right.}
\] 
The exact solution for the density at  $T=0.5$ is given by:
\[
\rho(x,T)=
{ \left\{\begin{array}{ccl}
0.5& \, \mbox{if}\, &x<-0.25,\:\\
0& \, \mbox{if}\, &-0.25<x<0.2,\:\\
0.5 & \, \mbox{else}.
\end{array}\right.}
\] 
\textbf{Figure~\ref{logmi}} shows the logarithm of minimal density captured by the first-order and higher order \textbf{(DDF)} scheme over time, which are of the order $-24$ and $-128$ respectively and indicates that the scheme is able to capture the vaccum well.
\begin{figure}[H]
\centering
\includegraphics[scale=.5]{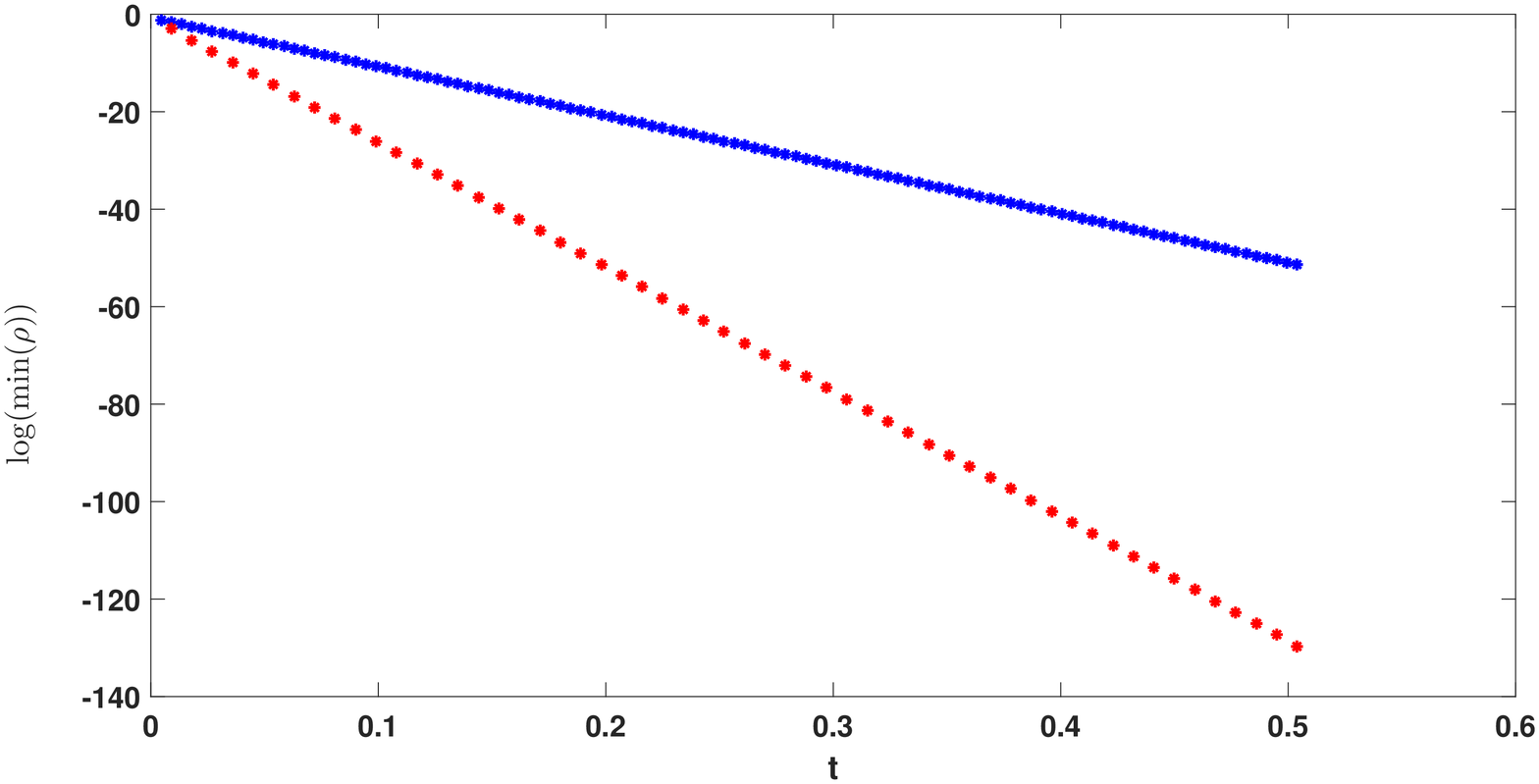}\\\includegraphics[scale=.5]{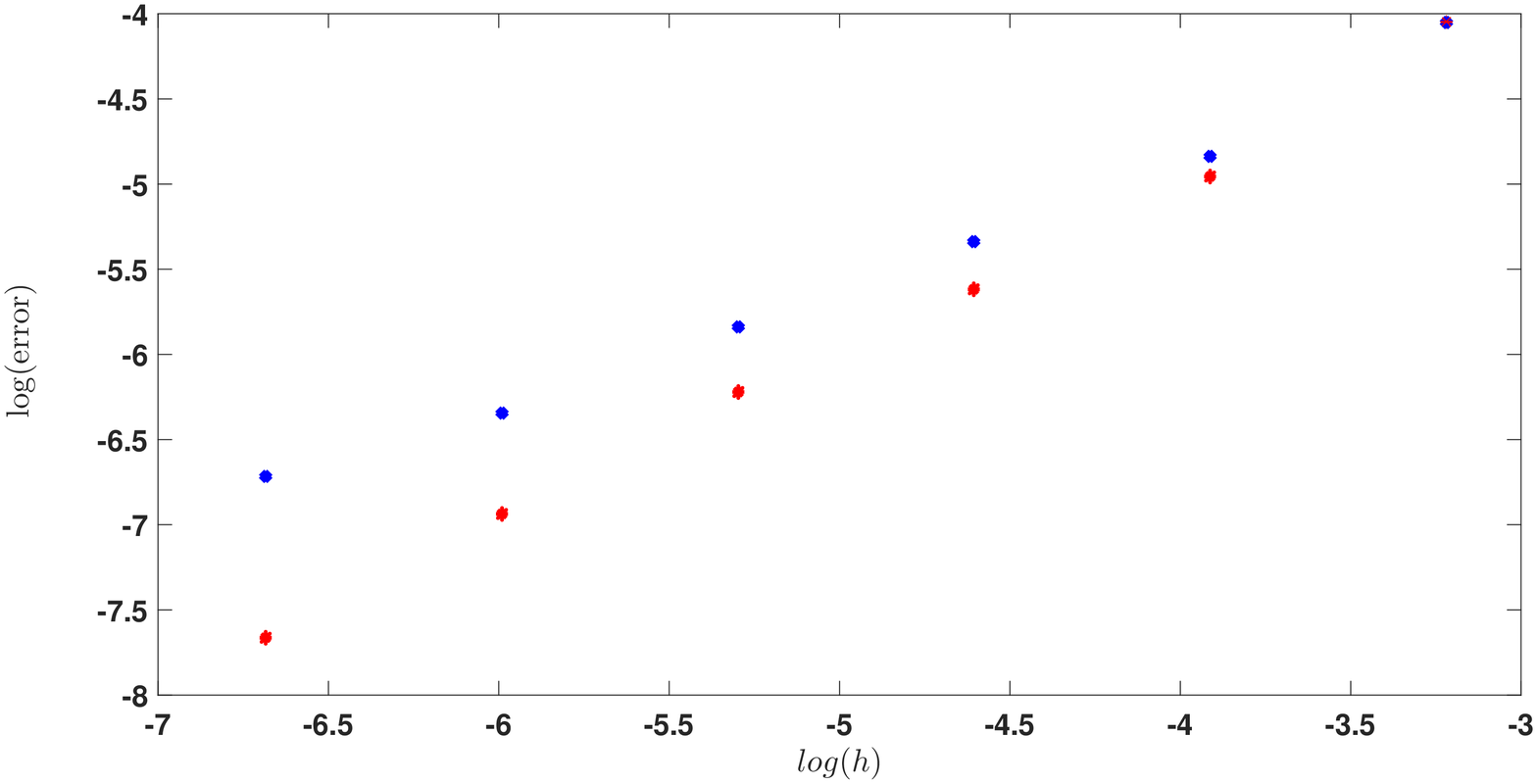}
\caption{Logarithmic of Minimal Density for Vaccum Solutions(Left); Logarithmic $L^1$ Error vs Logarithmic Space  Discretization $h$ in Vaccum Experiment(Right);  First Order ({\color{blue}****}), Second Order ({\color{red}****})}
\label{logmi}
\end{figure} 
Since the exact solution is known, we now verify the convergence rates of the scheme. \textbf{Figure~\ref{logmi}} shows that the slopes of the graphs of the logarithm of $L^1$ error versus $\log(h)$ are around $0.72$ and $1$ for first and higher order $\textbf{(DDF)}$ schemes respectively, which indicate that the convergence rates are near the expected convergence rates for discontinuous solutions.
\item  \underline{Interaction of two singular
shocks for \textbf{(PGD)}}: This experiment is to show the performance of \textbf{(DDF)} scheme for the piecewise constant initial data which develops  $\delta\,-$ shock as time progresses. This data has been used to describe \textbf{(PGD)} system as a model of collision of two semi-infinite clouds of dust from
the moment of impact onwards in \cite{leveque2004dynamics}, where the emerged  $\delta\,-$ shock is called as delta double-rarefaction when both clouds have been fully accreted. We demonstrate the performance of the scheme to capture this phenomenon by numerically solving  \textbf{(PGD)} system subject to the following initial data
\[\rho(x,-1)=
{ \left\{\begin{array}{cll}
2 & \, \text{if}\, -2<x<-1, \\
1 & \, \text{if}\, 1<x<5,\\
0 & \, \, \text{otherwise,}
\end{array}\right.}\quad  u(x,-1)=
{ \left\{\begin{array}{cll}
1 & \, \text{if}\,  -2<x<-1, \\
-1 & \, \text{if}\, 1<x<5,\\
0 & \,\, \text{otherwise,}
\end{array}\right.}
\] \begin{figure}[h]
    \centering
    \includegraphics[width=\textwidth,keepaspectratio]{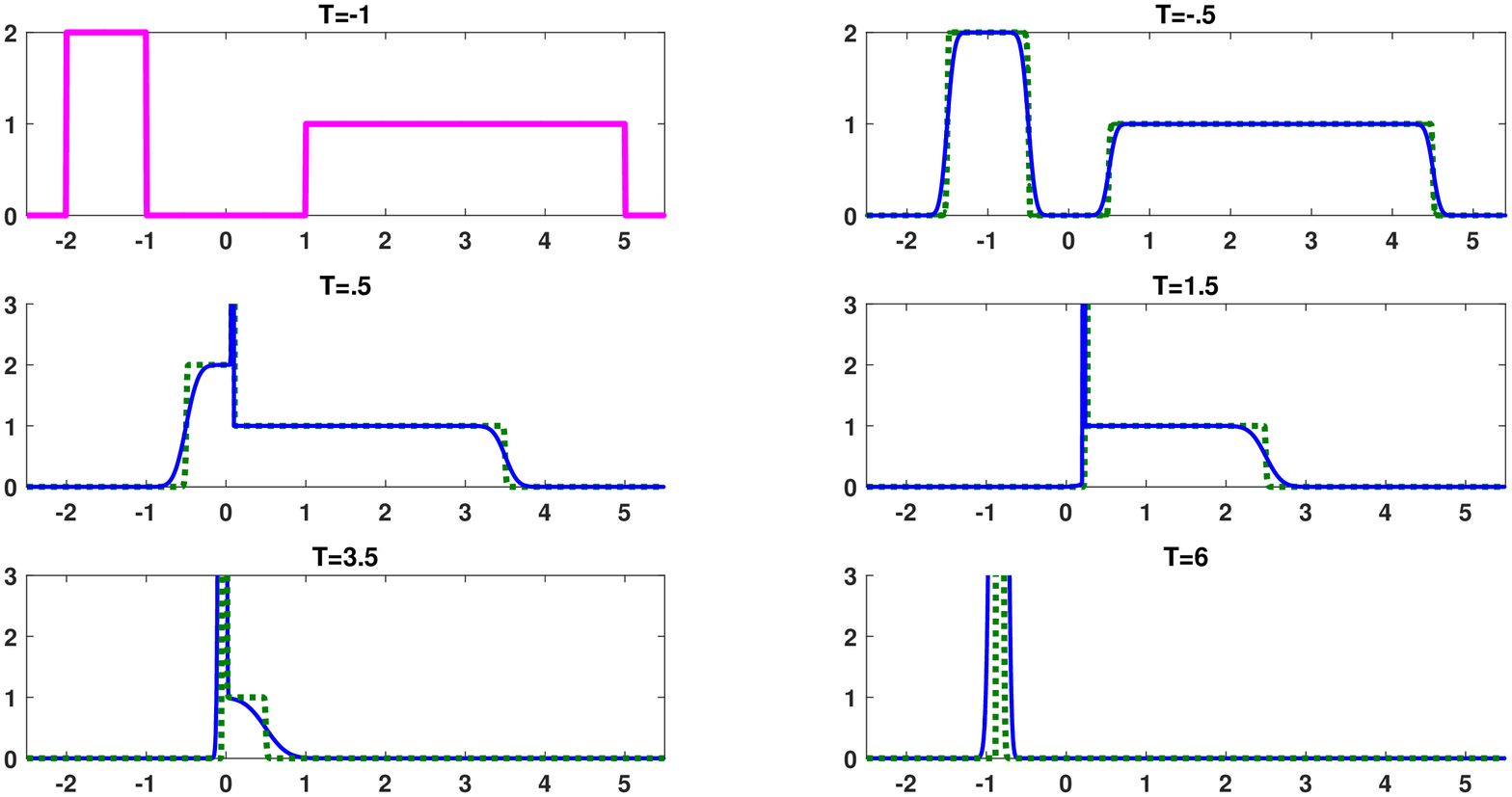}
    \caption{\textbf{(PGD)};Comparison of First and second-order \textbf{(DDF)} scheme: first-order Numerical Solution $\rho$({\color{red}\protect\tikz[baseline]{\protect\draw[line width=0.5mm, dash dot] (0, .8ex)--++(1, 0) ;}}), Higher Order Numerical Solution $\rho$ with Superbee Limiter({\color{green}\protect\tikz[baseline]{\protect\draw[line width=0.5mm, loosely dashed] (0, .8ex)--++(1, 0)}}); Initial Data $\rho$({\color{magenta}\protect\tikz[baseline]{\protect\draw[line width=.5mm] (0, .8ex)--++(1, 0) ;}})}
    \label{leveque}
\end{figure} where the initial data can be seen as collision of 2 dust clouds which are of length $1$ and $4$, which move in opposite directions. Let the domain be $[-2.5,\: 5.5],\: h=0.01.$ \textbf{Figure \ref{leveque}} plots the numerical densities at times $T=-1, -0.5, 0.5, 1.5 ,3.5, 6$. The locations of the shocks have been compared with \cite[Fig.~2,3]{leveque2004dynamics} and it can be seen that the scheme captures the shocks at the expected locations. All the figures have been plotted on the
same vertical scale with peak values of the computed density clipped near in the region of the delta
functions. Moreover, the weights of the  $\delta\,-$ shocks captured by scheme described in \cite{leveque2004dynamics} at $T=0.5, 1.5, 3.5$ and $T=6$ are $20.2, 58.2, 51.6$ and $29.8$ respectively, while for first order \textbf{(DDF)} scheme, they are $73.9,301.2,93.14$ and $40.3$ and for higher order \textbf{(DDF)}, 
they are $88.2,331.1,157$ and $99.7$ respectively. The number of the cells, which have not been shown to display the clipped densities, are similar to those in \cite{leveque2004dynamics} and around 4 and 11 on either side of the  $\delta\,-$ shock at times $T=3.5$ and $T=6.$
\item \underline{Moving Shocks of \textbf{(GPGD)}}:
This experiment is to show the performance of the first order \textbf{(DDF)} scheme to capture entropic $\delta\,-$ shocks for moving shocks of \textbf{(GPGD)}. With $g(u)=u^3$ and initial data as Riemann Data $(\rho_l,\:\rho_r)=(1,\:0.25), (u_l ,\:u_r)=(1, 0)$ on the domain $[-0.5,\:0.5], $ with $T=0.4988,\:h=0.025,\:\mbox{cfl}=0.5$, it can be seen in \textbf{Figure~\ref{below}} that the scheme captures the entropic  $\delta\,-$ shock as the shock location lies between $g(0)$ and $g(1)$, which fits the results of \cite{huang,mitrovic2007delta}.
\begin{figure}[H]
 \centering
 \includegraphics[width=\textwidth,keepaspectratio]{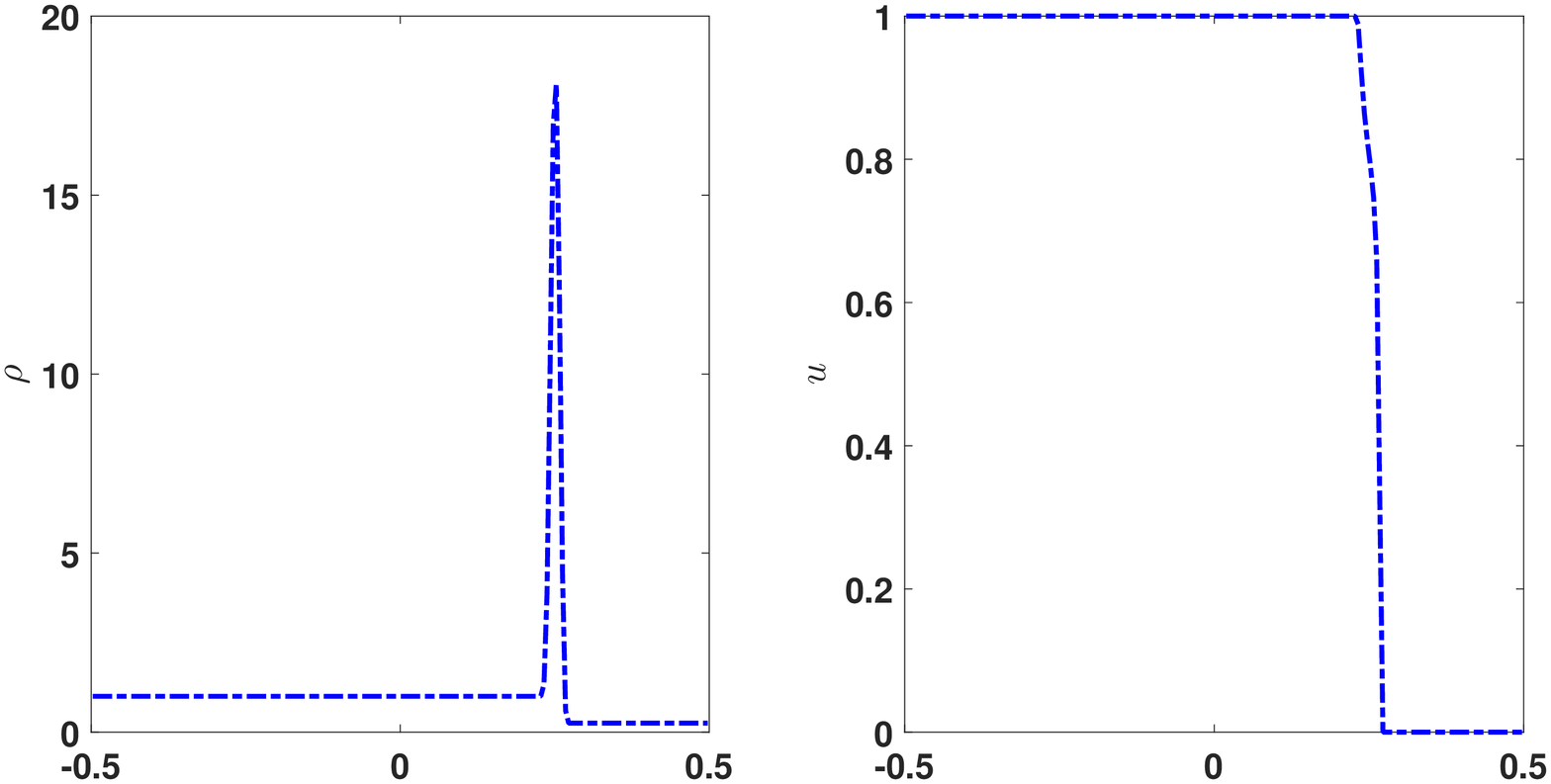}
 \caption{\textbf{(GPGD)}: Numerical Solution $U$({\color{blue}\protect\tikz[baseline]{\protect\draw[line width=0.5mm, dash dot] (0, .8ex)--++(1, 0) ;}})}
 \label{below}
\end{figure}
\item {\underline{The system \eqref{hy}-\eqref{hy1} with  $P=0,\:S=\beta \rho$:}}
This experiment is to display the performance of the first order \textbf{(DDF)} scheme for capturing the solutions of \eqref{hy}-\eqref{hy1} with zero pressure and non-zero source term $\beta \rho$.
Like \textbf{(PGD)}, this system is also non--strictly hyperbolic and has double eigenvalue $u+\beta t$. It has been established in \cite{shen2015riemann} that whenever $u_l>0>u_r$, it admits $\delta\,-$ shock at $x(t)=v_\delta t +\frac{1}{2}\beta t^2$ with the weight $w(t)=w_\delta(t)$, where \[v_\delta=\frac{\sqrt{\rho_l}u_l+\sqrt{\rho_r}u_r}{\sqrt{\rho_l}+\sqrt{\rho_r}},\:w_\delta=\sqrt{\rho_l\rho_r}[u], \] 
 and when $u_l<u_r, $ it admits the following vacuum solution 
\[(\rho,\:v)(x,\:t)=
{ \left\{\begin{array}{ccl}
(\rho_l,\:u_l)& \, \mbox{if}\, &x<u_lt+\frac{1}{2}\beta t^2, \\
(\rho_r,\:u_r) & \, \mbox{if}\, &x>u_rt+\frac{1}{2}\beta t^2\\
\mbox{Vacuum}& \, \mbox{Otherwise}.
\end{array}\right.}
\]
To extend the first order \textbf{(DDF)} scheme for this system, we follow the approach of \cite{shen2015riemann} and write the system as \begin{eqnarray*}
\rho_t + (\rho (v+\beta t))_x=0,\:\\
w_t + (w(v+\beta t))_x=0,\:
\end{eqnarray*}
where $w=\rho v.$
The first equation can be treated like \eqref{1} for given $v(x,\:t)+\beta t$. For given $\rho(x,\:t)>0,\:$ the function \[w\mapsto\frac{w^2}{\rho(x,\:t)}+w\beta t, \] is a convex function with the minimum at $-0.5\rho\beta t$ and hence, the second equation can be treated in a similar way to \eqref{second}. Hence, the numerical fluxes can be defined in the following way:
\[F_{i+\frac{1}{2}}^{n}=\begin{pmatrix} F^{\rho,\:n}_{i+\frac{1}{2}},\max\Big(q(\rho_i^n,\:\max(w_i^n,\:G_i^n)), q(\rho_{i+1}^n,\:\min(w_{i+1}^n,\:G_{i+1}^n)) \Big)\:
\end{pmatrix}^{T}\]
where $q(\rho,\:w)=\frac{w^2}{\rho},\:G_i^n=-0.5\rho_i^n\beta t^n.$
We consider the domain $[-1.2,\:1.2]$ with $T=0.4983,\:M=500$ and $\beta=0.5$. It can be seen in \textbf{Figure~\ref{fig:1}} that the scheme is able to capture the expected weight and location of  $\delta\,-$ shock efficiently. 
\begin{figure}[h]
 \centering
 \includegraphics[width=\textwidth,keepaspectratio]{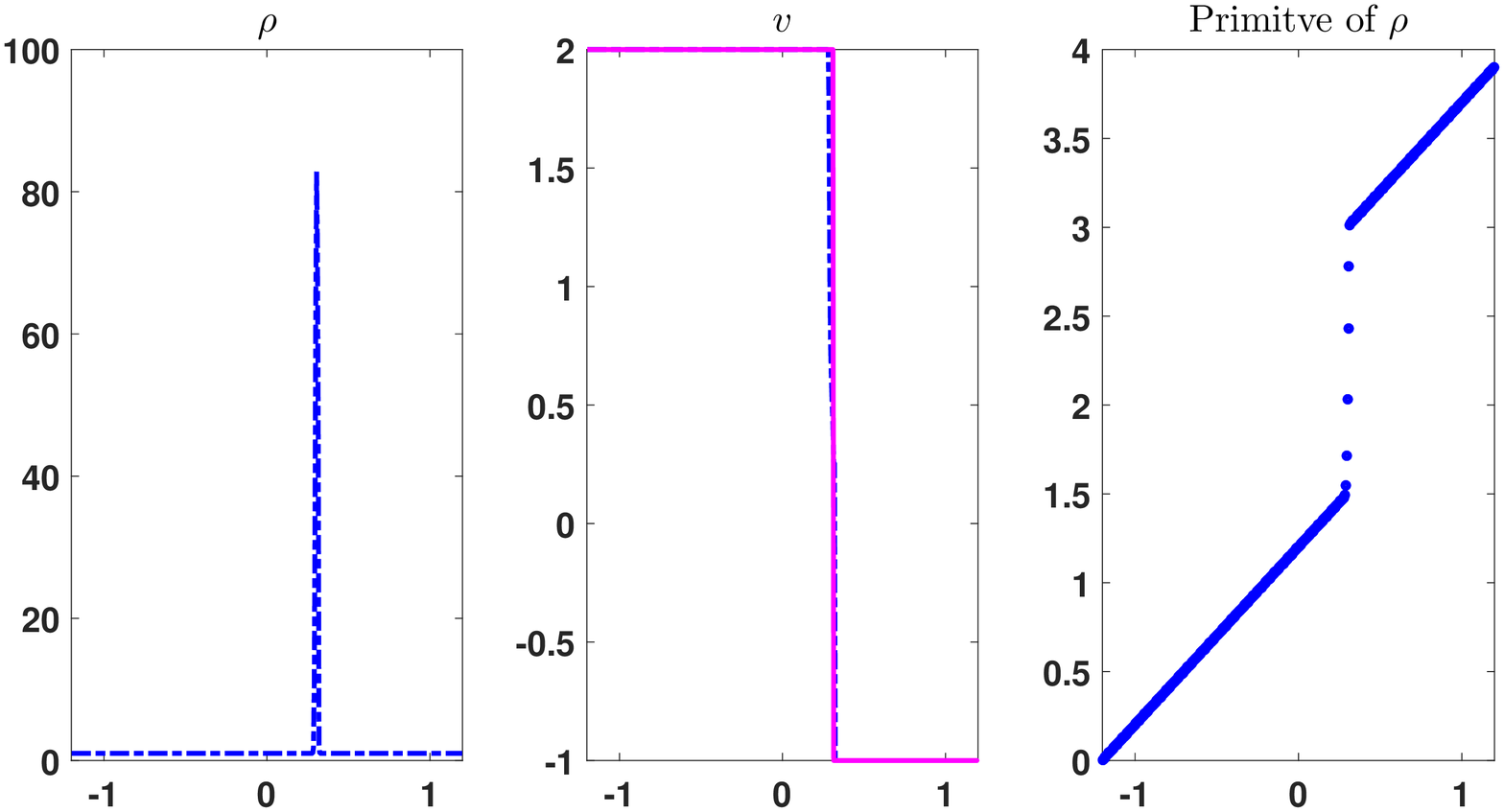} 
 \caption{\textbf{(PGDS)}:Numerical Solution $U$({\color{blue}\protect\tikz[baseline]{\protect\draw[line width=0.5mm, dash dot] (0, .8ex)--++(1, 0) ;}}), Exact Solution $v$({\color{magenta}\protect\tikz[baseline]{\protect\draw[line width=.5mm] (0, .8ex)--++(1, 0) ;}}), Primitive of $\rho$({\color{blue}****})}
 \label{fig:1}
\end{figure}
 The performance of the scheme to capture vacuum solutions has been displayed in  \textbf{Figure~\ref{fig:13}}, with $\rho_l=\rho_r=1, u_l=-2$ \begin{figure}[H]
 \centering
 \includegraphics[width=\textwidth,keepaspectratio ]{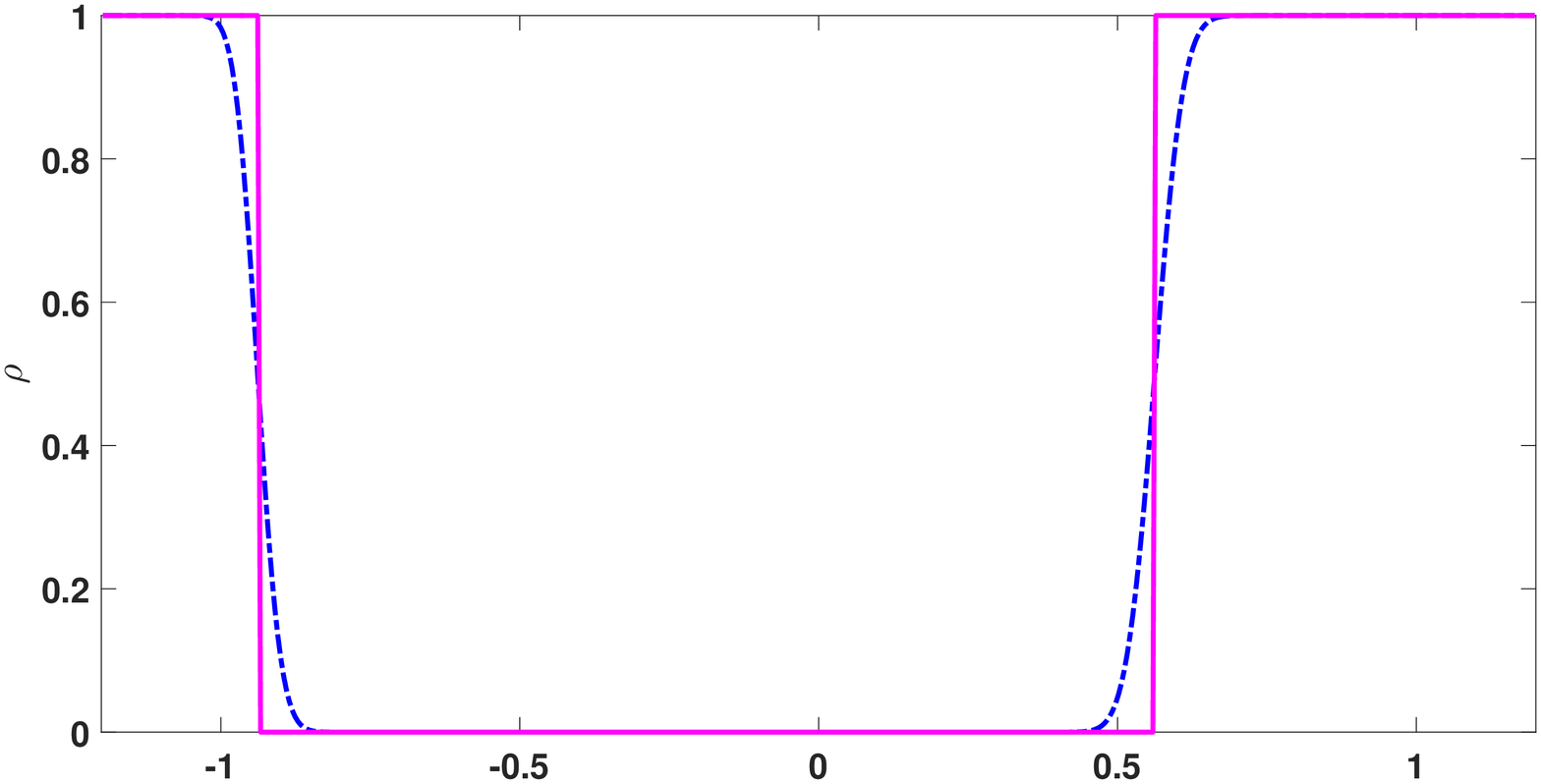} 
 \caption{\textbf{(PGDS)}: Numerical Solution $\rho$({\color{blue}\protect\tikz[baseline]{\protect\draw[line width=0.5mm, dash dot] (0, .8ex)--++(1, 0) ;}}), Exact Solution $\rho$({\color{magenta}\protect\tikz[baseline]{\protect\draw[line width=.5mm] (0, .8ex)--++(1, 0) ;}})}
 \label{fig:13}
\end{figure}
and $u_r=1.$ It is clear that the scheme can capture the vacuum solutions well, where the minimal density is of the order $10e-233$.
\item  {\underline{The system \eqref{hy}-\eqref{hy1} with  $P\ne0,\:S=0$:}}
We now extend \textbf{(DDF)} to capture the $\delta\,-$ shock type solutions of \eqref{hy}-\eqref{hy1}
where zero source, but with non--zero pressure.
 Since the first equation does not change, there is no change in its approximation. For the second equation, since the pressure term $P(\rho)$ is only a function of $\rho,\:$ hence the second equation can still be considered as a convex-convex discontinuous flux for a given $\rho(x,\:t^n)$ and the numerical flux is given by:
\begin{equation}\label{FwNs}
 F^{w,\:n}_{i+\frac{1}{2}}=\max\left(\frac{{(\max(w_i^n,\:0))}^2}{\rho_i^n}+P(\rho_i^n),\:\frac{{(\min(w_{i+1}^n,\:0))}^2}{\rho_{i+1}^n}+P(\rho_{i+1}^n)\right). 
\end{equation}
We now use this extension of \textbf{(DDF)} scheme to compute the numerical solutions of
\textbf{(CGD)}, which have non--zero $P$. 
We now exhibit the performance of the extended scheme for the physical systems described in the beginning of the article.
\begin{enumerate}
\item \underline{ $\delta\,-$ shocks of \textbf{(CGD)}}: 
This system is of type \eqref{hy} with $S=0,\:P=s\rho^{-\alpha}.$
It has been pointed out in \cite{wang2013riemann} that whenever $u_l\ge u_r,\:0<\alpha<1$, this system admits  $\delta\,-$ shocks at the location $x(t)=u_\delta t, $ where 
\[w_\delta=\sqrt{\rho_l \rho_r {[u]}^2 -[\rho][P]}, u_\delta =
\displaystyle\frac{[w]t +w_\delta}{[\rho]}.\]
\textbf{Figure~\ref{Fig 6_co}} shows the performance of extension of \textbf{(DDF)} scheme 
with $\alpha=0.5,\:s=5 ,\:M=1000$, with initial data as Riemann data $\rho_l =3,\:\rho_r =1,\:u_l=4, u_r=-4$ on the domain $[-2,\:2]$ and at times $T_1=0.05,\:T_2=0.1996$.
\begin{figure}[H]
 \centering
 \includegraphics[width=\textwidth,keepaspectratio]{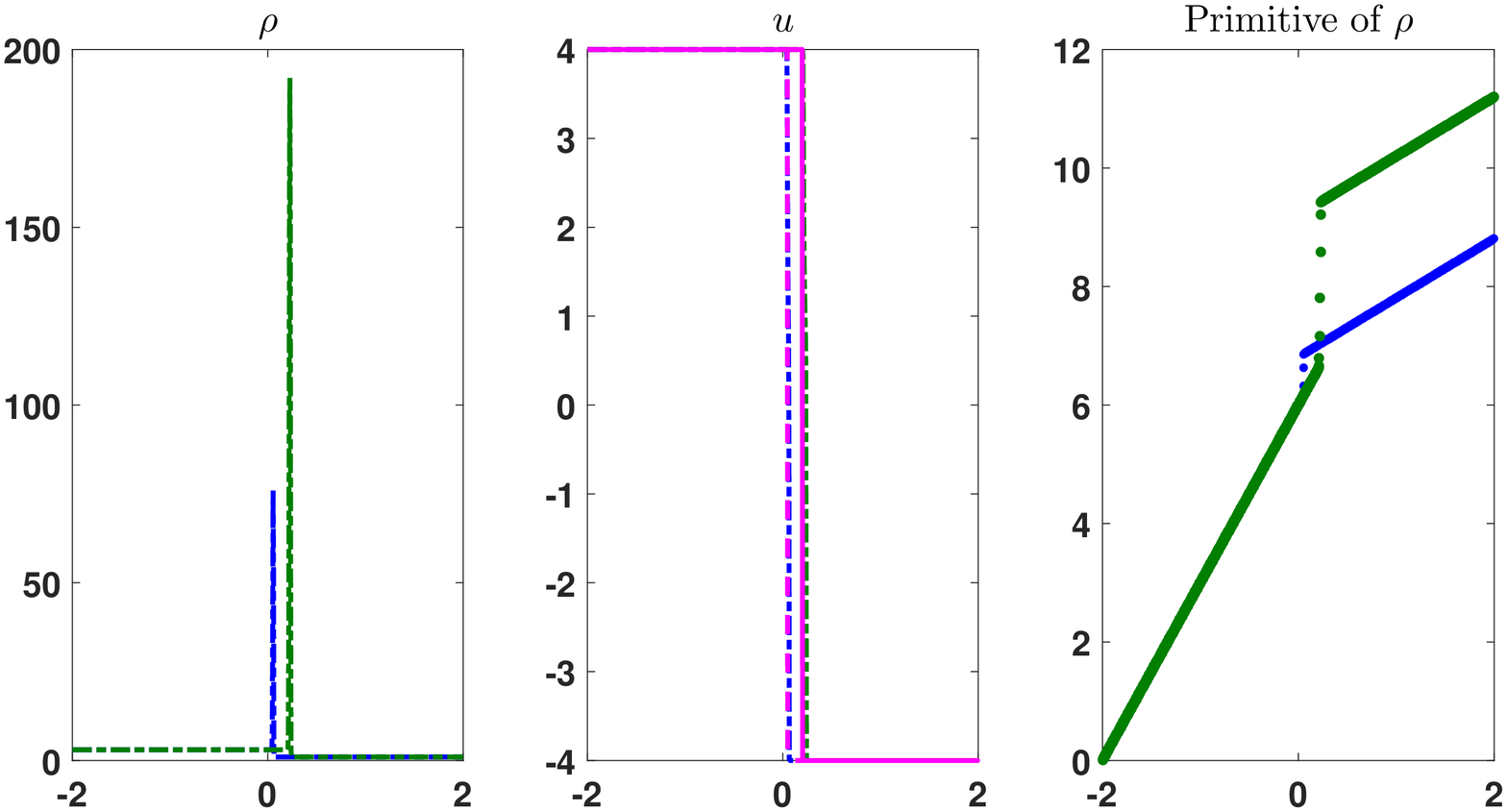}
 \caption{\textbf{(CGD)}:Numerical Solution $U(\cdot,\:T_1)$({\color{blue}\protect\tikz[baseline]{\protect\draw[line width=0.5mm, dash dot] (0, .8ex)--++(1, 0) ;}}), Exact Solution $u(\cdot,\:T_1)$({\color{magenta} \protect\tikz[baseline]{\protect\draw[line width=.5mm] (0, .8ex)--++(1, 0) ;}}), Primitive of $\rho(\cdot,\:T_1)$({\color{blue}*****}), Numerical Solution $U(\cdot,\:T_2)$({\color{darkgreen}\protect\tikz[baseline]{\protect\draw[line width=0.5mm, dash dot] (0, .8ex)--++(1, 0) ;}}), Exact Solution $u(\cdot,\:T_2)$({\color{magenta} \protect\tikz[baseline]{\protect\draw[line width=.5mm] (0, .8ex)--++(1, 0) ;}}), Primitive of $\rho(\cdot,\:T_2)$({\color{darkgreen}*****})}
 \label{Fig 6_co}
\end{figure}
It can be seen that the \textbf{(DDF)} scheme can capture the expected shock location for $\rho$ in contrast to the solutions displayed in \cite[Fig.~4.1]{wang2013riemann} where the shocks lag behind a few units. Also, the height of  $\delta\,-$ shock is more than those presented in \cite[Fig.~4.1]{wang2013riemann}. 

To further evaluate the efficiency of our scheme, we compare the numerical weight of the $\delta\,-$ shock with the expected weight $w_\delta T, $ where $w_\delta=14.0081$ for the given data. expected weights are hence given by $0.7004$ and $2.7960$. It can be seen in \textbf{Figure~\ref{Fig 6_co}} that the numerical weight matches well with the expected weight.
The locations for numerical  $\delta\,-$ shocks at $T_1$ and $T_2$ are compared with their exact locations in \textbf{Table.~\ref{tab:LL2}} .
 \begin{table}[H]\centering\begin{tabular}{|c|c|c|c|}
 \hline
 \textbf{Time} & \textbf{Exact Shock Location} & \textbf{Observed Shock Location} \\
 \hline
 $T_1$ & $0.0538$&$0.54$\\
 \hline
 $T_2$ &$0.2170$&$0.226$\\
 \hline
 \end{tabular}
 \caption{\textbf{(CGD)}: Location of  $\delta\,-$ shock by \textbf{(DDF)} scheme}
 \label{tab:LL2}
 \end{table}
 \end{enumerate}
 \end{enumerate}
\subsection{Two Dimensional Experiments}
We consider \textbf{(PGD)} system in multi dimensions with different initial conditions and show the performance of higher order \textbf{(DDF)} scheme with Superbee Limiter, used for $\rho,u$ and $v$. The first order \textbf{(DDF)} scheme behaves in a similar manner, with some expected diffusion.
\begin{enumerate}[label=Experiment \arabic*:]
\item {\underline{Vaccum Solutions:}} We first consider the following test data, considered in \cite{yang2013discontinuous}:
 \[\rho(x,y,0)=0.5, (u,v)(x,y,0)=
{ \left\{\begin{array}{ccl}
(0.3,0.4) & \, \mbox{if}\,  &x>0,y>0, \\
(-0.4,0.3) & \, \mbox{if}\, &x<0,y>0,\\
(-0.3,-0.4) & \, \mbox{if}\, &x<0,y<0,\\
(0.4,-0.3) & \, \mbox{if}\, &x>0,y<0.\\
\end{array}\right.}\]
\textbf{Figure~\ref{lev_rare}} shows the numerical density and the velocity vector field at the time $T=0.1.$ We can see that the vaccum solutions have been captured well and maintain non-negative density.
The minimal density achieved by the higher order \textbf{(DDF)} scheme is of the order of $10e-4.$
\begin{figure}[H]
    \centering
    \includegraphics[width=\textwidth,keepaspectratio]{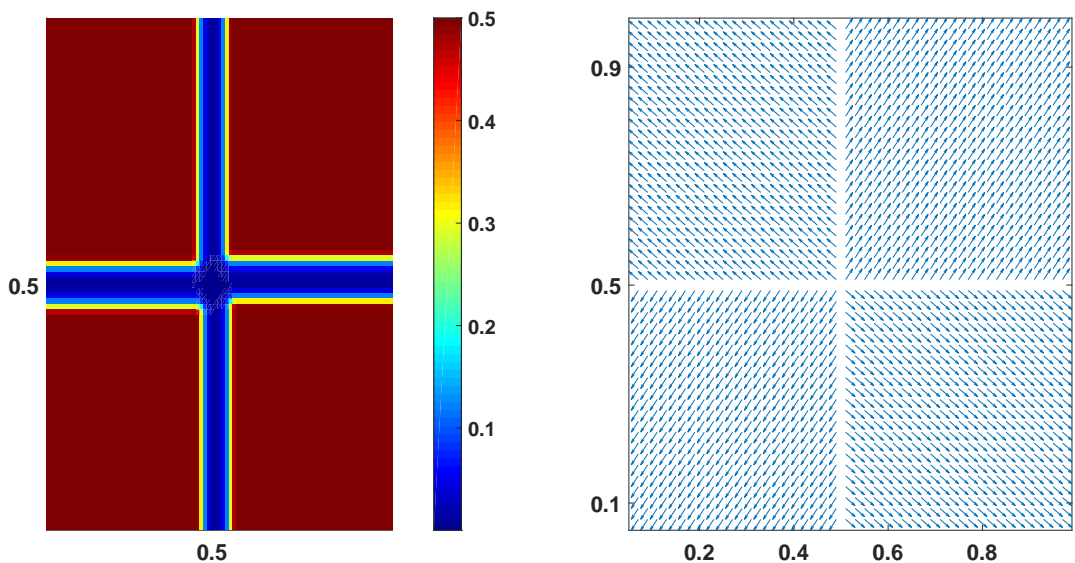}
    \caption{Numerical density (left) and velocity field (right) at $T = 0.1$}
    \label{lev_rare}
\end{figure}
\item {\underline{ $\delta\,-$ shock formation due to particles moving towards centre:}} 
We consider the following initial data considered in \cite{yang2013discontinuous}:
\[\rho(x,y,0)=1/100,\big(u(x,y,0),v(x,y,0)\big)=-\frac{1}{10}\big(\cos(\theta),\sin(\theta)\big),\]
where $\theta$ is the polar angle. \begin{figure}[H]
    \centering
    \includegraphics[width=\textwidth,keepaspectratio]{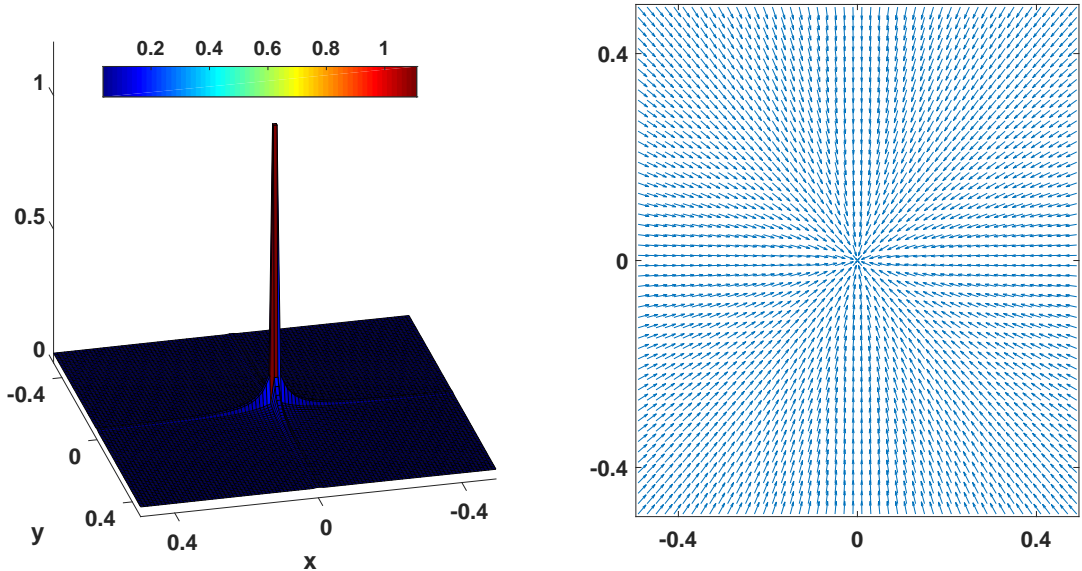}
    \caption{Density plots(left) and velocity vector plots(right) at times $T=0.5$}
    \label{shucenter}
\end{figure}
It is clear from \textbf{Figure~\ref{shucenter}} that a single  $\delta\,-$ shock of height over $1$ is formed at the origin, towards which all the particles are moving.
\item {\underline{ $\delta\,-$ shock formation due to cloud collision:}} As in one dimension, we consider an example to capture the phenomenon of cloud collision in two dimensions. The initial conditions are taken such that $v=0, u\ne0$ so as not to have any directional impact on the solutions, i.e., the particles do not change their directions and move on their own way. For this
purpose, initial data is taken to model two clouds with the same speed and density but having preassigned velocity to move in opposite directions. In particular, we consider the following initial data:
\[\big(\rho(x,y,0),u(x,y,0),v(x,y,0)\big)=
{ \left\{\begin{array}{cll}
(1,0.5,0) &  \text{ if   } x\in [-0.3,-0.2], y \in [-0.15, 0.05],\\
(1,-0.5,0) & \text{ if   }x\in [0.2,0.3],y \in [0.05,0.15],\\
(0.1,0,0) &\text{ otherwise}.
\end{array}\right.}\]A part of each cloud is physically 
expected to entirely
merge at $T\approx 0.52$, and another part of the cloud moves further on its own way post this time. The results are compared with \cite[Fig~~2]{jung2020relaxation}, where the highest densities achieved are $1.4, 8$ and $10$ at times $T=0.2,0.52$ and $0.8$. The results given by \textbf{(DDF)} scheme in \textbf{Figures~\ref{cloud}} and \ref{section1} show that the density concentrations are well captured by the scheme with slightly higher concentrations than those captured by the relaxation schemes in \cite{jung2020relaxation}.  As expected physically, there are no fluctuations in the $y-$ direction since $v=0$. It can be seen from the velocity plot at time $T=0.8,$ that a strong  $\delta\,-$ shock is formed at the centre due to collision of clouds, but the remaining part of the clouds, keep their velocity as it is and move in their own directions.\begin{figure}[H]
    \centering
    \includegraphics[width=\textwidth,keepaspectratio]{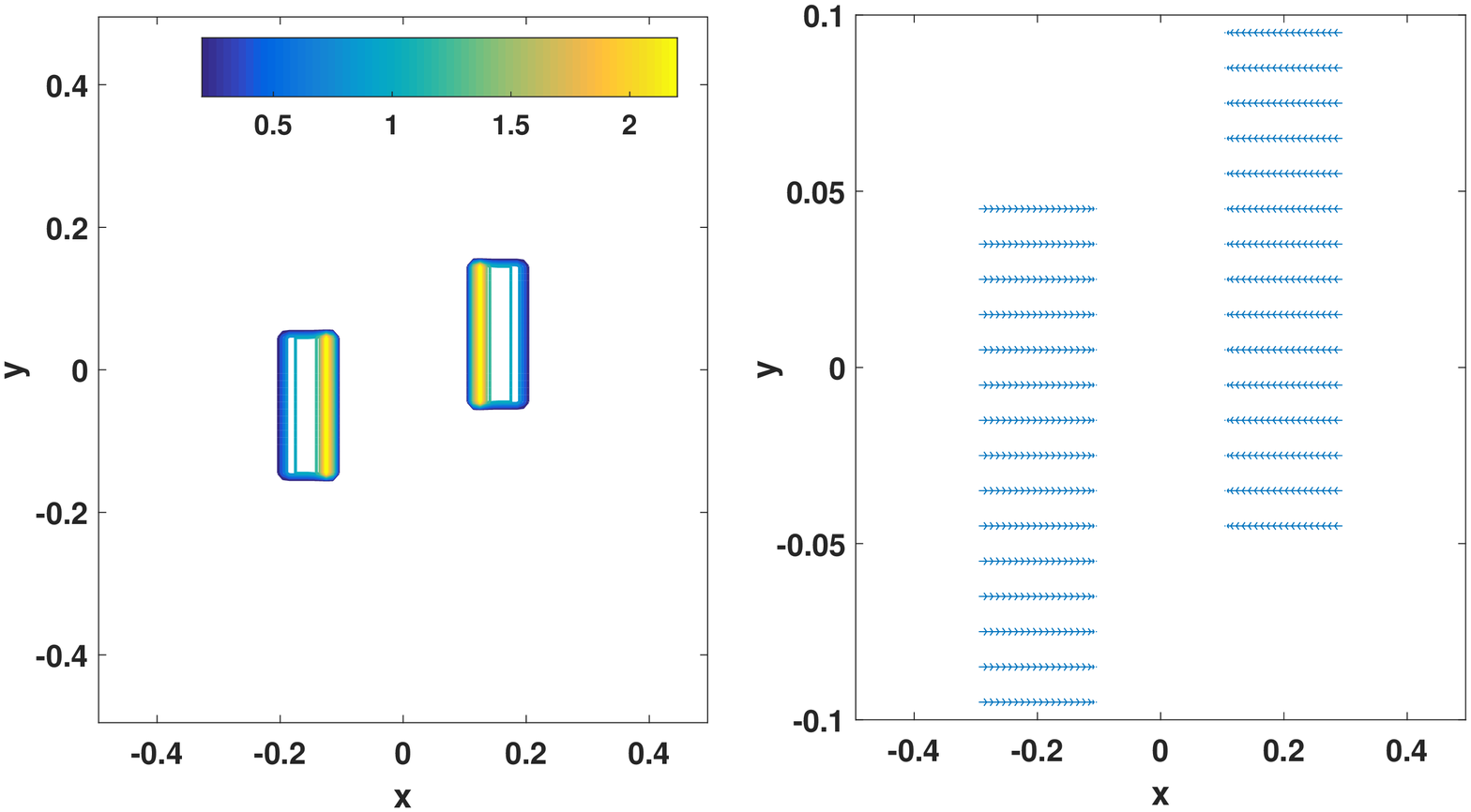}
  \includegraphics[width=\textwidth,keepaspectratio]{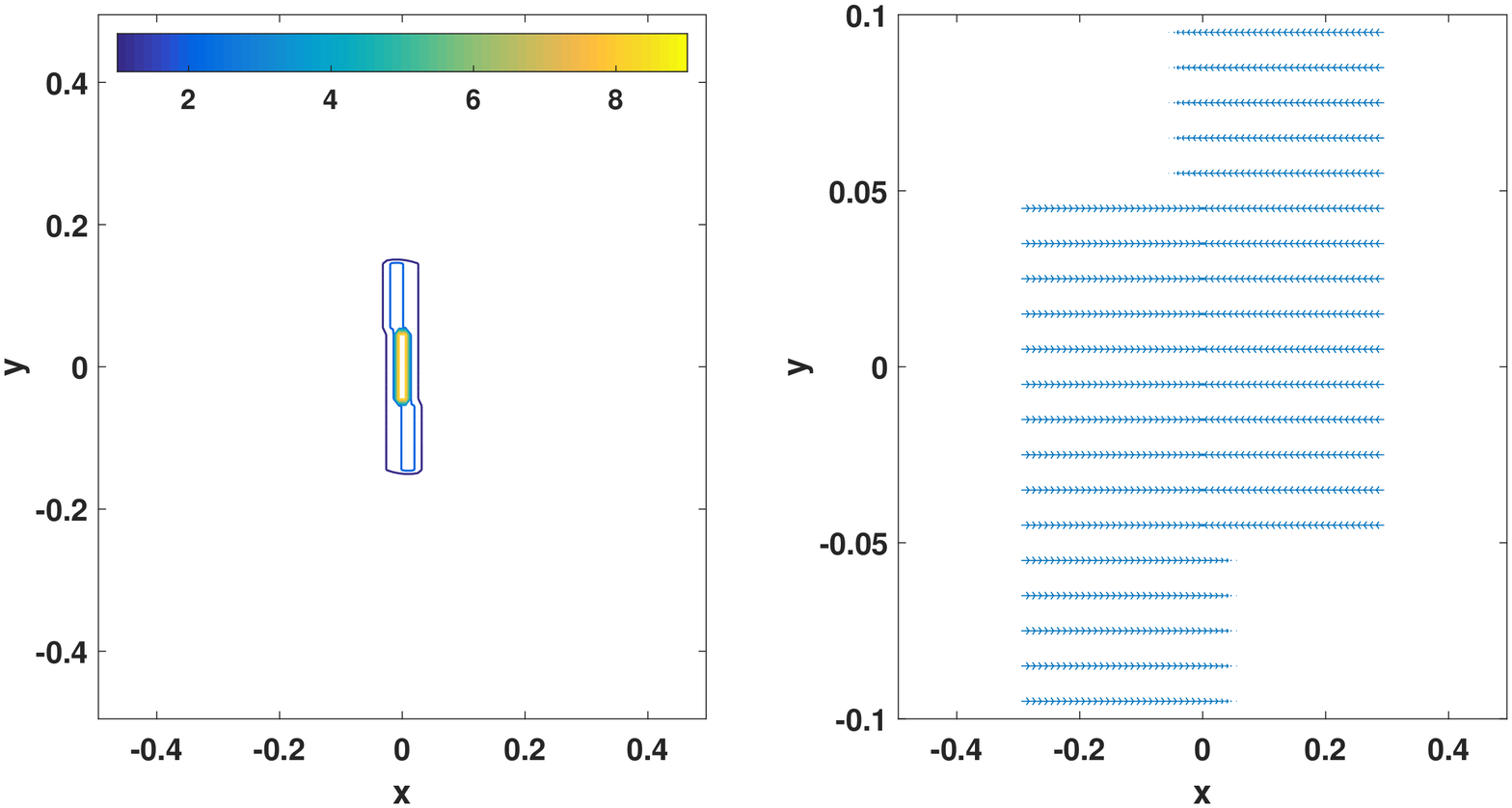}
    \includegraphics[width=\textwidth,keepaspectratio]{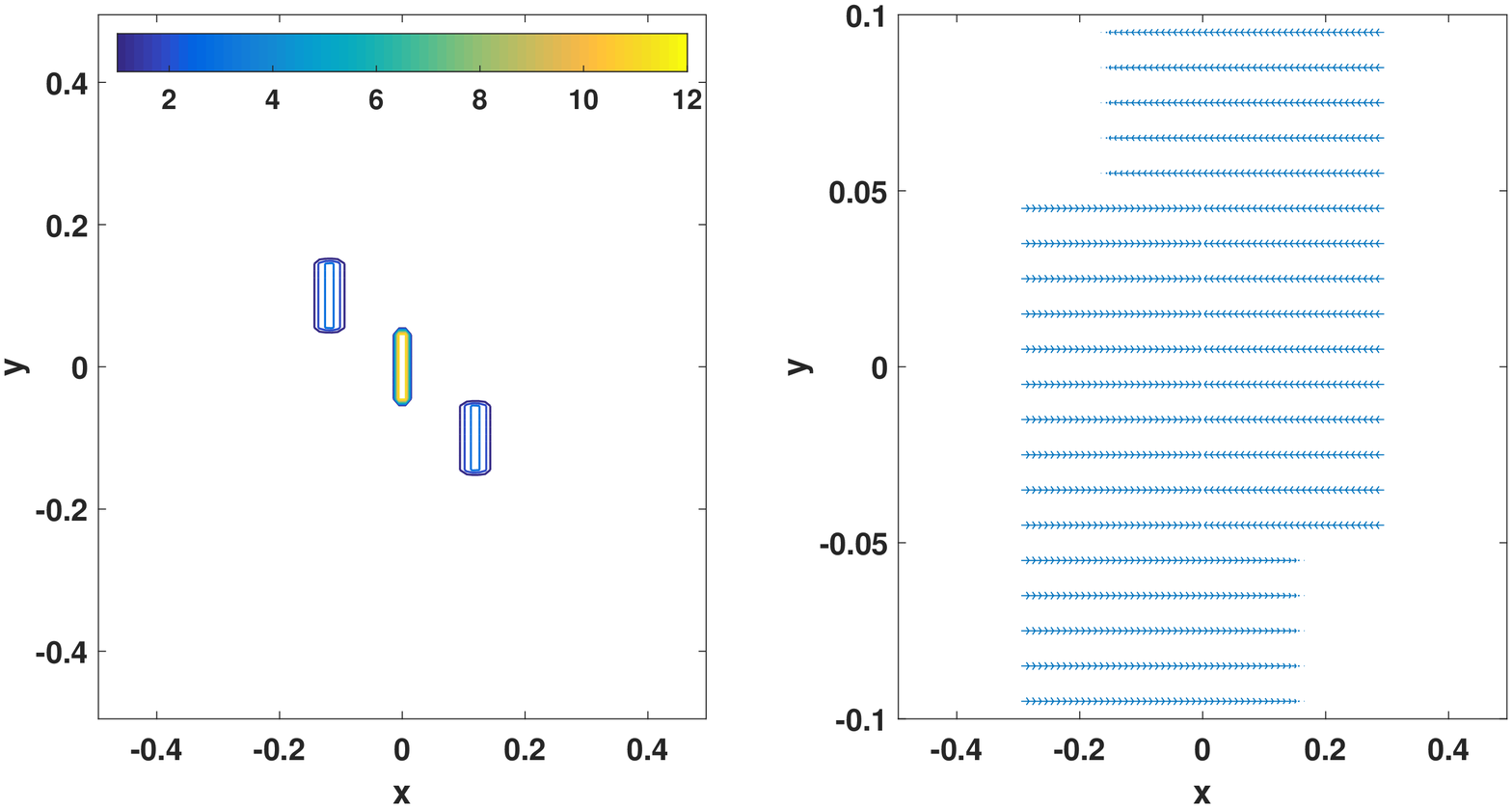}
    \caption{Density contour plots(left) and velocity vector plots(right) at times $T=0.2,0.52,0.8$}
    \label{cloud}
\end{figure}
\begin{figure}[H]
    \centering
    \includegraphics[width=\textwidth]{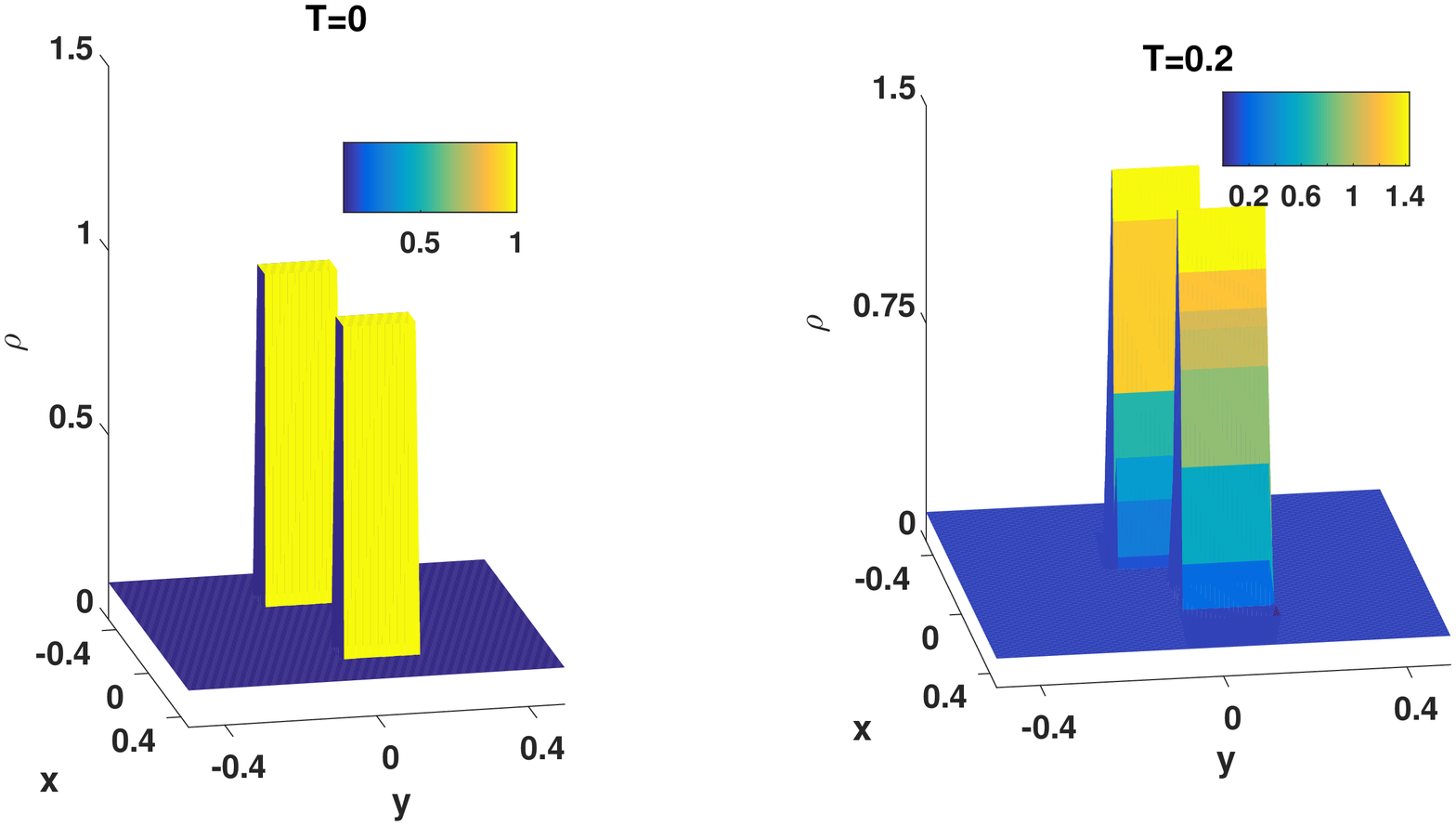}
     \includegraphics[width=\textwidth,keepaspectratio]{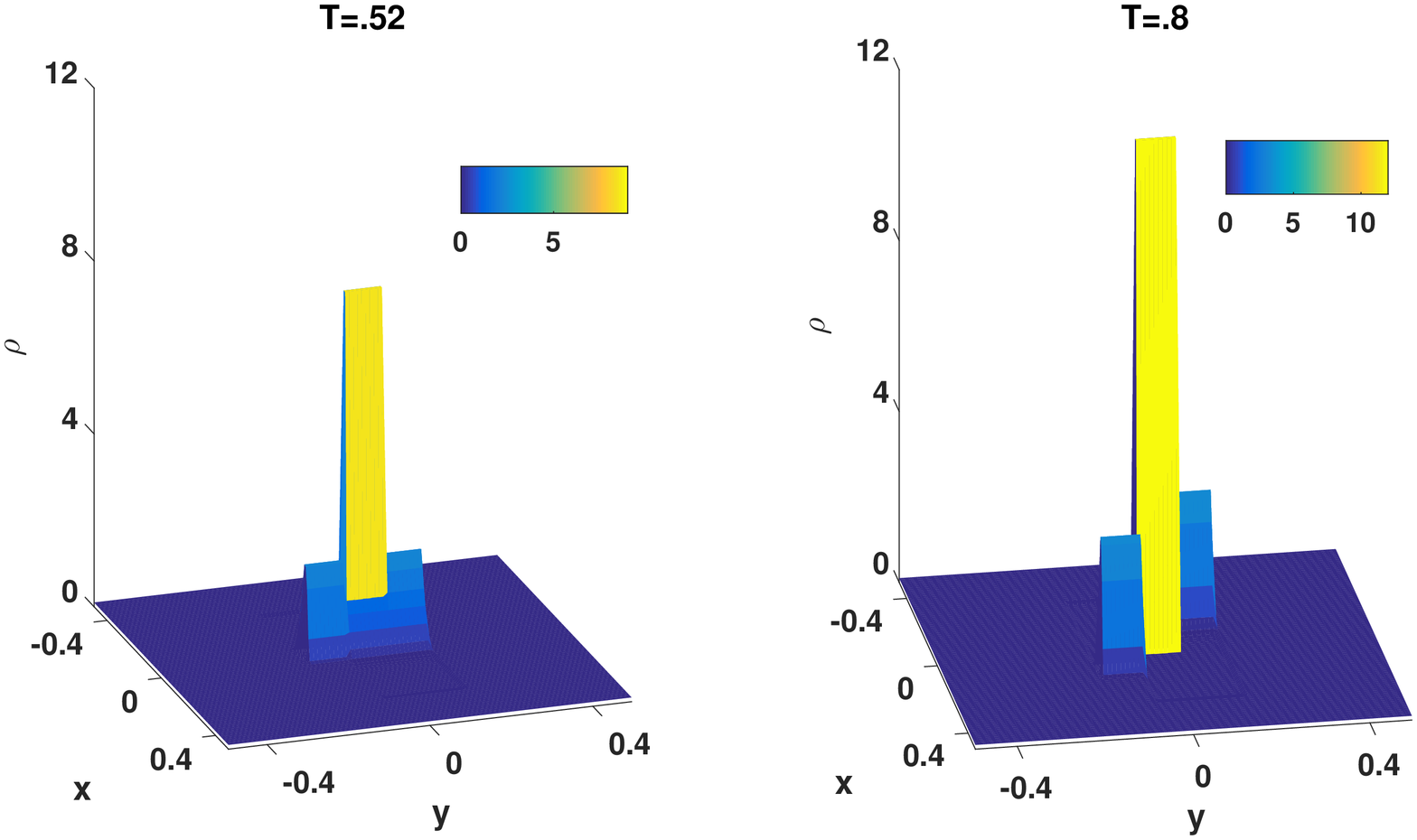}
    \caption{Density three dimensional plots $T=0,0.2,0.52,0.8$}
    \label{section1}
\end{figure}
Next, we plot the density and velocity profiles in \textbf{Figure~\ref{section}} at $T=0.52 $ and $T=0.8$ at the section $x=0$ to see the behavior around the  $\delta\,-$ shock. It can be seen that the density and velocity are stable and do not have any non-physical fluctuations.
\begin{figure}[H]
    \centering
    \includegraphics[width=\textwidth,keepaspectratio]{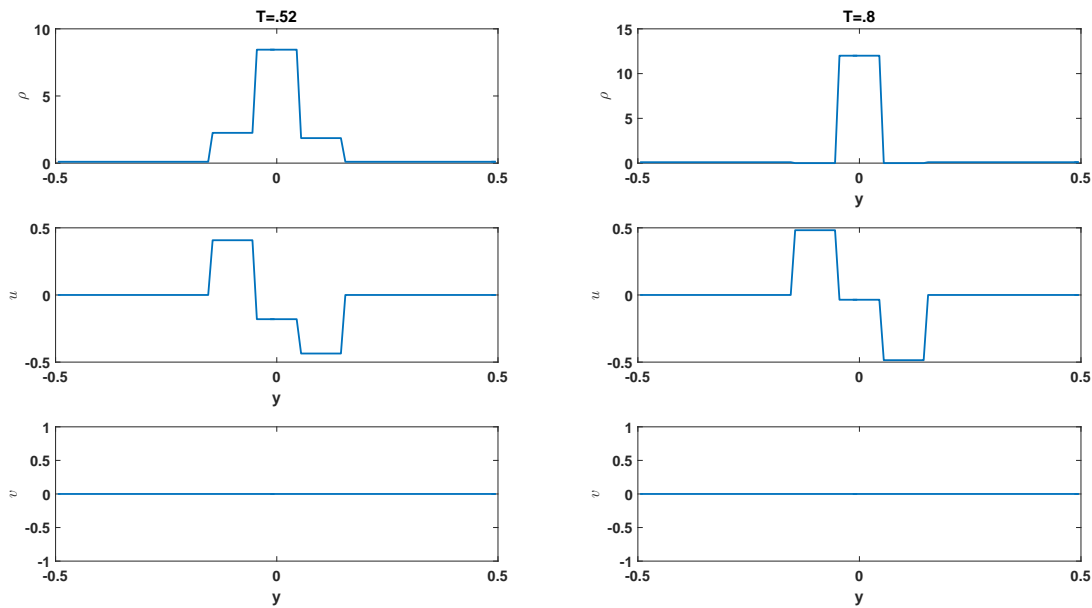}
    \caption{Numerical Solution $U$ at the section $x=0$}
    \label{section}
\end{figure}

\end{enumerate}
\bibliographystyle{elsarticle-num}
\bibliography{jcp} 

\begin{thebibliography}{10}
\expandafter\ifx\csname url\endcsname\relax
  \def\url#1{\texttt{#1}}\fi
\expandafter\ifx\csname urlprefix\endcsname\relax\def\urlprefix{URL }\fi
\expandafter\ifx\csname href\endcsname\relax
  \def\href#1#2{#2} \def\path#1{#1}\fi

\bibitem{bouchut1994zero}
F.~Bouchut, On zero pressure gas dynamics, Advances in kinetic theory and
  computing 22 (1994) 171--190.

\bibitem{korchinski1978solution}
D.~J. Korchinski, Solution of a {R}iemann problem for a $2\times2$ system of
  conservation laws possessing no classical weak solution., Ph.D Thesis (1978).

\bibitem{danilov2005delta}
V.~Danilov, V.~Shelkovich, Delta-shock wave type solution of hyperbolic systems
  of conservation laws, Quarterly of Applied Mathematics 63~(3) (2005)
  401--427.

\bibitem{panov2006delta}
E.~Y. Panov, V.~Shelkovich, $\delta^{'}$- {S}hock waves as a new type of
  solutions to systems of conservation laws, Journal of Differential Equations
  228~(1) (2006) 49--86.

\bibitem{shelkovich2006riemann}
V.~Shelkovich, The {R}iemann problem admitting $\delta$-, $\delta^{'}$-shocks,
  and vacuum states (the vanishing viscosity approach), Journal of Differential
  Equations 231~(2) (2006) 459--500.

\bibitem{tan1994delta}
D.~C. Tan, T.~Zhang, T.~Chang, Y.~Zheng, Delta-shock waves as limits of
  vanishing viscosity for hyperbolic systems of conservation laws, Journal of
  Differential Equations 112~(1) (1994) 1--32.

\bibitem{huang}
F.~Huang, Weak {S}olution to {P}ressureless {T}ype {S}ystem, Communications in
  Partial Difference Equations 30 (2005) 283--304.
\newblock \href {https://doi.org/10.1081/PDE-200050026}
  {\path{doi:10.1081/PDE-200050026}}.

\bibitem{mitrovic2007delta}
D.~Mitrovi{\'c}, M.~Nedeljkov, Delta shock waves as a limit of shock waves,
  Journal of Hyperbolic Differential Equations 4~(04) (2007) 629--653.

\bibitem{shen2015riemann}
C.~Shen, The riemann problem for the pressureless {E}uler system with the
  {C}oulomb-like friction term, IMA Journal of applied Mathematics 81~(1)
  (2015) 76--99.

\bibitem{chaplygin1944gas}
S.~Chaplygin, Gas jets, Sci. Mem. Moscow Univ. Math. Phys. 21 (1944) 1--121.

\bibitem{wang2013riemann}
G.~Wang, The {R}iemann problem for one dimensional generalized {C}haplygin
  {G}as {D}ynamics, Journal of Mathematical Analysis and Applications 403~(2)
  (2013) 434--450.

\bibitem{bouchut2003numerical}
F.~Bouchut, S.~Jin, X.~Li, Numerical approximations of pressureless and
  isothermal gas dynamics, SIAM Journal on Numerical Analysis 41~(1) (2003)
  135--158.

\bibitem{leveque2004dynamics}
R.~J. LeVeque, The dynamics of pressureless dust clouds and delta waves,
  Journal of Hyperbolic Differential Equations 1~(02) (2004) 315--327.

\bibitem{sheng1999riemann}
W.~Sheng, T.~Zhang, The Riemann problem for the transportation equations in gas
  dynamics, Vol. 654, American Mathematical Soc., 1999.

\bibitem{garg2020class}
N.~K. Garg, A class of upwind methods based on generalized eigenvectors for
  weakly hyperbolic systems, Numerical Algorithms 83~(3) (2020) 1091--1121.

\bibitem{berthon2006relaxation}
C.~Berthon, M.~Breu{\ss}, M.-O. Titeux, A relaxation scheme for the
  approximation of the pressureless {E}uler equations, Numerical Methods for
  Partial Differential Equations: An International Journal 22~(2) (2006)
  484--505.

\bibitem{yang2013discontinuous}
Y.~Yang, D.~Wei, C.-W. Shu, Discontinuous {G}alerkin method for {K}rause's
  consensus models and {P}ressureless {E}uler equations, Journal of
  Computational Physics 252 (2013) 109--127.

\bibitem{bryson2005semi}
S.~Bryson, A.~Kurganov, D.~Levy, G.~Petrova, Semi-discrete central-upwind
  schemes with reduced dissipation for hamilton--jacobi equations, IMA journal
  of numerical analysis 25~(1) (2005) 113--138.

\bibitem{nessyahu1990non}
H.~Nessyahu, E.~Tadmor, Non-oscillatory central differencing for hyperbolic
  conservation laws, Journal of computational physics 87~(2) (1990) 408--463.

\bibitem{adimurthi2000conservation}
Adimurthi, G.~D. Veerappa~Gowda, Conservation law with discontinuous flux, J.
  Math. Kyoto Univ. 43~(1) (2003) 27--70.

\bibitem{agg1}
A.~Aggarwal, M.~R. Sahoo, A.~Sen, G.~Vaidya, Solutions with concentration for
  conservation laws with discontinuous flux and its applications to numerical
  schemes for hyperbolic systems, Studies in Applied Mathematics 145~(2) (2020)
  247--290.

\bibitem{adimurthi2005optimal}
Adimurthi, S.~Mishra, G.~D. Veerappa~Gowda, Optimal entropy solutions for
  conservation laws with discontinuous flux-functions, Journal of Hyperbolic
  Differential Equations 2~(04) (2005) 783--837.

\bibitem{mishra2005}
S.~Mishra, Convergence of upwind finite difference schemes for a scalar
  conservation law with indefinite discontinuities in the flux function., SIAM
  J. Numer. Anal. 43~(2) (2005) 559--577.

\bibitem{mishra2existence}
Adimurthi, S.~Mishra, G.~D. Veerappa~Gowda, Existence and stability of entropy
  solutions for a conservation law with discontinuous non-convex fluxes,
  Networks \& Heterogeneous Media 2~(1) (2006) 127.

\bibitem{adimurthi2004godunov}
Adimurthi, J.~Jaffr{\'e}, G.~D. Veerappa~Gowda, Godunov-type methods for
  conservation laws with a flux function discontinuous in space, SIAM Journal
  on Numerical Analysis 42~(1) (2004) 179--208.

\bibitem{adimurthi2016godunov}
Adimurthi, A.~Aggarwal, G.~D. Veerappa~Gowda, Godunov-type numerical methods
  for a model of granular flow, Journal of Computational Physics 305 (2016)
  1083--1118.

\bibitem{aggarwal2016godunov}
Adimurthi, A.~Aggarwal, G.~D. Veerappa~Gowda, Godunov-{T}ype {N}umerical
  {M}ethods for a {M}odel of {G}ranular {F}low on {O}pen {T}ables with {W}alls,
  Communications in Computational Physics 20~(4) (2016) 1071--1105.

\bibitem{shu1989efficient}
C.-W. Shu, S.~Osher, Efficient implementation of essentially non-oscillatory
  shock-capturing schemes, ii, Journal of Computational Physics 83~(1) (1989)
  32--78.

\bibitem{jung2020relaxation}
S.~Jung, R.~Myong, A relaxation model for numerical approximations of the
  multidimensional pressureless gas dynamics system, Computers \& Mathematics
  with Applications 80~(5) (2020) 1073--1083.

\end{thebibliography}

\end{document}